\documentclass[12pt,reqno]{amsart}

\usepackage[T1]{fontenc}
\usepackage{enumitem}
\usepackage{hyperref}
\hypersetup{
  colorlinks   = true,
  citecolor    = blue,
  linkcolor    = blue
}

\usepackage{amsmath,amsfonts,amsthm,amssymb,color,tikz, comment,bbm}
\usepackage{mathrsfs}   
\usepackage{comment}

\usepackage{pdfsync}
\usepackage[font={scriptsize}]{caption}

\usepackage[left=1in, right=1in, top=1.1in,bottom=1.1in]{geometry}
\usepackage{yhmath} 
\usepackage{stmaryrd} 

\newcommand{\hb}[1]{\textcolor{blue}{#1}}

\newcommand{\hw}{\hat{w}}

\newcommand{\ii}{\imath}

\newcommand{\tcf}{\tilde{\mathcal{F}}}
\newcommand{\tu}{\tilde{U}}

\newcommand{\defeq}{\mathrel{\vcenter{\baselineskip0.5ex\lineskiplimit0pt\hbox{\scriptsize.}\hbox{\scriptsize.}}}=}

 \newcommand{\Heis}{{\mathbf{H}^{n}}}   
 \newcommand{\fHeis}{\tilde{\bf H}^n}
     
 \newcommand{\bW}{\mathbf{W}}
 \newcommand{\dW}{\dot{\mathbf{W}}}
 \newcommand{\bE}{\mathbf{E}}
 
 \newcommand{\kB}{\mathfrak{B}}
 \newcommand{\Ff}{\hat{f}}
 
 \newcommand{\cW}{\mathcal{W}}
 \newcommand{\Id}{\mathrm{Id}}

  \newcommand{\THeis}{\tilde{\bf H}^n}
  \newcommand{\HDelta}{\widehat{\Delta}}
  \newcommand{\HDL}{\widehat{\mathcal{D}}_{\lambda}}
    \newcommand{\UHDL}{\widehat{\underline{\mathcal{D}}}_{\lambda}}
  
  \newcommand{\nV}{\mathbf{V}^{\zeta,\alpha}}

\newcommand{\rf}{\tilde{f}}
\newcommand{\rg}{\tilde{g}}

\newcommand{\rphi}{\tilde{\phi}}
\newcommand{\supp}{\mathrm{Supp}}

\newcommand{\C}{\mathbb C}

\newcommand{\R}{\mathbb R}
\newcommand{\N}{\mathbb N}

\newcommand{\Z}{\mathbb Z}
  
\newcommand{\cb}{\mathcal B}
\newcommand{\cac}{\mathcal C}

\newcommand{\cd}{\mathcal D}

\newcommand{\cf}{\mathcal F}
\newcommand{\cg}{\mathcal G}
\newcommand{\ch}{\mathcal H}

\newcommand{\ck}{\mathcal K}

\newcommand{\cs}{\mathcal S}
\newcommand{\ct}{\mathcal T}

\newcommand{\al}{\alpha}

\newcommand{\ep}{\varepsilon}

\newcommand{\ga}{\gamma}

\newcommand{\ka}{\kappa}
\newcommand{\la}{\lambda}

\newcommand{\vp}{\varphi}

\newcommand{\lp}{\left(}
\newcommand{\rp}{\right)}
\newcommand{\lc}{\left[}
\newcommand{\rc}{\right]}

\newcommand{\lla}{\left\langle}
\newcommand{\rra}{\right\rangle}

\numberwithin{equation}{section}

\newtheorem{theorem}{Theorem}[section]

\newtheorem{claim}[theorem]{Claim}

\newtheorem{corollary}[theorem]{Corollary}

\newtheorem{definition}[theorem]{Definition}

\newtheorem{hypothesis}[theorem]{Hypothesis}
\newtheorem{lemma}[theorem]{Lemma}

\newtheorem{proposition}[theorem]{Proposition}

\theoremstyle{remark}
\newtheorem{remark}[theorem]{Remark}

\theoremstyle{remark}

\newcommand{\bean}{\begin{eqnarray*}}
\newcommand{\eean}{\end{eqnarray*}}
\newcommand{\ben}{\begin{enumerate}}
\newcommand{\een}{\end{enumerate}}
\newcommand{\beq}{\begin{equation}}
\newcommand{\eeq}{\end{equation}}

\begin{document}

\title[Besov spaces and PAM on Heisenberg group]
{Weighted Besov spaces on Heisenberg groups\\
and applications to the Parabolic Anderson model}


\author[{F. Baudoin \and L. Chen \and C-H. Huang \and C. Ouyang \and S. Tindel \and J. Wang}]
{Fabrice Baudoin \and Li Chen \and Che-Hung Huang  \and \\ Cheng Ouyang \and Samy Tindel \and Jing Wang}

\address{Fabrice Baudoin: 
Department of Mathematics,
Aarhus University,
Denmark.}
\email{fbaudoin@math.au.dk}
\thanks{F. Baudoin was partly funded by NSF DMS-2247117 when some of the work was performed}

\address{Li Chen: 
Department of Mathematics,
Aarhus University,
Denmark.}
\email{lchen@math.au.dk}
\thanks{L. Chen is partly funded by the Simons Foundation \#853249}

\address{Che-Hung Huang: Department of Mathematics, Purdue University,  United States.}
\email{huan1160@purdue.edu}

\address{Cheng Ouyang: Department of Mathematics, 
Statistics and Computer Science, University of Illinois at Chicago, 
United States.}
\email{couyang@uic.edu}
\thanks{C. Ouyang is partly funded by the Simons Foundation \#851792}

\address{Samy Tindel: Department of Mathematics, Purdue University,  United States.}
\email{stindel@purdue.edu}
\thanks{S. Tindel is partly funded by  NSF DMS-2153915}

\address{Jing Wang: Department of Mathematics, Department of Statistics, Purdue University,  United States.}
\email{jingwang@purdue.edu}
\thanks{J. Wang is partly funded by NSF DMS-2246817}

\begin{abstract}This article aims at a proper definition and resolution of the parabolic Anderson model on  Heisenberg groups $\Heis$. This stochastic PDE is understood in a pathwise (Stratonovich) sense. We consider a noise which is smoother than white noise in time, with a spatial covariance function generated by negative powers $(-\Delta)^{-\alpha}$ of the sub-Laplacian on $\Heis$. We give optimal conditions on the covariance function so that the stochastic PDE is solvable. A large portion of the article is dedicated to a detailed definition of weighted Besov spaces on $\Heis$. This definition, related paraproducts and heat flow smoothing properties, forms a necessary step in the resolution of our main equation. It also appears to be new and of independent interest. It relies on a recent approach, called projective, to Fourier transforms on $\Heis$.
\end{abstract}

\maketitle

\tableofcontents

\section{Introduction}

In \cite{BLOTW,BOTW} we have started a long term project aiming at studying parabolic Anderson models and polymer measures on geometric structures. The broader question for us is to determine if the geometric environment is likely to change the exponents of the model in a significant way. In this context Heisenberg groups (denoted by $\Heis$ in the sequel) are good test examples, due to their sub-Riemannian nature. The study in \cite{BOTW} was centered on stochastic PDEs in the It\^o sense on $\Heis$. We did find that the critical exponents obtained in that case were substantially different from the ones obtained in a flat Euclidean space. See also the wide variety of exponents obtained in fractal settings (\cite{BLOTW}). 

In the current contribution we go back to the Heisenberg group setting, but we now consider pathwise (Stratonovich) equations. This is crucial in order to define polymer measures in a straightforward way. Also notice that the Stratonovich setting relies on a set of tools which is completely distinct from the It\^o setting. Before we turn to a summary of our findings, let us recall some basic geometric notation about $\Heis$.

\subsection{General facts about Heisenberg groups}
For a given $n\ge 1$,
the Heisenberg group $\Heis$ can be identified with $\R^{2n}\times \R$, equipped with the group multiplication defined for $(x,y,z), (x',y',z)\in \R^{2n+1}$: 
\begin{align}\label{eq-groupaction}
(x,y,z)*(x',y',z'):=\left(x+x', y+y', z+z'+2\omega((x,y),(x',y')) \right),
\end{align}
where $\omega:\R^{2n}\times\R^{2n}\to\R$, $\omega((x,y),(x',y'))=\sum_{i=1}^n x_i'y_i-x_iy_i'$ is the standard symplectic form on $\R^{2n}$.
The identity in $\Heis$ is $e=(0,0,0)$ and the inverse is given by $(x,y,z)^{-1}=(-x,-y,-z)$. In the sequel, for $q=(x,y,z)$ and $q'=(x',y',z')$, we will simply write $qq'$ for the product $q* q'$.
Its Lie algebra $\mathfrak{h}$ can be identified with $\R^{2n+1}$ with the Lie bracket given by
\begin{equation*}
[(a,b, c), (a', b', c')]=\left(0, 0, 2 \omega((a,b), (a',b'))\right), \quad (a,b), (a',b')\in \R^{2n}.
\end{equation*}
Note that $\mathfrak{h}$ can also be considered as the tangent space $T_e(\Heis)$. One can then define a basis of left invariant vector fields on $\Heis$ by 
\begin{equation}\label{eq-basis}
X_i(p)=\partial_{x_i}+2y_i \, \partial_{z}, \quad Y_{i}(p)=\partial_{y_i}-2x_i \, \partial_{z}, \quad Z(p)=\partial_{z}, \quad i=1,\dots, n \, ,
\end{equation}
for any $p=(x,y,z)\in \Heis$. The homogeneous norm on $\Heis$ is given by
\begin{equation}\label{eq:def-homogeneous-norm}
\lvert p \rvert_{\mathrm{h}}=(|x|^2+|y|^2+|z|)^{1/2} \, .
\end{equation}
We will also use the family of dilations $\{\delta_\lambda;\lambda>0\}$ defined as follows:
\begin{equation}\label{eq-dilation}
    \delta_\lambda q=(\lambda x, \lambda y, \lambda^2 z) \, ,
    \quad\text{for}\quad q=(x,y,z)\in\Heis.
\end{equation}
The Heisenberg group $\Heis$ appears as a flat model of sub-Riemannian manifold. Its sub-Riemannian structure can be reflected by restricting diffusions on its horizontal distribution $\mathcal{D}:=\{\mathcal{D}_p\}_{p\in \Heis}$ where 
\[
\mathcal{D}_p:=\mathrm{Span}\{X_i(p), Y_i(p), i=1,\dots, n\}.
\] 
A standard choice of sub-Riemannian metric is the inner product such that $\mathcal{H}\equiv\{X_i(p), Y_i(p),\linebreak i=1,\dots, n\}$ is an orthonormal frame at $p$. This orthonormal frame is referenced to as horizontal. Related to this notion is the  Carnot--Carath\'eodory distance on $\Heis$, which will be abundantly used in the remainder of this paper. For $q_1,q_2\in\Heis$, it is defined as \begin{equation}\label{cc-metric}
d_{cc}(q_1,q_2)= \inf
\left\{\int_0^1|\dot{\gamma}|_\mathcal{H}dt;\ \gamma:[0,1]\to\Heis\   \text{horizontal}, \gamma(0)=q_1, \gamma(1)=q_2\right\}.    
\end{equation}
 The  Carnot--Carath\'eodory distance on $\Heis$ is left-invariant i.e. 
 \begin{align}\label{left-invariant property}
 d_{cc}(pq_1,pq_2)=d_{cc}(q_1,q_2)
 \end{align}
 and satisfies the following scaling property
 \[
 d_{cc}(\delta_\lambda q_1,\delta_\lambda q_2)= \lambda d_{cc}( q_1,q_2).
 \]
It is then  an easy consequence of the scaling properties of $d_{cc}$ and $|\cdot|_h$ that the induced norm of $d_{cc}$ is equivalent to $|\cdot|_h$. Namely there exists a $c>1$ such that for all $q\in\Heis$,
\begin{equation}\label{eq-equiv-cc-h}
\frac1c |q|_h\le d_{cc}(e,q) \le c|q|_h.
\end{equation}
The  diffusion operator  related to the sub-Riemannian metric, that is
\begin{equation}\label{eq-Laplacian}
\Delta=\sum_{i=1}^nX_i^2+Y_i^2 \, ,
\end{equation}
is called the sub-Laplacian. It is essentially self-adjoint on the space of smooth and compactly supported functions $\mathcal{C}_0^\infty(\mathbb{R}^{2n+1})$, and has a unique  self-adjoint extension to the $L^2$ space with respect to the Haar measure $\mu$. The sub-Laplacian also admits a heat kernel $p_t$, issued from $e$, which was explicitly computed in \cite{Gaveau}: for any $q=(x,y,z)\in \Heis$ we have
\begin{equation}\label{eq-Heis-kernel}
p_t(q)=\frac{1}{(2\pi t)^{n+1}}\int_\R e^{i\frac{\lambda z}{t}}\left(\frac{2\lambda}{\sinh{2\lambda}}\right)^{n}\exp\left(-\frac{\lambda}{t} |(x,y)|^2\coth (2\lambda) \right)d\lambda,
\end{equation}
where the norm $|(x,y)|$ is defined by
\begin{equation}\label{eq:radial}
|(x,y)|^2= x_1^2+\cdots +x_n^2+y_1^2+\cdots +y_n^2 \, .
\end{equation}
Related to the heat kernel $p_{t}$, consider the heat semi-group
 $\{P_t ; \, t\ge0\}$  generated by the self-adjoint operator $\Delta$. Namely for any $f\in L^2(\R^{2n+1}, \mu)$ we have
\begin{equation}\label{eq-semigroup}
P_tf(p)=\int_{\R^{2n+1}} f(q)p_t(q^{-1}p)d\mu(q) \, ,
\end{equation}
where $\mu$ designates the Haar measure on $\Heis$.
The insight on $P_{t}$ provided by Fourier analysis is recalled and analyzed later in our paper.
For our SPDEs computations in Section~\ref{sec:spdes-existence}, we will also resort to resolvent kernels corresponding to the operator $(-\Delta)^{-\alpha}$ where $0<\alpha<n+1$. Those kernels are given by 
\begin{align}\label{green function}
G_\alpha(p,q)=\frac{1}{\Gamma(\alpha)} \int_0^{+\infty} t^{\alpha-1} p_t(q^{-1}p) \, dt, 
\quad\text{for all}\quad
p,q \in \Heis,
\end{align}
and were extensively studied in \cite{BOTW} to which we refer for further details.

\subsection{Main results}

The scope of our paper is twofold. First we introduce a notion of weighted Besov space on $\Heis$, which is compatible with heat flows. Then we apply this new notion to a parabolic Anderson model on $\Heis$. Let us detail those two points separately.

The notion of weighted Besov space in $\R^d$ is clearly rooted in a proper partition of unity in Fourier modes (see for example \cite{BCD2, MW}). In $\Heis$ this notion is far from being trivial. To start with, Fourier transforms in $\Heis$ are traditionally seen as operator-valued. In order to get a more handy version of the transform, in this paper we resort to the recent approach introduced in \cite{BCD2}. Roughly speaking, this method is based on a series of projections onto rescaled Hermite functions. The Fourier variables are then elements of the finite dimensional set $\fHeis=\N^n\times \N^n\times\R$ (although more complex structures are considered in \cite{BCD2} when taking closures). Once this new setting is accepted it remains to construct a convenient partition of unity, satisfying proper decay and support conditions. This step is based on delicate differential bounds detailed in Sections \ref{sec-FT} and \ref{sec-basic}. Overall our notion of weighted Besov space, although natural once the appropriate setting has been introduced, is an important part of our endeavor. We believe it is new and of independent interest.

Once the Besov space setting is established, we turn our attention to a parabolic Anderson model in $\Heis$. Formally it can be written as the following linear equation over $\R_+\times \Heis$:
\begin{equation}\label{eq:pam}
\partial_{t} u_t(q) = \frac{1}{2}\Delta u_t(q) + u_t(q) \, \dot{\bW}^{\zeta,\alpha}_t(q).
\end{equation}
In equation \eqref{eq:pam}, the time integral has to be understood in the young (or Stratonovich) sense, while the product $q\mapsto u_t(q) \dot{\bW}^{\zeta,\alpha}_t(q)$ is interpreted as a distributional product. Our aim is to handle the case of a centered Gaussian noise $\dot{\bW}^{\zeta,\alpha}$ with a covariance function informally given by 
\begin{align}\label{eq-intro-cov}
\bE\left[ \dot{\bW}_t^{\zeta, \alpha}(q_1)\dot{\bW}_s^{\zeta, \alpha}(q_2)\right]
=
|t-s|^{-\zeta} \, G_{2\alpha}(q_1,q_2),
\end{align}
where $G_{2\alpha}$ is the kernel introduced in \eqref{green function}. Otherwise stated $\dot{\bW}^{\zeta, \alpha}$ behaves in time like the derivative of a fractional Brownian motion type processes, while in space one can write $\dot{\bW}^{\zeta, \alpha}$ as $(-\Delta)^{-\alpha}\dot{W}$ (where $\dot{W}$ stands for a white noise on $\Heis$). A more formal definition is provided in Section \ref{sec-app-Gau}. However, it should be clear from \eqref{eq-intro-cov} that $\dot{\bW}^{\zeta, \alpha}$ is a distribution in both time and space. For this type of noise we prove that equation \eqref{eq:pam} admits a unique solution under the following condition on the coefficients $\zeta$ and $\alpha$:
\begin{equation}\label{eq-intro-cond}
\zeta\in (0,1),\quad \mbox{and}\quad \frac{n+1}2-(1-\zeta)<\alpha<\frac{n+1}2.
\end{equation}
In order to achieve this result, we will solve a mild form of the stochastic heat equation~\eqref{eq:pam}. We resort to Young integration techniques in time and rely heavily on our weighted Besov space framework for the spatial component. Notice that weights are essential in this context, since a noise like $\dot{\bW}^{\zeta, \alpha}$ is spatially unbounded (regardless of the Besov scale $\alpha$). In fact we will solve \eqref{eq:pam} in a space of weighted processes which combines time and space regularity in a subtle way (see Definition \ref{def-D-space} for a full description).

As mentioned above, one of our main goals is to quantify the influence of underlying geometric structures on relevant physical random systems (such as the parabolic Anderson model or directed polymers). Let us add a few comments about the exponents obtained in the current work:
 
\begin{enumerate}[wide, labelwidth=!, labelindent=0pt, label= \textbf{(\roman*)}]
\setlength\itemsep{.05in}

\item In \cite{BOTW} we have established existence and uniqueness results for an equation of the form \eqref{eq:pam}, albeit driven by a white noise in time. In~\cite{BOTW} we also dealt with a stochastic integral interpreted in the It\^o sense, as opposed to the Stratonovich-Young setting considered here. In this context we had established existence and uniqueness of the solution, under the condition 
 \begin{equation}\label{eq-intro-condition}
\frac{n}2<\alpha<\frac{n+1}2.
\end{equation}
If we wish to compare \eqref{eq-intro-condition} with our condition \eqref{eq-intro-cond}, one should recall that (formally) the white noise case in time corresponds to $\zeta=1$ in \eqref{eq-intro-cond}. Hence in the white noise case hypothesis~\eqref{eq-intro-cond} becomes an incompatible condition 
  \begin{equation}\label{eq-intro-cond-incom}
\frac{n+1}2<\alpha<\frac{n+1}2.
\end{equation}
This phenomenon is not surprising. Indeed, we have also seen in \cite[Proposition 3.12]{BOTW} that for  $\alpha>\frac{n+1}2$, one can slightly modify the definition of $\dot{\bW}^{\zeta, \alpha}$ in order to get a function (versus a distribution for $\alpha<\frac{n+1}2$) in space. Hence condition \eqref{eq-intro-cond-incom} just tells us the following: if we wish to consider the stochastic heat equation \eqref{eq:pam} driven by a white noise in time, in the Stratonovich sense, then the noise should be a function in space. This is well known in other contexts, see for instance \cite{TZ}.

\item One can refer e.g to \cite{HHNT} for a comparison with equations on $\R^d$ in the Stratonovich sense. We have argued in \cite{BOTW} that the noise in $\R^d$ which resembles most our noise $\dot{\bW}^{\zeta, \alpha}$ is based on Bessel kernels. This is a noise $\dot{\bW}^{\R^d,\zeta, \alpha}$ whose covariance function is formally similar to \eqref{eq-intro-cov-R}:
\begin{align}\label{eq-intro-cov-R}
\bE\left[ \dot{\bW}_t^{\R^d,\zeta, \alpha}(x) \, \dot{\bW}_s^{\R^d,\zeta, \alpha}(y)\right]
=
|t-s|^{-\zeta} \, G^{\R^d}_{2\alpha}(x,y),
\end{align}
where $G^{\R^d}_{2\alpha}(x,y)$ is the resolvent kernel for $(\mathrm{Id}-\Delta_{\R^d})^{-2\alpha}$. Now let us translate the conditions in \cite[Proposition 5.22 and Hypothesis 4.8]{HHNT} to the case of a Bessel kernel, for which the Fourier transform of the spatial covariance is $(1+|\zeta|^2)^{-2\alpha}$. We let the reader check that we get
\[
\alpha>\frac{d}4-(1-\zeta).
\]
This is exactly our condition \eqref{eq-intro-cond} in the Heisenberg case, once we realize that the relevant dimension in $\Heis$ is the Hausdorff dimension $Q=2n+2$ (as opposed to the topological dimension $2n+1$). As in \cite{BOTW}, this is a clear manifestation of the sub-Riemannian nature of $\Heis$.

\item
The current article is restricted to equations defined in the Young sense. It would be natural to observe how our exponents are affected by more singular noises, recurring either to regularity structures or paracontrolled techniques. Notice that this type of development has been initiated by~\cite{MS}. However the reference~\cite{MS} focuses on a compact version of $\Heis$, which conveniently avoids the use of weighted spaces. We plan to delve deeper into this direction in future communications.
\end{enumerate}

Our paper is organized as follows: in Section \ref{sec-FT} we recall the basic Fourier analysis setting which will prevail throughout the paper. Then we construct a Gevrey type class of functions in $\Heis$, which is crucial step in the definition of weighted Besov space. This step culminates in a Bernstein type lemma in Section \ref{sec-Berns}. Section \ref{sec-basic} is devoted to define Besov spaces and derive related smoothing effects of the heat semigroup, as well as a nontrivial notion of paraproduct. The application of this theory to stochastic PDEs is contained in Section \ref{sec:spdes-existence}.

\paragraph{\textbf{Notation.}}
Throughout this paper, the set $\N$ stands for the family of integers $\{0,1,2,\ldots\}$.
For any multi-index integer $k=(k_1,\dots, k_n)\in \N^n$, we use the notation  $k!:=k_1!\cdots k_n!$ and $|k|:=k_1+\cdots +k_n$.  We extend the notation $|k|$ to indices $k\in\Z^n$, by letting $|k|=|k_1|+\cdots+|k_n|$.

\section{Fourier transform on Heisenberg group}\label{sec-FT}

In this section we summarize some basic facts about Fourier transforms on $\Heis$ as can be found in \cite{BCD}. We will then study some properties of functions whose Fourier transform is compactly supported. Those properties will be crucial for a proper notion of Besov weighted space. Eventually we will define a notion of fractional Laplacian which is invoked for the description of our noisy inputs, and recall the notion of tempered distribution in $\Heis$.
 
\subsection{Fourier transform}\label{sec:fourier-general}
It is well-known that all irreducible representations of the Heisenberg group $\Heis$ are unitary equivalent to the Schr\"odinger representations $(U^\lambda)_{\lambda \in \mathbb{R}}$, that is the family of group morphisms $q=(x,y,z) \in \Heis \to U_q^\lambda$ between $\Heis$ and the unitary group of $L^2(\R^{n})$ defined by
\begin{equation}\label{eq:representation1}
U_q^\lambda u (\xi) = e^{-i\lambda\left( z+2\langle y,\,\xi- x\rangle \right) } u \left( \xi -2x\right), \quad \xi \in \mathbb R^n.
\end{equation}
Based on those representations, the group Fourier transform of a function $f \in L^1(\Heis)$ is defined for each $\lambda \in \R\setminus\{0\}=:\R^*$ as the operator valued function on $L^2(\mathbb R)$ given by
\begin{equation}\label{eq-Fourier-transf}
\mathcal{F}(f) (\lambda)=\int_{\Heis} f(q)U_{q}^\lambda \,  d\mu(q) .
\end{equation}
Clearly $\mathcal{F}(f) $ takes values in the space of bounded operators on $L^2(\R^n)$. One of the important features of $\cf$ is its compatibility with the Laplace operator. Namely for a generic smooth function $f$ we have
\begin{equation}\label{eq-Ft-Del}
\mathcal{F}(\Delta f)(\lambda)=4 \, \mathcal{F}(f)(\lambda)\circ \Delta^\lambda_{\mathrm{osc}} \, ,
\end{equation}
where the operator $\Delta^\lambda_\mathrm{osc}$ is defined by $\Delta^\lambda_\mathrm{osc}:= \sum_{j=1}^n\partial_j^2-\lambda^2|x|^2$ and is called harmonic oscillator.

In this paper we will consider the Fourier transform from  a projective point of view that was introduced by \cite{BCD}. In order to introduce this notion, let $(\Phi_k)_{k\in \N^n}$ be the family of Hermite functions on $\R^n$ that is given by
\begin{equation}\label{a1}
\Phi_0(x)=\pi^{-\frac{n}{4}}e^{-\frac{|x|^2}{2}},
\quad\text{and}\quad \Phi_k(x)=\left(\frac1{2^{|k|}k!}\right)^{\frac12}C^k\Phi_0(x) \, ,
\quad \text{for } k\in \N^n \, .
\end{equation}
In \eqref{a1}, we have set $C^k=\prod_{j=1}^{|k|}C_j^{k_j}$,  where $C_j:=-\partial_j+M_j$ denotes the creation operator and $M_j$ is the multiplication operator such that for any function $u$ on $\R^n$, $M_ju(x)=x_ju(x)$. The family $(\Phi_k)_{k\in \N^n}$ forms an orthonormal basis of $L^2(\R^n)$  and the functions $\Phi_k$ are the eigenfunctions of the harmonic oscillator $\Delta_\mathrm{osc}:=\Delta^1_\mathrm{osc}=\sum_{j=1}^n\partial_j^2-|x|^2$ appearing in~\eqref{eq-Ft-Del}. Specifically, for $k\in \N^n$ we have
\begin{equation}\label{a11}
\Delta_\mathrm{osc}\Phi_k=-(2|k|+n)\Phi_k.
\end{equation}

We now introduce a scaling factor which will account for the Fourier transform with respect to the $z$-variable in $\Heis$. That is for $\lambda \in \mathbb R^*$ and $k \in \mathbb N^n$, we consider the family of  rescaled Hermite functions
\begin{equation}\label{a2}
\Phi_k^\lambda (x)= |\lambda|^{n/4} \Phi_k \lp \sqrt{|\lambda|} x\rp, \quad x \in \mathbb R^n \, .
\end{equation}
This family again forms an orthonormal basis of $L^2(\mathbb R^n)$, and satisfies that
\[
\Delta^\lambda_\mathrm{osc}\Phi^\lambda_k (x)= \left(\sum_{j=1}^n\partial_j^2-\lambda^2|x|^2\right)\Phi^\lambda_k (x) =-(2|k|+n)|\lambda|\Phi^\lambda_k (x).
\]
We denote by $\|\cdot\|_{\textsc{hs}}$ the Hilbert-Schmidt norm. Note that owing to the fact that $(\Phi^\lambda_k)_{k\in\N^{n}}$ is an orthonormal basis of $L^2(\mathbb R^n)$, we have
\begin{equation}\label{eq:projective-hs-norm}
\| \mathcal{F}(f) (\lambda) \|_{\textsc{hs}}^2=\sum_{k\in \N^n}   \| \mathcal{F}(f) (\lambda)  \Phi_k^\lambda \|_{L^2(\R^n)}^2.
\end{equation}

Identity \eqref{eq:projective-hs-norm} leads us to introduce the following projective definition of the Fourier transform that is convenient for defining the tempered distribution on $\Heis$. 

\begin{definition}\label{def:projective-fourier}
Let $\fHeis:=\mathbb N^n \times  \mathbb N^n  \times \mathbb R^*$, whose generic element is denoted by $\hw:=(m,\ell,\lambda)$. For any $f\in L^2(\Heis)$ we define the (projective) Fourier transform of $f$ as the function $\hat f:\fHeis\to \C$ given by
\begin{equation}\label{eq-F-trans-decomp}
\hat f (m,\ell,\lambda):= \langle \mathcal{F}(f) (\lambda) \Phi_m^\lambda , \Phi_\ell^\lambda \rangle_{L^2(\R^n)} \, ,
\end{equation}
where we recall that the functions $\Phi_{k}^{\la}$ are given by \eqref{a2}.
\end{definition} 

\noindent
Notice that in \eqref{eq-F-trans-decomp}, the Fourier transform is now complex valued as in $\R^{n}$. It can also be written as the following integral representation that is closely related to Wigner transforms of Hermite functions: for any $(m,\ell,\lambda)\in \fHeis$,
\begin{equation}\label{eq-FT-int}
 \hat f (m,\ell,\lambda)=\int_\Heis  \overline{e^{i\lambda z} K_{m,\ell,\lambda} (q)} f(q) \, d\mu(q),
\end{equation}
where for all $q= (x,y,z)\in \Heis$, the kernel $K_{m,\ell,\lambda} (q)$ is given by
\begin{equation}\label{eq-FT-K}
K_{m,\ell,\lambda} (q)=\int_{\R^n}e^{2i\lambda \langle y,\xi\rangle}\Phi^\lambda_m(x+\xi)\Phi^\lambda_\ell(-x+\xi)d\xi.
\end{equation}

We now label some important properties of the projective Fourier transform which will be used in the sequel:
\begin{enumerate}[wide, labelindent=0pt, label= \textbf{(\roman*)}]
\setlength\itemsep{.05in}

\item
The Fourier inversion formula for $q=(x,y,z)$ is  obtained by integrating against the kernel $K_{m,\ell,\lambda} (q)$. Namely recalling that we write $\hw=(m,\ell,\lambda)$, it can be shown that
\begin{equation}\label{eq-FT-int-inv}
f(q) 
=
\frac{2^{n-1}}{\pi^{n+1}} 
\sum_{m,\ell\in \N^n}\int_{\mathbb{R}} e^{i\lambda z} K_{m,\ell,\lambda} (q) \hat f (m,\ell,\lambda)|\lambda|^nd\lambda
=:
\frac{2^{n-1}}{\pi^{n+1}} \int_{\fHeis}e^{i\lambda z} K_{\hat{w}} (q) \hat{f} (\hat{w})d \hat{w},
\end{equation}
where the second identity is a notation which will be used in the sequel without further mention.

In fact, it is shown in \cite{BCD} that the Fourier transform as defined above can be extended as a bi-continuous isomorphism between $L^2(\Heis)$ and $L^2(\fHeis)$, and the following Plancherel formula holds,
\begin{align}\label{eq-planch}
\int_{\Heis} |f(q)|^2 d\mu(q)=
\frac{2^{n-1}}{\pi^{n+1}} \sum_{m,\ell\in \N^n} \int_{\R} \, \hat f (m,\ell,\lambda)^2 |\lambda |^n d\lambda .
\end{align}

\item
The interaction between Laplacian and Fourier transform is easily read in projective mode.
Specifically combining \eqref{eq-Ft-Del} with \eqref{a11} and\eqref{eq-F-trans-decomp},  we can easily find that for any $(m,\ell, \lambda)\in \fHeis$,
\begin{equation}\label{eq-FT-del1}
\widehat{\Delta f}(m,\ell, \lambda)=-4|\lambda|(2|m|+n)\Ff (m,\ell, \lambda).
\end{equation}
Similarly for any $N\ge1$,  we have
\begin{equation}\label{eq-FT-delN}
\widehat{\Delta^N f}(m,\ell, \lambda)=\big(-4|\lambda|(2|m|+n)\big)^N\Ff (m,\ell, \lambda).
\end{equation}
Relation \eqref{eq-FT-del1} also enables to introduce fractional power of the Laplacian. Namely for $\alpha\in\R$ we define $(-\Delta)^\alpha f$ through its Fourier transform:
\begin{equation}\label{eq-FT-delalpha}
\widehat{(-\Delta)^\alpha f}(m,\ell, \lambda)=\big(4|\lambda|(2|m|+n)\big)^\alpha\Ff (m,\ell, \lambda).
\end{equation}
\item
The Fourier transform of the Dirac distribution at the identity $e$ is explicitly computed as:
\[
\hat {\delta_e} (m,\ell,\lambda)=\langle \mathcal{F}(\delta_e) (\lambda) \Phi_m^\lambda , \Phi_\ell^\lambda \rangle_{L^2(\R^n)}= \langle U_e^\lambda (\Phi_m^\lambda ), \Phi_\ell^\lambda \rangle_{L^2(\R^n)}=\mathbbm{1}_{\{m=\ell\}},
\]for all $(m,\ell,\lambda)\in\fHeis$. In the same way, let $p_t$ be the sublaplacian heat kernel defined by \eqref{eq-Heis-kernel}. Then the projective Fourier transform of $p_t$ satisfies
\begin{equation}\label{eq-pt-hat}
    \hat{p}_t(m,\ell,\lambda) = e^{-4t|\lambda|(2|m|+n)}\mathbbm{1}_{\{m=\ell\}}.
\end{equation}

\item
For any integrable functions $f, g$ on $\Heis$,  the following convolution identity holds:
\begin{equation}\label{eq-conv}
\widehat{f\star g}(m,\ell,\lambda)=(\hat{f}\cdot \hat{g})(m, \ell, \lambda)
\end{equation}
where 
\begin{equation}\label{def-dot}
(\hat{f}\cdot \hat{g})(m, \ell, \lambda):=\sum_{j\in \N^n}\hat{f}(m, j, \lambda)  \hat{g}(j, \ell, \lambda).
\end{equation}
\end{enumerate}

 We close this section by recalling from \cite{BCD} that  we can equip $\fHeis$ with a metric $\widehat{d}$. That is for all $\hat{w}=(m,\ell,\lambda)$ and  $\hat{w}'=(m',\ell',\lambda')\in\THeis$, we define
  \begin{equation}\label{dist}
  \widehat{d}(\hat{w}, \hat{w}')=|\lambda(m+\ell)-\lambda'(m'+\ell')|_1+|(m-\ell)-(m'-\ell')|+n|\lambda-\lambda'|
   \end{equation}
  where $|\cdot|_1$ denotes the Manhattan norm  on $\mathbb{R}^n$, i.e. for $m \in \mathbb{R}^n$, $|m|_1=\sum_{i=1}^n |m_i|$. 
 Similarly, recall that the norm $|k|$ for $k\in\Z^n$ is given in the notation at the end of our introduction. 
   Taking limits $\hat{w}'\to(0,0,0)$ in relation~\eqref{dist}, we also set \begin{equation}\label{dist0}
  \widehat{d}_0(\hat{w})=|\lambda|(|m+\ell|+n)+|m-\ell| \, ,
  \end{equation}
  for all $\hat{w}=(m,\ell,\lambda)$ in $\fHeis$.
  
\subsection{Some differential operators on $\fHeis$}\label{sec:diff-on-tilde-H}

It is known  that functions in $\R^n$ with compactly supported Fourier transform have several regularity properties, which can be extended to weighted spaces. In the following subsections we shall study these properties in the Heisenberg group, extending~\cite{MW} to our geometric setting. As a preparation to this task, in this section we introduce some
differentiations on $\fHeis$. Notice that we borrow those differentiation operators  from~\cite{BCD}. They nicely preserve the property of Fourier transform that converts multiplication in direct coordinates into differentiation in Fourier modes.
\begin{definition} As in the general notation given in the introduction, we denote by $\mathbb{N}$ the set $\{0,1,2,\ldots\}$. 
Recall that the space $\fHeis$ is introduced in Definition~\ref{def:projective-fourier}. We define a differential operator $\HDelta$ acting on functions 
$g:\fHeis\to\C$ in the following way: for any $\hat{w}=(m,\ell, \lambda)$ we set~$\widehat{ w}_j^\pm\defeq(m\pm e_j,\ell\pm e_j,\lambda)$, where $e_j$ is the element of $\mathbb{N}^n$  with the $j$-th component $1$ and all the  other components $0$.
Then $\HDelta g$ is given by
\begin{multline}\label{b1}
 \HDelta g(\hat{w}) \defeq \\
  - \frac{ 1}{2|\lambda|} ( |m+\ell| +n)\, g(\hat{w}) 
+\frac{1}{2|\lambda|} \sum_{j=1} ^n \left( \sqrt{(m_j+1) (\ell_j+1)}\, g(\hat{w}_j^+) 
+\sqrt{m_j\ell_j}\, g(\widehat{ w}_j^-)\right).
\end{multline}
If in addition $g$ is differentiable with respect to $\lambda$, we also define a differentiation with respect to the $\lambda$ variable:
 \begin{equation}\label{b2}
 \HDL g(\hat{w})  \defeq  {\partial_\lambda g}(\hat{w})+ \frac{n}{2\lambda} g(\hat{w})
+\frac{1}{2 \lambda}\sum_{j=1}^n \left(
\sqrt{m_j\ell_{j} }\, g(\hat{w}_j^-)-\sqrt{(m_j+1)(\ell_j+1)} \, g({\hat{w}_j}^+) \right).
 \end{equation}
\end{definition}

\begin{remark}
The operator $\HDelta$ should be considered as a differentiation in $\fHeis$. One can justify this interpretation easily in two ways:\begin{enumerate}[wide, labelwidth=!, labelindent=0pt, label= {(\alph*)}]
  \setlength\itemsep{-.1in}  
  \item On the diagonal in $\fHeis$, namely for $\hat{w}=(m,m,\lambda)$, it is readily checked that \eqref{b1} becomes \begin{equation}\label{b1'}
        \HDelta g(m,m,\lambda) =
   \frac{ 1}{2|\lambda|} \sum_{j=1} ^n m_j \big(g(\widehat{ w}_j^-)-g(\hat{w})\big)+(m_j+1) \big(g(\widehat{ w}_j^+)-g(\hat{w})\big),
    \end{equation} and the right hand side of \eqref{b1'} is easily interpreted as a Laplace type operator.\\
    \item A second justification of $\HDelta$ as a Laplacian is the content of the technical Lemma \ref{mult-diff} below. It states that multiplying by  $|x|^2+|y|^2$ in $\Heis$ corresponds to applying $\HDelta$ in $\fHeis$. This is analogous to similar relations in $\mathbb{R}^d$ involving the Laplacian.
\end{enumerate}
Notice that  the same kind of observations are valid for the operator $\HDL$ in \eqref{b2}.
\end{remark}
\begin{remark}
 In \eqref{b1}-\eqref{b2}, we adopt the following convention:   if there exists $j\in\{1,...,n\}$ such that either $m_j=-1$ or $\ell_j=-1 $, then we set $g(m, \ell,\lambda)=0$. This allows us to define quantities such as $\HDelta g(0,0,\lambda)$  or  $ \HDL g(0,0,\lambda)$.
 \end{remark}
 
 We now label some general analytic definitions on $\Heis$ and $\fHeis$, which will be used throughout the paper. We first introduce some multiplication operators on $\Heis$.
 
 \begin{definition}\label{def:schwartz-and-multiplication}
 On the Heisenberg group $\Heis$, we consider the following objects:
 \begin{enumerate}[wide, labelwidth=!, labelindent=0pt, label= \emph{\textbf{(\roman*)}}]
\setlength\itemsep{.05in}

\item
The Schwartz space of smooth functions with rapid decay on $\Heis$ is identified with the Schwartz space on $\R^{2n+1}$. It is denoted by $\cs(\Heis)$.

\item
The multiplication operators $M^2$ and $M_0$ on the Schwartz space $\cs(\Heis)$:  for any $q=(x,y,z)\in\Heis$, we set
 \begin{equation}\label{b3}
 (M^2f)(q):=(|x|^2+|y|^2)f(q),\quad\text{and}\quad (M_0f)(q):=-iz f(q).
 \end{equation}

\end{enumerate}
  \end{definition}
  Next we define what is meant by Schwartz space on the dual set $\fHeis$. As expected, it is related to smoothness and decay properties. 
  \begin{definition}\label{def:Schwartz-space} Let $\fHeis$ be the set introduced in Definition~\ref{def:projective-fourier}. Then $\mathcal{S}(\fHeis)$ is the set of functions $\tilde{\phi}$ on $\fHeis$ such that 
 \begin{enumerate}[wide, labelwidth=!, labelindent=0pt, label= \emph{\textbf{(\roman*)}}]
\setlength\itemsep{-.1in}
      \item For all $(m,\ell)\in\mathbb{N}^{2n}$, the map $\lambda\mapsto \tilde{\phi}(m,\ell,\lambda)$ is smooth on $\mathbb{R}^*$.\\
      \item For all $N\in\mathbb{N}$, the functions $\HDelta^N\tilde{\phi}$ and $\HDL^N \tilde{\phi}$ decay faster than any expression of the form $d_0(\hat{w})^{-\kappa}$ for all $\kappa>0$, where $d_0$ is introduced in \eqref{dist0}.
  \end{enumerate}
      Note that the space $\mathcal{S}(\fHeis)$ can be equipped with a family of semi-norms\begin{equation}\label{semi-norms}
\lVert\tilde{\phi}\rVert_{N,\,\kappa,\,\mathcal{S}(\fHeis)}=\sup_{\hat{w}\in\fHeis}\left(1+d_0(\hat{w})\right)^\kappa\left(\HDelta^N\tilde{\phi}(\hat{w})+\HDL^N\tilde{\phi}(\hat{w})\right).
      \end{equation}
  \end{definition}

  \begin{remark}
      Our Definition \ref{def:Schwartz-space} is not completely exact, and we refer to \cite[Definition 2.5]{BCD} for the complete version. Specifically we have refrained from introducing the operator $\widehat{\Sigma}_0$ and the completion $\widehat{\bf H}^n$, which play a prominent role in~\cite{BCD}, for sake of clarity. For our purposes it will be sufficient to consider the semi-norms \eqref{semi-norms}.
  \end{remark}
  One of the advantages of introducing the space $\mathcal{S}(\fHeis)$ is that we now have an isomorphism property for the Fourier transform. This result is borrowed from \cite[Theorem 2.6]{BCD}.

  \begin{proposition}\label{isomorphism}
      Let $\mathcal{S}(\Heis)$ be the space given in Definition \ref{def:schwartz-and-multiplication} and consider $\mathcal{S}(\fHeis)$ introduced in Definition \ref{def:Schwartz-space}. Then the projective Fourier transform $f \to \hat{f}$ is an isomorphism from $\mathcal{S}(\Heis)$ to $\mathcal{S}(\fHeis)$, with inverse given by \eqref{eq-FT-int-inv}.
  \end{proposition}
One of the main results contained in \cite{BCD} is that the multiplication operators introduced in Definition \ref{def:schwartz-and-multiplication} are compatible with the differentiations on $\fHeis$. We label this result in the following lemma.

\begin{lemma}\label{mult-diff}
Recall that the operators $\HDelta$ and $\HDL$ are respectively defined by~\eqref{b1} and~\eqref{b2}. 
The multiplications $M^2$ and $M_0$ are given in~\eqref{b3}. Then for any $f\in \cs(\Heis)$, we have
\begin{equation}\label{eq-FT-M}
\widehat{M^2f}=-\HDelta {\Ff},
\quad\text{and}\quad
 \widehat{M_0f}=\HDL {\Ff}.
\end{equation}
\end{lemma}

\subsection{Gevrey classes on the Heisenberg group}
With the preliminaries of Sections~\ref{sec:fourier-general} and~\ref{sec:diff-on-tilde-H} in hand, we are now ready to define the Gevrey classes on $(\fHeis, \widehat{d})$ that characterize regularities of functions on $\fHeis$. Recall that the Gevrey class on $\R^n$ with index $s$ is defined as the set of infinitely differentiable functions that satisfies (for any multiindex $k$)
\begin{equation}
\label{g-est}
\sup_K|\partial^k f|\le C^{|k|+1}(k!)^s \, ,
\quad\text{for any compact $K\subset \R^n$.}
\end{equation}
 The Gevrey classes $s\ge1$ interpolate between analytic functions and $C^\infty$ functions. On Heisenberg groups, this kind of property will be quantified through the operators $\HDelta$ and $\HDL$.

 \begin{definition}\label{def-diff}
 Consider the operators $\HDelta$ and $\HDL$ defined in~\eqref{b1} and~\eqref{b2}, and a parameter $s\geq 1$. We define $\mathcal{G}^s(\THeis)$ to be the set of functions $\tilde{g}$ on $\THeis$ such that:
 \begin{enumerate}[wide, labelwidth=!, labelindent=0pt, label= \emph{(\roman*)}]
\setlength\itemsep{.05in}
     \item 
     For any $(m,\ell)$ in $\mathbb{N}^{2n}$, the map $\lambda\mapsto \tilde{g}(m,\ell, \lambda)$ is smooth on $\mathbb{R}^{*}$.
     
     \item 
     For every compact set $K\subset(\fHeis, \widehat{d})$, there exists $C<\infty$ such that for every integer $N\geq 0$, 
     \begin{equation}\label{gevrey-est}
      \sup_K\left( \max\left(|\HDelta^N \tilde{g}|,\,  |\HDL^N \tilde{g}|
     \right)\right)\leq C^{2N+1}[(2N)!]^s.\end{equation}
  \end{enumerate}
 \end{definition}
 
In $\R^{d}$, the exponential decay of a function $f$ can be characterized by the fact that $\cf f$ is in a Gevrey class. We now show that the same is true on Heisenberg groups. We start with
 the proposition below, which shows that decay of a function implies regularity of its Fourier transform in a Gevrey class.
  
 \begin{proposition} For $p= (x,y,z)\in\Heis$, recall from~\eqref{eq:def-homogeneous-norm} that $\lvert p \rvert_{\mathrm{h}}=(|x|^2+|y|^2+|z|)^{1/2}$ is the homogeneous norm of $p$. Consider a parameter $s\ge 1$ and a function $f\in\mathcal{S}(\Heis)$ satisfying the following relation:
\begin{equation}\label{eq-a}
|f(p)|\leq Ce^{-\varepsilon\lvert p \rvert_{\mathrm{h}}^{1/s}} ,
\quad\text{for all}\quad p\in\Heis \, ,
\end{equation}
for some  $C>0$ and $\varepsilon>0$. Then the Fourier transform  $\Ff$ lies in  $\mathcal{G}^s(\THeis)$.
\end{proposition}

\begin{proof} 
We use the integral representation \eqref{eq-FT-int} of the Fourier transform $\Ff$. First, it is not difficult to see   that the integral kernel $K_{m,\ell, \lambda}$ satisfies \begin{equation}\label{eq-b}
|{K}_{m,\ell, \lambda}(q)|\leq 1,\quad\text{for all $(m,\ell, \lambda)\in \fHeis$ and $q\in\Heis$.}
\end{equation} To this aim, 
recall  relation \eqref{eq-FT-K} defining the integral kernel $K_{m,\ell, \lambda}$. Bounding the complex exponential by $1$, applying the Cauchy-Schwarz inequality in the right hand side of \eqref{eq-FT-K} and recalling that $(\Phi^\lambda_m)_{m\in\N^{n}}$ is an orthonormal family of $L^2(\mathbb R^n)$, we easily get the estimate~\eqref{eq-b}.

Next, let $N\geq 0 $ be an integer. Invoking relation \eqref{eq-FT-M} and the definition \eqref{eq-FT-int} of Fourier transform, we get    
\[
|\HDelta^N \Ff(m,\ell, \lambda)|=\left|\int_{\Heis}\overline{e^{i\lambda z}\, K_{m,\ell,\lambda}(q)}\,(|x|^2+|y|^2)^{N} f(q)\,d\mu(q)\right|.
\]
Plugging \eqref{eq-a} and \eqref{eq-b} into the equality above, we end up with the following relation for all $(m,\ell, \lambda)\in \fHeis$:
\begin{equation}\label{eq-c}
|\HDelta^N \Ff(m,\ell, \lambda)|\leq C\int_{\Heis}(|x|^2+|y|^2)^{N} e^{-\varepsilon\lvert q\rvert_{\mathrm{h}}^{1/s}} d\mu(q).
\end{equation}
Using the exact same arguments, we also have, for all $(m,\ell, \lambda)\in \fHeis$,
\begin{align}
\nonumber|\HDL^N \Ff(m,\ell, \lambda)|&=\left|\int_{\Heis}\overline{e^{i\lambda z}\, K_{m,\ell,\lambda}(q)}\,|z|^{N} f(q)\,d\mu(q)\right|\\\label{eq-d}
&\leq C\int_{\Heis}|z|^{N} e^{-\varepsilon\lvert q\rvert_{\mathrm{h}}^{1/s}} d\mu(q).
\end{align} 
Summarizing \eqref{eq-c} and \eqref{eq-d}, we have thus obtained that 
\begin{equation}
\label{eq-e}
M:=\sup_{\hat{w}\,\in\, \fHeis}\hspace{-0.1cm}\left( \max\left(|\HDelta^N \hat{f}\, (\hat{w})|,\,  |\HDL^N \hat{f}\,(\hat{w})|
\right)\right)\leq C\int_\Heis\lvert q\rvert_{\mathrm{h}}^{2N} e^{-\varepsilon\lvert q\rvert_{\mathrm{h}}^{1/s}} d\mu(q).
\end{equation}
We now resort to the elementary estimate $t^{2N}e^{-t}\leq (2N)!$, applied to $t=\frac{\varepsilon}{2s}\lvert q\rvert_{\mathrm{h}}^{1/s}$, in the right hand side of \eqref{eq-e}. We have 
\[
\left(\frac{\varepsilon}{2s}\right)^{2N}|q|_h^{2N/s}\,e^{-\frac{\varepsilon}{2s}|q|_h^{1/s}}\leq (2N)!,
\]
and raising both sides of the inequality to the power of $s$, we get  
 \begin{equation}\label{eq-f}
 |q|_h^{2N}\,e^{-\frac{\varepsilon}{2}|q|_h^{1/s}}\leq C_1^{2N}[(2N)!]^s,
 \end{equation}
 where $C_1=(\varepsilon/2s)^{-s}$. Applying \eqref{eq-f} to the  right hand side of \eqref{eq-e} gives
\begin{align*}
M&\leq C\cdot C_1^{2N}[(2N)!]^s\int_\Heis e^{-\frac{\varepsilon}{2}\lvert q\rvert_{\mathrm{h}}^{1/s}} d\mu(q). \end{align*}Thus $M\leq C_2^{2N+1}[(2N)!]^s$, where $C_2=C_1+C\int e^{-\frac{\varepsilon}{2}\lvert q\rvert_{\mathrm{h}}^{1/s}} d\mu$. (Note that the constant $C_2$ is independent of $N$.) This proves that  $\Ff\in\mathcal{G}^s(\THeis)$.   
\end{proof}

We have just seen that if a function $f$ defined on $\Heis$ has a proper exponential decay, then its Fourier transform is in a space $\mathcal{G}^s$. We now examine the converse result: if a function $\tilde{\phi}$ defined on $\fHeis$ is in a class $\mathcal{G}^s$, then its inverse Fourier transform decays exponentially. This is detailed in the proposition below.  

  \begin{proposition} \label{exp-decay-est}
  Recall that the space $\mathcal{S}(\fHeis)$ is introduced in Definition \ref{def:Schwartz-space}. Let $s \ge 1$ and $\tilde{\phi}$ be an element of   $\mathcal{G}^s(\THeis)\cap \cs(\fHeis)$ with bounded support in $\fHeis$. We call $\phi$ its inverse Fourier transform, as defined in~\eqref{eq-FT-int-inv}. Then there exist constants $C=C_{\phi}>0$ and $\varepsilon=\ep_{\phi}>0$ such that  
  \begin{equation}\label{eq-g}
|\phi(q)|\leq Ce^{-\varepsilon\lvert q \rvert_{\mathrm{h}}^{1/s}}\ \, \text{for all $q\in\Heis$}. 
\end{equation}
    \end{proposition}
\begin{proof} Let $\tilde{\phi}\in \mathcal{G}^s(\THeis)\cap \cs(\fHeis)$. Owing to Proposition \ref{isomorphism}, the inverse Fourier transform $\phi$ of $\tilde{\phi}$ is well defined.
 Then it suffices to show that the above estimate~\eqref{eq-g} holds uniformly over $|q|_h\geq 1$. By \eqref{eq-FT-M} and \eqref{eq-FT-int-inv}, we have that for any integer $k\ge1$ and $q\in\Heis$,
 \begin{align*}
 |M^{2k}\phi(q)|
 =\left|\frac{2^{n-1}}{\pi^{n+1}}\int_{\fHeis}e^{i\lambda z} K_{\hat{w}} (q)((-\HDelta)^{k}\tilde{\phi})(\hat{w})\,d\hat{w}
    \right|.
 \end{align*} 
 Owing to \eqref{eq-b}, our assumption (ii) in Definition \ref{def-diff} and the fact that $\tilde{\phi}$ 
 has bounded support, we obtain that for all $q\in \Heis$,
\begin{equation}
\label{eq-h}
 |M^{2k}\phi(q)|\leq C^{2k+1}[(2k)!]^s.
\end{equation} 
 Similarly we can also obtain that for all $q\in \Heis$,
\begin{align}\label{eq-i}
|M_0^{k}\phi(q)|
\leq C^{2k+1}[(2k)!]^s.
\end{align}
 Given any integer $k\geq 1$, we choose $N$ such that $2N-1\leq \frac{k}{s}\leq 2N$. Then for any $|q|_h\geq 1$, we have 
\begin{align}\label{eq-mid-1}
    |q|_h^{k/s}|\phi(q)| \leq |q|_h^{2N}|\phi(q)| &\leq C_1^{2N}(|M^{2N}\phi(q)|+|M_0^{N}\phi(q)| )\notag \\
    &\leq C_2^{2N+1}[(2N)!]^s,
\end{align}where we have used \eqref{eq-h} and \eqref{eq-i} for the last inequality. 
 Since $(2N)!\leq 2^{2N}(2N-1)!$, one can bound the right hand side of \eqref{eq-mid-1} accordingly in order to get
 \begin{align}\label{eq-mid-2} |q|_h^{k/s}|\phi(q)| \leq 
 C_3^{2N+1}[(2N-1)!]^s\leq C_3^{2N+1}[(2N-1)]^{s(2N-1)}\leq C_4^{k+1}k^k.
\end{align}
 It then follows from \eqref{eq-mid-2} that there exists $C>0$ such that uniformly in $q\in \Heis$,
 \begin{equation}\label{eq-j}
e^{\varepsilon|q|_h^{1/s}}|\phi(q)|=\sum_{k=0}^\infty\frac{(\varepsilon|q|_h^{1/s})^k}{k!}|\phi(q)|\leq C\sum_{k=0}^\infty\frac{(\varepsilon C k)^k}{k!}, 
 \end{equation}
uniformly for some $C>0$.  At last since $1/k!\leq (e/k)^k$, it suffices to choose $\varepsilon$ so that $\varepsilon Ce<1$ in the right hand side of \eqref{eq-j} to obtain the desired estimate \eqref{eq-g}.
\end{proof}

Up to now we have seen general properties of Gevrey classes. We now produce an example of  functions in $ \mathcal{G}^s$ with bounded support in Fourier modes. This proves in particular that the set of functions $\tilde{\phi}$ satisfying the assumptions of Proposition \ref{exp-decay-est} is nontrivial.

\begin{proposition}\label{prop 2.13}
Let $f$ be a function on $[0,\infty)^n\times\mathbb{Z}^n\times\mathbb{R}$ satisfying the following: 
\begin{enumerate}[wide, labelwidth=!, labelindent=0pt, label= \emph{(\alph*)}]
\setlength\itemsep{.0in}
    \item The support of $f$ is such that $\textnormal{Supp}(f)\subset[0,\infty)^n\times\{0\}\times\mathbb{R}$.   \\[-0.2cm]
    \item  There exists $C=C_f<\infty$ and $s\ge 1$ such that for every   $\alpha\in \mathbb{N}^{n+1}$, uniformly in $(x,\kappa,\lambda)$ we have \begin{equation}\label{est-gevrey}
    |\partial_{x,\lambda}^\alpha f(x_1,...,x_n,\kappa,\lambda)|\leq C^{|\alpha|+1}(\alpha!)^s.\end{equation} 
\end{enumerate} 
\noindent
For $\hat{w}=(m,\ell,\lambda)\in \fHeis$, we define a function $\Theta_f(\hat{w})$ by \begin{equation}\label{def-theta-f}
\Theta_f(\hat{w})\defeq f(|\lambda|R(m,\ell), m-\ell,\lambda),\quad \text{with}\quad R(m,\ell)\defeq (m_j+\ell_j+1)_{1\leq j\leq n}.\end{equation} Then $\Theta_f\in \mathcal{G}^s(\fHeis)$.\end{proposition}

\begin{proof}
Item (i) in Definition \ref{def-diff} is easily verified thanks to the fact that $f$ is a smooth function. We will thus focus on checking relation \eqref{gevrey-est}. We will divide our proof in several steps.

\noindent
\textit{Step 1: Estimates for $\HDelta\Theta_f$.}
Due to the fact that the support of $f$  in the second variable is reduced to the singleton $\{0\}$, one can in fact  restrict ourselves to variables of the form $\hat{w}=(m,m,\lambda)$. For notational simplicity, we thus set $R(m)=R(m,m)$ and $f(|\lambda|R(m),\lambda)=f(|\lambda|R(m),0,\lambda)$. Then using the definition \eqref{b1} of $\HDelta$ as well as relation \eqref{def-theta-f}, some elementary algebraic manipulations show that 
\begin{multline}\label{m}
(\HDelta\Theta_f)(m,m,\lambda)
=-\frac{1}{2|\lambda|}\Big[
(2|m|+n)f(|\lambda|R(m),\lambda) \\
-\sum_{j=1}^n\big((m_j+1)f(|\lambda|R(m+e_j),\lambda)+m_jf(|\lambda|R(m-e_j),\lambda)\big)
\Big].
\end{multline}
Since we have assumed $|f|\leq C$, it  follows easily from \eqref{m} that 
\begin{equation}\label{est-i}
\big|(\HDelta\Theta_f)(m,m,\lambda)\big|\leq \frac{2|m|+n}{|\lambda|}\, C.
\end{equation}
Moreover, iterating the estimate \eqref{est-i} one gets the following upper bound for $\HDelta^2\Theta_f$:
\begin{multline*}
\big|(\HDelta^2\Theta_f)(m,m,\lambda)\big|  \\
\leq 
\frac{1}{2|\lambda|}\left((2|m|+n)\,\frac{2|m|+n}{|\lambda|}+\sum_{j=1}^n (m_j+1)\, \frac{2|m|+2+n}{|\lambda|}+m_j\, \frac{2|m|+n}{|\lambda|}\right) C \, ,
\end{multline*}
which can be simplified as
\begin{equation}\label{est-j}
\big|(\HDelta^2\Theta_f)(m,m,\lambda)\big|
\leq
\frac{(2|m|+n)(2|m|+n+2)}{|\lambda|^2}\, C.
\end{equation}
Therefore an induction procedure yields, for $p\geq 0$, \begin{align}\nonumber
\big|(\HDelta^{p+1}\Theta_f)(m,m,\lambda)\big|&\leq \frac{(2|m|+n)(2|m|+n+2)...(2|m|+n+2p)}{|\lambda|^{p+1}}\,C\\[0.1cm]\label{est-k}&\leq\frac{(2|m|+n+2p)!}{(2|m|+n-2)!}\,|\lambda|^{-(p+1) }\,C.\end{align}
Recalling the elementary binomial relation $(j+k)!\leq j!\,k!\, 2^{j+k}$,   we thus have
\begin{equation}\label{est-c}
\big|(\HDelta^{p+1}\Theta_f)(m,m,\lambda)\big|\leq (2p+2)!\,2^{2|m|+n+2p} \,|\lambda|^{-(p +1)}\,C.    
\end{equation} As mentioned above, due to the fact that $\Theta_f$ is supported on the diagonal, \eqref{est-c} is enough to prove that   for any compact set $K\subset \fHeis$, there exists a 
 positive constant $C_K$ such  that for all $p\geq 0 $ we have 
 \begin{equation}\label{c1}
 \sup_K\big|(\HDelta^{p}\Theta_f)\big|\leq C_K^{2p+1}[(2p)!]^s.
 \end{equation}

\noindent
\textit{Step 2: Estimates for $\HDL\Theta_f$.} We proceed to prove an estimate of the form~\eqref{c1} for $\HDL$. First decompose the right hand side of \eqref{b2} as 
\begin{equation}\label{d}    
\HDL=\partial_\lambda+\UHDL,
\end{equation} where the operator $\UHDL$ is defined by \begin{equation}\label{e}
\UHDL g(\hat{w})= \frac{n}{2\lambda} g(\hat{w})
+\frac{1}{2 \lambda}\sum_{j=1}^n \left(
\sqrt{m_j\ell_{j} }\, g(\hat{w}_j^-)-\sqrt{(m_j+1)(\ell_j+1)} \, g({\hat{w}_j}^+) \right).
\end{equation}
First let us bound the term $\partial_\lambda \Theta_f$ from the right hand side of \eqref{d}. Similarly to the estimates for $\HDelta$, it is enough to focus our attention on  diagonal points of the form $(m,m,\lambda)$ in $\fHeis$. Also notice that we are estimating all our quantities on compact sets of $\fHeis$. Those sets avoid the hyperplane $\la=0$, which means that we can consider $\la\mapsto|\la|$ as a smooth function. Then it is readily checked from expression \eqref{def-theta-f} that
\begin{equation}\label{g}
\partial_\lambda\Theta_f(m,m,\lambda)=(\partial_\lambda f)(|\lambda|R(m), \lambda)+\text{sgn}\,(\lambda)\sum_{j=1}^n(2m_j+1)(\partial_{x_j}f)(|\lambda|R(m), \lambda).
\end{equation} 
Iterating this differentiation, it can be verified that  $\partial_\lambda^j\Theta_f$ satisfies the estimate
\begin{equation}\label{g1}
\big|\partial_\lambda^j\Theta_f(m,m,\lambda)\big|\leq (2|m|+n+1)^j C^{j+1} (j!)^s.
\end{equation} 
As far as $\UHDL$ in \eqref{e} is concerned, we get 
\begin{align}\nonumber
&\UHDL\Theta_f(m,m,\lambda)\\[0.1cm]
\label{h}&=\frac{n}{2\lambda}f(|\lambda|R(m), \lambda)+\frac{1}{2\lambda}\sum_{j=1}^n \big(m_jf(|\lambda|R(m-e_j), \lambda)-(m_j+1)f(|\lambda|R(m+e_j), \lambda)\big).
\end{align}
Let us now estimate iterates of $\UHDL\Theta_f$. Indeed
like in \eqref{est-i}, one can simply bound $|f|$ by $C$ in  \eqref{h}. This enables to write
\[
\big|(\UHDL\Theta_f)(m,m,\lambda)\big|\leq \frac{|m|+n}{|\lambda|}\, C.\] 
One can then proceed as in \eqref{est-j} and get 
 \begin{align*}\big|(\UHDL^2\Theta_f)(m,m,\lambda)\big|&\leq 
 \frac{n}{2|\lambda|}\,\frac{|m|+n}{|\lambda|}\, C+\frac{1}{2|\lambda|}\sum_{j=1}^n\left( m_j\,\frac{|m|+n}{|\lambda|}\, C+(m_j+1)\, \frac{|m|+n+1}{|\lambda|}\, C\right)\\[0.1cm]
&\leq  
\frac{(|m|+n)(|m|+n+1)}{|\lambda|^2}\, C.
 \end{align*}
 Then similarly to \eqref{est-k}, an induction procedure yields that for $k\geq 0$ we have
 \begin{equation}\label{o}
 \big|(\UHDL^k\Theta_f)(m,m,\lambda)\big|\leq\frac{1}{|\lambda|^k}\,\frac{(|m|+n+k-1)!}{(|m|+n-1)!}\,C.
\end{equation}
With \eqref{o} in hand, we can go back to the decomposition~\eqref{d} of $\HDL$.  According to this decomposition, if one wishes to upper bound $\HDL^{\,k}\Theta_f$, an estimation of  derivatives of the following form is in order:
\[
\partial_\lambda^j\left(\UHDL\right)^{k-j}\Theta_f.
\]
We will omit details of those computations for sake of conciseness. However, a slight elaboration of \eqref{h} reveals that 
 that  $(\UHDL^k\Theta_f)(m,m,\lambda)$ is the product of $\frac{1}{(2\lambda)^k}$ and a sum of $N_k$ terms of translates of $f$, where $N_k$ satisfies an inequality similar to \eqref{o}. Namely we have
 \begin{equation}\label{n}
 N_k\leq 2^k\,\frac{(|m|+n+k-1)!}{(|m|+n-1)!}=\prod_{j=0}^{k-1}\,(2|m|+2n+2j).
 \end{equation}
 Therefore, differentiating with respect to $\lambda$ and putting together \eqref{g1}, \eqref{o} and \eqref{n}, we get 
 \begin{multline}\label{p}
     \left|\frac{d^j}{d\lambda^j}(\UHDL^k\Theta_f)(m,m,\lambda)\right|
     \leq
     \sum_{p=0}^j\binom{j}{p}\frac{(k+p-1)!}{(k-1)!}\,\frac{1}{|\lambda|^{k+p}}\,(2|m|+2k+n+1)^{j-p} \\
     \times\frac{(|m|+n+k-1)!}{(|m|+n-1)!}\,C^{j-p+1}((j-p)!)^s.
     \end{multline}
     In \eqref{p} we now simply bound $((j-p)!)^s$ by $(j!)^s\,/\,p!$ and  thus obtain
     \begin{multline}\label{p1}
     \left|\frac{d^j}{d\lambda^j}(\UHDL^k\Theta_f)(m,m,\lambda)\right|
     \leq\sum_{p=0}^j\binom{j}{p}\,\binom{k+p-1}{k-1}\,\frac{1}{|\lambda|^{k+p}}\,(2|m|+2k+n+1)^{j-p}\\
     \times\frac{(|m|+n+k-1)!}{(|m|+n-1)!}\,C^{j-p+1}(j!)^s.
     \end{multline}
     Exactly like we did for $\eqref{est-c}$, we now invoke the relation $(j+k)!\leq j!\,k!\, 2^{j+k}$. This enables to write 
     \[\frac{(|m|+n+k-1)!}{(|m|+n-1)!}\leq 2^{|m|+n+k-1}k!,\quad \binom{j}{p}\leq 2^j,\quad \text{and}\quad\binom{k+p-1}{k-1}\leq 2^{k+p-1}.\]
     Plugging the above into \eqref{p1}, we obtain
     \begin{equation*}
\left|\frac{d^j}{d\lambda^j}(\UHDL^k\Theta_f)(m,m,\lambda)\right|
\leq
\sum_{p=0}^j2^j\,2^{k+p-1}\,\frac{1}{|\lambda|^{k+p}}\,2^{j   -p}\,(|m|+k+n)^{j-p}\left(2^{|m|+n+k-1}k!\right)C^{j-p+1}(j!)^s.
\end{equation*}
     Rearranging the terms in this expression, we discover that 
     \begin{equation}\label{q}
     \left|\frac{d^j}{d\lambda^j}(\UHDL^k\Theta_f)(m,m,\lambda)\right|
     \leq 2^{2j+2k+|m|+n}\,C^{j+1}\,k!\,(j!)^s\, \sum_{p=0}^j\frac{1}{|\lambda|^{k+p}}\,(|m|+k+n)^{j-p}.
 \end{equation}
 Let us now return to the iterations of order $N\geq1$ for $\HDL$. Namely invoking \eqref{d} we write
 \[\HDL^N\Theta_f = \sum_{j=0}^N\binom{N}{j}\,\frac{d^j}{d\lambda^j}\,\UHDL^{N-j}\Theta_f.\]
 Gathering \eqref{q} into the above identity we obtain
\begin{multline}\label{c2}
     \left|(\HDL^N\Theta_f)(m,m,\lambda)\right|\\
     \leq\sum_{j=0}^N\binom{N}{j}\, 2^{2N+|m|+n}\,C^{j+1}\,(N-j)!\,(j!)^s \sum_{p=0}^j\frac{1}{|\lambda|^{N-j+p}}\,(|m|+N-j+n)^{j-p}.
     \end{multline}
     In addition one can simplify the binomial terms and rearrange terms in~\eqref{c2}, which enables to get 
     \begin{equation*} \left|(\HDL^N\Theta_f)(m,m,\lambda)\right|
     \leq \sum_{j=0}^NN!\, 2^{2N+|m|+n}\,C^{j+1}\,(j!)^{s-1} \sum_{p=0}^j\frac{1}{|\lambda|^{N-j+p}}\,(|m|+N-j+n)^{j-p}.
     \end{equation*}
     Bounding trivially $j!$ by $N!$ and $C^j$ by $C^N$ above, this yields
     \begin{equation}\label{r}
    \left|(\HDL^N\Theta_f)(m,m,\lambda)\right|
\leq 
(N!)^s\, 2^{2N+|m|+n}\,C^{N+1}\, \sum_{j=0}^N\sum_{p=0}^j\frac{1}{|\lambda|^{N-j+p}}\,(|m|+N-j+n)^{j-p}.
\end{equation} 
One can further simplify expression \eqref{r}. Specifically, if $(m,m,\lambda)$ lies in a compact set $K\subset \fHeis$, we find that there exists a constant $C_1=C_{1, K}>1$ such that
 \begin{eqnarray}\label{c3}
  \mathcal{A}_K
  &\equiv&
  \sup_K \sum_{j=0}^N\sum_{p=0}^j\frac{1}{|\lambda|^{N-j+p}}\,(|m|+N-j+n)^{j-p} \notag\\
 & \leq &
  \sum_{j=0}^N\sum_{p=0}^j C_1^{N-j+p}(C_1+N-j)^{j-p}
  = C_{1}^{N} \sum_{j=0}^N\sum_{p=0}^j C_{1}^{-(j-p)}(C_1+N-j)^{j-p}.
  \end{eqnarray}
  We now resort to some very rough estimates for the right hand side of~\eqref{c3}. Namely bound each $C_{1}^{-(j-p)}$ by 1 and every term $(C_1+N-j)$ by $C_{1}+N$. We obtain
  \begin{equation*}
\mathcal{A}_K \le C_{1}^{N} N (N+1) \lp C_{1}+N  \rp^{N} \, .
\end{equation*}
Invoking the elementary bound $N!\ge (N/3)^{N}$, we thus get the existence of a constant $C_{2}>C_{1}$ such that
  \begin{equation*}
\mathcal{A}_K
\le
C_2^{N}\, N! \, .
\end{equation*}
Plugging this inequality back into \eqref{r} and recalling that we are considering a parameter $s\ge 1$,
  we conclude that for $N\geq 0$ the following holds true:
  \begin{equation}\label{c4}
     \sup_K\left|(\HDL^N\Theta_f)(m,m,\lambda)\right|
     \leq (N!)^s\,C_3^N\, C^{N+1}\, 2^{N+2}\, C_2^{N}\, N! 
     \leq C_4^{2N+1} \, [(2N)!]^s. 
  \end{equation}

\noindent
\textit{Step 3: Conclusion.}  
Summarizing our considerations so far, we have obtained relations~\eqref{c1} and~\eqref{c4}. It is now straightforward to see that~\eqref{gevrey-est} is fulfilled, which finishes the proof.
\end{proof}

\subsection{Bernstein's lemma}\label{sec-Berns}

As a key step in the definition of weighted Besov spaces, one would like to accurately quantify the effect of derivatives on functions whose support is bounded in Fourier modes. This kind of result is often referred to as Bernstein's lemma, see \cite[Section 2.2]{MW}. We will prove this result for $\Heis$ in the current section, beginning with a few preliminary notions.
Let us start by introducing a broad class of weights which will be used in the sequel. 

\begin{definition}\label{def:v-moderate}
We are given two functions $v, w:\Heis\to \mathbb{R}_+ $, considered as weights on $\Heis$. We say that $w$ is $v$-moderate if for every $p, q\in\Heis$,we have
\begin{equation}\label{d1}
w(p q)\leq v(p)w(q).
\end{equation}
\end{definition}

With the notion of $v$-moderate weight in hand, we now state a weighted version of Young's inequality adapted to the Heisenberg group framework.
\begin{proposition}\label{young-ineq} 
Let $v, w:\Heis\to \mathbb{R}_+ $ be two weight functions, and assume that $w$ is $v$-moderate according to Definition~\ref{def:v-moderate}. Consider 3 parameters $\alpha, \beta, \gamma\in [1,\infty]$ such that 
\[
\frac{1}{\alpha}+\frac{1}{\beta}=\frac{1}{\gamma}+1.
\] 
Then for every measurable functions $f, g:\Heis\to \mathbb{R}$, we have
\begin{equation}\label{d2}
\|(f \star g)w\|_{L^\gamma}\leq \|f\,v\|_{L^\alpha}\|g\,w\|_{L^\beta}.
\end{equation}
\end{proposition}

\begin{proof}
This follows from a similar argument as in the case of $\mathbb{R}^n$ (see, e.g., \cite[Theorem 2.4]{MW}). Recall that for all $p\in\Heis$ we have
\begin{equation*}
(f \star g)(p) =  \int_\Heis f(pq^{-1})g(q)d\mu(q)=\int_\Heis f(q)g(q^{-1}p)d\mu(q).
\end{equation*}
Therefore it is readily checked that
\begin{equation*}
 |f \star g|(p)w(p) 
 \leq 
 \int |f(q)| \, |g(q^{-1}p)| w(p) \, d\mu(q) \, .
\end{equation*}
Next recast relation~\eqref{d1} as $w(p)\le v(q) w(q^{-1}p)$. We obtain
\begin{equation*}
|f \star g|(p)w(p) 
\leq \int|f(q)| \, |g(q^{-1}p)| v(q)w(q^{-1}p) \, d\mu(q)
=
[(|f|v)\star(|g|w)](p).
\end{equation*}
Our claim \eqref{d2} then follows from Young's inequality for functions on locally compact groups (see \cite[Lemma 1.4]{BCD2}).
\end{proof}
We now introduce a pseudo-norm on $\Heis$ which avoids some singularities near the origin $e$. We label its main properties in the next lemma.

\begin{lemma}
    Recall that 
the Carnot--Carath\'eodory distance is defined in  \eqref{cc-metric}. We define a pseudo-norm $|\cdot|_*$ on $\Heis$ by
\begin{equation}\label{pseudonorm}
 |p |_*= \sqrt{1+d_{cc}(e, p)^2},\quad \text{for $p\in\Heis$.}
 \end{equation}
Then  $|\cdot|_*$ satisfies the following properties: 
\begin{equation}\label{rela-b}
    |e|_*=1,\quad |p^{-1}|_*=|p|_*\qquad \text{for  $p\in\Heis$,}
\end{equation} as well as the triangular inequality
\begin{equation}\label{tri-ineq}
    |pq|_*\leq |p|_*+|q|_*.
\end{equation}
\end{lemma}
\begin{proof}
Relation \eqref{rela-b} is trivial. In order to prove \eqref{tri-ineq}, let us invoke the triangular inequality for $d_{cc}$ and the left invariance property \eqref{left-invariant property}. This enables to write
\begin{equation*}
    |pq|^2_* = 1+d_{cc}(e, pq)^2
    \leq 1+(d_{cc}(e,p)+d_{cc}(p,pq))^2=1+(d_{cc}(e,p)+d_{cc}(e,q))^2.
    \end{equation*} 
    Expanding the right hand side above and resorting to the definition \eqref{pseudonorm} of $|\cdot|_*$, it is thus easily seen that 
    \begin{eqnarray*}
    |pq|^2_*&\leq& 
    1+d_{cc}(e,p)^2+d_{cc}(e,q)^2+2\,d_{cc}(e,p)d_{cc}(e,q) \\
    &\le&
    (|p|_*+|q|_*)^2,
    \end{eqnarray*}
    which finishes the proof of \eqref{tri-ineq}.
    \end{proof}

The pseudo-norm $|\cdot|_*$ is the main building block in order to introduce the exponential weights considered in this paper. We label the definition of those weights for further use.
\begin{definition}\label{weight}
Recall that $s$ is a parameter such that $s\ge 1$.
  Let $\eta$ be a real number in $(0,1/s)$ and  $\nu\in\mathbb{R}$. We define the weight $ w_\nu:\Heis\to\R_+$ by
    \[
    w_\nu(q)=e^{-\nu|q|_*^\eta},
    \]where we recall that $|q|_*$ is defined in \eqref{pseudonorm}. We also denote by $L_\nu^p$ the corresponding space $L^p(\Heis,w^p_\nu(q)\, d\mu(q))$, and  by $(\cdot,\cdot)_\nu$ the inner product in $L^2_\nu$.
\end{definition}

\begin{remark}
Our notation for weighted $L^p$-spaces differs slightly from the one introduced in \cite{MW}. Namely the spaces in \cite{MW} are of the form $L^p(\Heis,w_\nu(q)\, d\mu(q))$, in contrast with our version where we also raise $w_\nu$ to a power $p$. While the main properties of the weighted spaces are not affected by this difference, our spaces are such that the weighted $L^\infty$-space can be seen as a limit of weighted $L^p$-spaces. This will be useful for the estimates involving the noise $\dW^{\zeta, \alpha}$ in \eqref{eq:pam}. It also simplifies the paraproduct computations in Section \ref{sec-paraprod} below.
\end{remark}
    We now turn to a basic lemma stating that our exponential weights are moderate.
    \begin{lemma}\label{lemma-moderate}
For $\nu\in\R$, recall that the weight $w_\nu$ is introduced in Definition \ref{weight}, for a parameter $0<\eta<1/s$. Then for all $\nu\geq 0$, $w_\nu$ is $w_{-\nu}$-moderate.
    \end{lemma}
    \begin{proof}
    
Our objective is to show that the pair $w=w_\nu$, $v=w_{-\nu}$ satisfies relation \eqref{d1}. Namely, for $p,q\in\Heis$ we wish to prove that \begin{equation}\label{obj}
    \frac{w_\nu(pq)}{w_\nu(q)}\leq w_{-\nu}(p).
\end{equation}To this aim, we first apply the mere definition of $w_\nu$ in order to write\[\frac{w_\nu(pq)}{w_\nu(q)}=\text{exp}(\nu(|q|_*^\eta-|pq|_*^\eta)).\]
Next we invoke \eqref{rela-b} and \eqref{tri-ineq}, which yields $|q|_*\leq|pq|_*+|p|_*$. Since we are considering a positive parameter $\nu$ and $\eta<1/s \le 1$, we obtain
    \begin{align}\nonumber
    \frac{w_\nu(pq)}{w_\nu(q)}&\leq \text{exp}(\nu(|pq|_*+|p|_*)^\eta-\nu |pq|_*^\eta)\\\label{ineq-f}&\leq \text{exp}(\nu|p|_*^\eta)=w_{-\nu}(p) \, ,
\end{align}where we resort to the elementary relation $|a+b|^\eta-|a|^\eta\leq|b|^\eta$ for the second inequality. Note that \eqref{ineq-f} is exactly our claim \eqref{obj}, which finishes the proof.
    \end{proof}

  It is known (see \cite[p.\ 23]{R}) that there exists a function $\varphi\in \mathcal{G}^s(\mathbb{R}^n)$  with compact support such that  $\varphi\equiv 1$ on the unit ball centered at the origin (with respect to the Euclidean distance $|\cdot|_2$). We recall this important result in the lemma below.
\begin{lemma} \label{bump-func}  Let $s> 1$ and recall that the Gevrey space $\cg^s(\R^n)$ on $\R^n$ is defined by \eqref{g-est}. Then there exists a  compactly supported function $\varphi\in \mathcal{G}^s(\mathbb{R}^n)$  such that  $\varphi\equiv 1$ on the unit ball $B_{1}(0)=\{x\in\mathbb{R}^n:|x|_2<1\}$.
\end{lemma} \noindent {\it Proof.} It is a standard fact that the function   $\psi(x)=\exp(-x^{-\kappa})\,\mathbbm{1}_{(x>0)}$, where  $\kappa=1/(s-1)$, lies in $\mathcal{G}^s(\mathbb{R})$. Since Gevrey classes on $\mathbb{R}^n$ are stable under multiplication, the compactly supported function 
\[\Psi(x_1,\ldots,x_n)=\prod_{j=1}^n\psi(1+x_j)\psi(1-x_j)\] lies in $\mathcal{G}^s(\mathbb{R}^n)$. Then we take $\varphi$ to be the convolution of the function  $\Psi/\lVert\Psi\rVert_{L^1}$ and the indicator function of $B_{2n}(0)$ (see e.g. \cite[p.\ 22]{R}), so that $\varphi\in \mathcal{G}^s(\mathbb{R}^n)$  has compact support and   $\varphi\equiv 1$ on  $B_{1}(0)$.\qed
  
  In the next lemma we extend Lemma \ref{bump-func} to the Heisenberg group. That is, we prove  the existence of a  Gevrey function $\phi\in\cg^s(\Heis)$ with bounded support, such that $\phi\equiv 1$ on the unit ball of $\Heis$.
  
  \begin{lemma}\label{lemma-phi}
  Consider a parameter $s> 1$ and the Gevrey space $\cg^s(\fHeis)$ on the Heisenberg group $\fHeis$, as introduced in Definition \ref{def-diff}. Recall that the distance $\widehat{d}_0$ is given by \eqref{dist0}.  Then there exists a function $\tilde{\phi} \in \cg^s(\fHeis)$ with bounded support such that $\tilde{\phi}\equiv 1$ on the unit ball $B_1:=\{\hat{w}\in\fHeis : \widehat{d}_0(\hat{w})<1\}$.
  \end{lemma} 
  \begin{proof}
Consider the function $\varphi$ on $\R^n$ constructed in Lemma \ref{bump-func}. For $(x,\kappa,\lambda)\in[0,\infty)^n\times\Z^n\times\R$, we set 
\begin{equation}\label{func-f}
f(x, \kappa, \lambda) = \varphi(x)\,\mathbbm{1}_{\{\kappa=0\}},
\quad\text{and}\quad
\tilde{\phi}(m,\ell,\lambda) = \Theta_f(m,\ell,\lambda),
\end{equation} 
where we recall the notation $\Theta_f$ introduced in \eqref{def-theta-f}.
Our aim is to show that $\tilde{\phi}=\Theta_f$ as defined above fulfills the conditions of our lemma.

\begin{enumerate}[wide, labelwidth=!, labelindent=0pt, label= \textbf{(\roman*)}]
\setlength\itemsep{.05in}

\item
Clearly the function $f$ satisfies both conditions (a) and (b) in Proposition \ref{prop 2.13}. Therefore if we define $\Theta_f$ as in  \eqref{def-theta-f}, we directly get that $\Theta_f \in \cg^s(\Heis)$.

\item
Observe that if $f$ is given as in \eqref{func-f}, then $\Theta_f$ can be written as   
\begin{equation}\label{func-thetaf}
\Theta_f (m,\ell,\lambda)=\varphi(|\lambda|R(m,\ell))\mathbbm{1}_{\{m=\ell\}}.
\end{equation}
Thanks to the relation above, one can show that the function $\Theta_f$ is identically $1$ on the ball $B_1$ in $\Heis$. To this aim, resorting to the expression \eqref{dist0}, it is readily checked that if we wish to have $\widehat{d}_0(\hat{w})<1$ then we need $m=\ell$. Hence one can describe $B_1$ as 
  \begin{equation}\label{eq-B-1}
  B_1=\{(m,\ell,\lambda)\in\fHeis :\, m=\ell,\, |\lambda|R(m,m)=|\lambda|(2|m|+n)<1 \}.
  \end{equation}
From \eqref{eq-B-1} and the fact that  $\lvert|\lambda|R(m,m)\rvert_2\leq \lvert|\lambda|R(m,m)\rvert_1$,   we have that
\[
\tilde{\phi}\lp m,\ell,\la  \rp
=
\Theta_f (m,\ell,\lambda)
=\vp\lp |\lambda|R(m,m) \rp
\equiv 1\, ,\ \quad \text{on}\ \, B_1.
\]
This proves that $\tilde{\phi}\equiv 1$ on the unit ball $B_1$.

\item
To see that the $\Theta_f$ has bounded support, let $R>0$ be such that $\varphi(x)=0$ whenever $|x|\geq R$. Then $\Theta_f(m,\ell,\lambda) = 0 $ whenever $\vert|\lambda|R(m,\ell)\vert_2\geq R$. Due to the fact that $\lvert|\lambda|R(m,\ell)\rvert_1\leq n^{1/2} \lvert|\lambda|R(m,\ell)\rvert_2$, we see that 
\begin{equation*}
|\lambda|(|m+\ell|+n) \geq n^{1/2} R
\quad\Longrightarrow\quad
\big\lvert|\lambda|R(m,\ell)\big\rvert_2 > R \, .
\end{equation*}
Hence it is readily seen that
\[
\Theta_f(m,\ell,\lambda) = 0 \quad \text{whenever } \quad |\lambda|(|m+\ell|+n) \geq n^{1/2} R.
\]
Noting that the support of $\Theta_f$ is contained in the diagonal $\{m=\ell\}$, we then conclude that 
\begin{equation}\label{suppf}
  \text{Supp}\, \Theta_f\subset \{(m,\ell,\lambda)\,:\,m=\ell,\,|\lambda|(2|m|+n) < n^{1/2} R\}.
  \end{equation}

At last  it is clear from \eqref{suppf}   that $\text{Supp}\, \Theta_f$ is bounded with respect to the metric $\widehat{d}$.
Specifically, one can prove that if $\hat{w}, \hat{w}'$ lye in  $\text{Supp}\, \Theta_f$, then $\widehat{d}(\hat{w},\hat{w}')\le M$ for a large enough constant $M$. Namely for all  $\hat{w}=(m,m,\lambda)$ and  $\hat{w}'=(m',m',\lambda')$ such that $\hat{w},\hat{w}'\in\text{Supp}\, \Theta_f$, according to the definition~\eqref{dist} of $\widehat{d}$ we have
\begin{align*}
\widehat{d}(\hat{w},\hat{w}')&=|\lambda(2 m)-\lambda'(2m')|_1+0+n|\lambda-\lambda'|\\&\leq|\lambda|(2 |m|)+|\lambda'|(2| m'|)+n|\lambda|+n|\lambda'|<2\,n^{1/2} R\equiv M,
\end{align*}
where the last inequality follows from \eqref{suppf}. We then obtain that the diameter of $\text{Supp}\, \Theta_f$ is bounded by $M$.
\end{enumerate}
Gathering the items (i)-(ii)-(iii) above, our claims are now proved.
\end{proof}

Before stating our version of Bernstein's lemma, let us address some scaling properties that shall be used later.
\begin{lemma}\label{lemma-g-tau}
Let $\tilde{\phi}\in \mathcal{G}^s(\fHeis)$ as be given in Lemma \ref{lemma-phi}. For any $\tau>0$, let $\tilde{\phi}_\tau(m,\ell,\lambda)=\tilde{\phi}(m, \ell, \lambda/\tau)$, $(m,\ell,\lambda)\in \fHeis$. Then $\tilde{\phi}_\tau$ is in $\mathcal{G}^s(\fHeis)$ with bounded support and $\tilde{\phi}_\tau\equiv1$ on $B_\tau\cap D:=\{\hat{w}\in\fHeis : \widehat{d}_0(\hat{w})<\tau, \text{ and } m=\ell\}$. Moreover, let
$g_\tau$ be the inverse Fourier transform of $\tilde{\phi}_\tau$, i.e. $\hat{g}_\tau=\tilde{\phi}_\tau$. Then for any $(x,y,z)\in \Heis$, we have
\begin{equation}\label{eq-scaling}
    g_\tau(x,y,z)=\tau^{n+1} g_1(\tau^{1/2}x,\tau^{1/2}y,\tau z)=\tau^{n+1} g_1(\delta_{\tau^{1/2}\,}q),
\end{equation}
where the dilation $\delta_{\tau^{1/2}}$ is defined by \eqref{eq-dilation}.
\end{lemma}
\begin{proof}
First from its definition and \eqref{func-thetaf} we know that 
\begin{equation}\label{eq-phi-tau}
\tilde{\phi}_\tau(m,\ell,\lambda)=\varphi\left(\frac{|\lambda|}{\tau}R(m,\ell)\right)\mathbbm{1}_{\{m=\ell\}}.
\end{equation}
Similar to what we have proved in  Lemma \ref{lemma-phi}, it is then easily seen that $\tilde{\phi}_\tau\equiv 1$ on the set $B_\tau$.
To see \eqref{eq-scaling}, resorting to \eqref{eq-FT-int-inv} we have 
\begin{align}\label{eq-gtau}
g_\tau(x,y,z)&=\frac{2^{n-1}}{\pi^{n+1}}\sum_{m,\ell\in\N^{n}}J_{m,\ell, \tau}(x,y,z),\end{align}
where we have set
\[J_{m,\ell,\tau}(x,y,z)=\int_\R e^{i\lambda z} K_{m,\ell,\lambda}(x,y,z)  \tilde{\phi}_\tau(m,\ell,\lambda)|\lambda|^n d\lambda.\]
Next writing the definition  \eqref{eq-FT-K} of $K_{m,\ell,\lambda}$ and having the expression \eqref{a2} of $\Phi_k^\lambda$ in mind, we obtain
\begin{multline*}
J_{m,\ell,\tau}(x,y,z) 
=   
\int_\R e^{i\lambda z} \bigg( \int_{\R^n} e^{2i\lambda\langle y,\,\xi\rangle}|\lambda|^{n/2}\Phi_m(|\lambda|^{1/2}(x+\xi))\\
\cdot\Phi_\ell^\lambda(|\lambda|^{1/2}(-x+\xi))d\xi \, \bigg)\tilde{\phi}_\tau(m,\ell,\lambda)|\lambda|^n d\lambda.
\end{multline*}
Moreover, it is easily seen from \eqref{eq-phi-tau} that $\tilde{\phi}_\tau(m,\ell,\lambda)=\tilde{\phi}_1(m,\ell,\lambda/\tau)$ for all $\tau>0$. Therefore one can apply the rescaling $\lambda\to\tau\lambda$ and $\xi\to \tau^{-1/2}\xi$ in order to get
\begin{multline*}
J_{m,\ell,\tau}(x,y,z)  =   \int_\R e^{i\tau\lambda z} \bigg( \int_{\R^n} e^{2i\lambda\langle \tau^{1/2}y,\,\xi\rangle}|\lambda|^{n/2}\Phi_m(|\lambda|^{1/2}(\tau^{1/2}x+\xi)) \\
\cdot \Phi_\ell^\lambda(|\lambda|^{1/2}(-\tau^{1/2}x+\xi))d\xi \, \bigg)\tilde{\phi}_1(m,\ell,\lambda)|\lambda|^n \tau^{n+1} d\lambda.
\end{multline*}
Invoking the definition  \eqref{eq-FT-K} of the kernels $K_{m,\ell,\lambda}$ again, the above expression for $J_{m,\ell,\tau}$ can be recast as
\begin{align*}
J_{m,\ell,\tau}(x,y,z) &=
\tau^{n+1} \int_\R e^{i\lambda (\tau z)} K_{m,\ell,\lambda}(\tau^{1/2}x,\tau^{1/2}y,\tau z)  \tilde{\phi}_1(m,\ell,\lambda)|\lambda|^n d\lambda\\&=\tau^{n+1}J_{m,\ell,1}(\tau^{1/2}x,\tau^{1/2}y,\tau z).
\end{align*}
Plugging this identity into  \eqref{eq-gtau}, this easily implies \eqref{eq-scaling}. 
\end{proof}

As a last preliminary step before we prove Bernstein's lemma, we state a scaling property for functions of the form $\Delta^k g_\tau$ for any positive integer $k$.

\begin{lemma}\label{lemma-g-tau-k-bd} For $\tau>0$, consider the function $g_\tau$ given by \eqref{eq-scaling}. For $k\in\N$, let us also recall that $\Delta^k$ is defined in \eqref{eq-FT-delN}. We denote by $g^{(k)}_\tau$ the dilated function of $\Delta^k g$, namely, for $q=(x,y,z)\in\Heis$, 
\begin{align}\label{eq-g-k}
g_\tau^{(k)}(q)=(\Delta^k g_1)_\tau(q)=\tau^{n+1}(\Delta^k  g_1)(\delta_{\tau^{1/2}}\,q),
\end{align} 
where the dilations $\delta_{\tau^{1/2}}$ are defined by \eqref{eq-dilation}. Also recall that the spaces $L^p_\nu$ have been introduced in Definition \ref{weight}. 
Then for any $\nu,  \gamma>0$,  there exists a constant $C>0$ such that
\begin{equation}\label{eq-bd-gk}
\lVert g_\tau^{(k)} \rVert_{L^\gamma_{-\nu }}\leq C\tau^{(n+1)(1-\frac{1}{\gamma})}.
\end{equation}
\end{lemma}

\begin{proof}[Proof of Lemma~\ref{lemma-g-tau-k-bd}] We first prove that $g^{(k)}$ in \eqref{eq-g-k} has exponential decay, which is enough to see that $g^{(k)}_\tau\in L^\gamma_{-\nu }$. As a preliminary step in this direction, let us show that   $\widehat{ g^{(k)}_1}$ sits in $\mathcal{G}^s(\fHeis)\cap\mathcal{S}(\fHeis)$ and has bounded support.
 To this aim, note that owing to \eqref{eq-FT-delN} we have
 \begin{equation}\label{eq-F-gk}
 \widehat{ g^{(k)}_1}(m,\ell,\lambda)=\widehat{\Delta^k_{\mathcal{H}\,} g_1}(m,\ell,\lambda)= \Big(-4|\lambda|(2|m|+n)\Big)^k\tilde{\phi}(m,\ell,\lambda).
\end{equation}
Hence $\widehat{g^{(k)}_1}$ has bounded support as $\tilde{\phi}$ does. In addition, recall from \eqref{func-f} that $\tilde{\phi}=\Theta_f$ with $f=\varphi\cdot\mathbbm{1}_{\{\kappa=0\}}$. Some elementary algebra using the definition \eqref{def-theta-f} of $\Theta_f$ shows that 
\begin{equation}
\label{eq-theta-f}
\Big(-4|\lambda|(2|m|+n)\Big)^k\tilde{\phi}(m,\ell,\lambda)=\Theta_{f_0}(m,\ell,\lambda),
\end{equation}
where $f_0(x,\kappa,\lambda)=(-4\sum x_i)^k\varphi(x)\mathbbm{1}_{\{\kappa=0\}}$. 
In addition, since $h\equiv\varphi\cdot\mathbbm{1}_{\{\kappa=0\}}$ verifies the assumptions of Proposition \ref{prop 2.13} and $f_0$ is defined as a polynomial in $x$ multiplied by $h$, it is readily checked that $f_0$ also  verifies the assumptions of Proposition \ref{prop 2.13}. Putting together~\eqref{eq-F-gk} and \eqref{eq-theta-f} and applying Proposition \ref{prop 2.13}, we then obtain that 
\[
 \widehat{g^{(k)}_1}\in \mathcal{G}^s(\fHeis)\cap\mathcal{S}(\fHeis).
\]
\noindent Therefore a straightforward application of   Proposition \ref{exp-decay-est} yields  that  $g^{(k)}_1$ is of exponential decay. Otherwise stated,   there exist $C>0$ and $\varepsilon>0$ such that
\begin{equation}\label{eq-g1k-bd}
|g_1^{(k)}(q)|\leq Ce^{-\varepsilon\lvert q \rvert_h^{1/s}}\ \, \text{for all $q=(x,y,z)\in\Heis$}.
\end{equation}
\indent Let us now turn to the scaling inequality \eqref{eq-bd-gk}. According to our definition \eqref{eq-g-k} of $g^{(k)}$, we have for  $q=(x,y, z)\in\Heis$, 
\[
g^{(k)}_\tau(q)=\tau^{n+1}g^{(k)}_1(\delta_{\tau^{1/2}}\, q).
\]
With our Definition \eqref{weight} of inner products in weighted spaces and using a scaling
$q=(x,y, z)\to \delta_{\tau^{1/2}\,} q =  (\tau^{-1/2}x, \tau^{-1/2} y, \tau^{-1} z)$, we get 
\begin{align*}
\lVert g_\tau^{(k)} \rVert_{L^\gamma_{-\nu }}
&=\left(\int_\Heis\tau^{\gamma (n+1)}\,|g_1^{(k)}|^\gamma(\delta_{\tau^{1/2}\,} q)\, e^{{\nu \gamma}|q|_*^\eta}\, d\mu(q)\right)^{\frac{1}{\gamma}}\\
&=\tau^{(n+1)(1-\frac{1}{\gamma})}\left(\int_\Heis|g_1^{(k)}|^\gamma(q)\, e^{{\nu \gamma}|(\delta_{\tau^{-1/2}\,} q)|_*^\eta}\, d\mu(q)\right)^{\frac{1}{\gamma}}.
\end{align*} 
Plugging \eqref{eq-g1k-bd} into the above inequality we thus obtain 
\begin{equation}
\label{est-d}
\lVert g_\tau^{(k)} \rVert_{L^\gamma_{-\nu}}
\le C \tau^{(n+1)(1-\frac{1}{\gamma})}\left(\int_\Heis e^{{\nu \gamma}|(\delta_{\tau^{-1/2}\,} q)|_*^\eta -\varepsilon\lvert q \rvert_h^{1/s}}\, d\mu(q)\right)^{\frac{1}{\gamma}}.
\end{equation}

We now resort to the definition \eqref{pseudonorm} of $|\cdot|_*$ and relation \eqref{eq-equiv-cc-h}. We also take into account that $\tau\ge1$. We discover that there exists  $c>0$ such that
\begin{align*}
|\delta_{\tau^{-1/2}}\,q|_*&=\sqrt{1+d_{cc}(e,\, \delta_{\tau^{-1/2}}\,q)}\le c(1+|q|_h).
\end{align*} 
Plugging this relation into \eqref{est-d} and recalling that in Definition \ref{weight} we have chosen $0<\eta<\frac1s$, we end up with
\begin{equation}\label{eq-gk-bd1}
\lVert g_\tau^{(k)} \rVert_{L^\gamma_{-\nu }}
\le 
C \tau^{(n+1)(1-\frac{1}{\gamma})}\left(\int_\Heis e^{{\nu \gamma}c(1+|q|_h)^{1/s} -\varepsilon\lvert q \rvert_h^{1/s}}\, d\mu(q)\right)^{\frac{1}{\gamma}}.
\end{equation}

We now choose  $\nu_0={\varepsilon}/{2c\gamma}>0$  and  $0<\nu<\nu_0$ in \eqref{eq-gk-bd1}, so that $\nu\gamma c<\varepsilon/2$. We obtain
\begin{align*}
\int_\Heis e^{{\nu \gamma}c(1+|q|_h)^{1/s} -\varepsilon\lvert q \rvert_h^{1/s}}\, d\mu(q)\leq\int_{|q|_h\le 1}e^{2\varepsilon}d\mu(q)+\int_{|q|_h\ge 1} e^{-\frac{\varepsilon}2\lvert q \rvert_h^{1/s}}d\mu(q)\le C_1
\end{align*} 
for some constant $C_1>0$. 
Plug the above bound into \eqref{eq-gk-bd1} we then obtain the desired estimate \eqref{eq-bd-gk}.
\end{proof}

The preliminary results we have proved so far allow us to state the main result of this section. That is, we are ready to prove the aforementioned Bernstein lemma.
\begin{proposition}\label{Bernstein} Let $f\in \mathcal{S}(\Heis)$, and consider  parameters $\nu_0>0$,   $k\in \mathbb{N}$  and $\beta\in [1,\infty]$. Then there exists $C=C_{\nu_0,\,\beta}$ such that for every $\nu\leq \nu_0$, $\alpha\geq\beta$ and $\tau\geq 1$, if we assume 
\begin{equation}
\label{supp-Fourier}
\emph{Supp}\, \hat{f} \subset  \{(m,\ell,\lambda):|\lambda|(2|m|+n)< \tau\},\end{equation}
then we have
\begin{equation}
\label{est-h}
\lVert \Delta^k f\rVert_{L^\alpha_\nu}\leq C\tau^{k+(n+1)(\frac{1}{\beta}-\frac{1}{\alpha})}\lVert f\rVert_{L^\beta_{\nu }},
\end{equation}
where we recall that the spaces $L^\alpha_\nu$ are introduced in definition \ref{weight}.
\end{proposition}

\begin{remark}
Condition \eqref{supp-Fourier} is interpreted as $\mathrm{Supp}\, \hat{f}$ being contained in a ball in $\fHeis$. It is compatible with \cite[Definition 3.1]{BBG}.
\end{remark}

\begin{proof}[Proof of Proposition \ref{Bernstein}] Recall the definition \eqref{def-dot} of our dot product $\tilde{\phi}_\tau\cdot f$. Also recall from \eqref{eq-phi-tau} that $\tilde{\phi}_\tau$ is supported on the diagonal $\{m=\ell\}$. Hence we have 
\begin{equation}\label{phi-dot-f}
(\tilde{\phi}_\tau\cdot \hat{f})(m,\ell,\lambda)=\sum_{j\in\mathbb{N}^n} \tilde{\phi}_\tau(m,j,\lambda ) \hat{f}(j,\ell,\lambda) =  \tilde{\phi}_\tau(m,m,\lambda ) \hat{f}(m,\ell,\lambda).
\end{equation} 
One of the crucial points in our proof is thus the following: since we have assumed \eqref{supp-Fourier} and since we have obtained the relation $\tilde{\phi}_\tau\equiv 1$ on $B_\tau\cap D$ in Lemma~\ref{lemma-g-tau}, it is readily checked that $\tilde{\phi}_\tau(m,m,\lambda ) \hat{f}(m,\ell,\lambda)= \hat{f}(m,\ell,\lambda)$ for all $(m,\ell,\lambda)$. Therefore, \eqref{phi-dot-f} can also be written as 
\begin{equation}\label{eq-phi-dot-f}
    (\tilde{\phi}_\tau\cdot \hat{f})(m,\ell,\lambda)=\hat{f}(m,\ell,\lambda).
\end{equation}
Taking inverse Fourier transforms on both sides of \eqref{eq-phi-dot-f}, resorting to \eqref{eq-conv} and recalling that we have set $\hat{g_\tau}=\tilde{\phi}_\tau$ in Lemma \ref{lemma-g-tau}, this yields
\begin{equation}\label{eq-g-star-f}
g_\tau\star f= f.    
\end{equation}
\indent Let us  now evaluate $\Delta^k g_\tau$ for $g_\tau$ in the left hand side of \eqref{eq-g-star-f}. Resorting to \eqref{eq-g-k}, some easy scaling considerations entail that 
\begin{equation}
\label{eq-Delta-k}
\Delta^k g_\tau = \tau^{n+1}\Delta^k (g_1\circ \delta_{\tau^{1/2}})=  \tau^{n+1+k}(\Delta^k g_1)\circ \delta_{\tau^{1/2}}.
\end{equation} Recalling again  that $g^{(k)}_\tau$ is defined by \eqref{eq-g-k}, we have obtained \begin{equation}\label{eq-del-g-tau}
    \Delta^kg_\tau=\tau^k g^{(k)}_\tau.
\end{equation}
Hence applying $\Delta^k$ on both sides of \eqref{eq-g-star-f} and combining with \eqref{eq-del-g-tau} we get 
\begin{align}\label{eq-k}
\Delta^k f = \Delta^k g_\tau\star f= \tau^k g_\tau^{(k)}\star f.
\end{align}
We are now in a position to apply the weighted Young's inequality in Proposition \ref{young-ineq}  to relation~\eqref{eq-k}. Namely, for  $\alpha\geq \beta$ and $\gamma$ such that $\gamma^{-1}+\beta^{-1}=1+\alpha^{-1}$,  by invoking the fact that $w_{\nu}$ is $w_{-\nu}$-moderate (see Lemma \ref{lemma-moderate}),   we obtain
\begin{eqnarray*}
\lVert\Delta^k f\rVert_{L^\alpha_\nu}
&=&\tau^{k}\,\lVert (g^{(k)}_\tau\star f) w_\nu\rVert_{L^\alpha}\\[0.1cm]
&\leq& 
\tau^{k}\lVert g_\tau^{(k)}\, w_{-\nu}\rVert_{L^\gamma}\lVert f w_\nu\rVert_{L^\beta}
= \tau^{k}\,\lVert g_\tau^{(k)} \rVert_{L^\gamma_{-\nu }}\lVert f \rVert_{L^\beta_{\nu}}.
\end{eqnarray*} 
Our claim \eqref{est-h} then follows from Lemma \ref{lemma-g-tau-k-bd} and the relation $\beta^{-1}-\al^{-1}=1-\ga^{-1}$.
\end{proof}

\begin{remark}\label{rmk-embedding-rho}
The results in Proposition \ref{Bernstein} also holds for the weight function $\rho_b(q)=C(1+|q|_*^b)$ for some $C, b>0$. Namely for any  $k\in \mathbb{N}$, $\alpha\geq\beta$ and $\tau\geq 1$, if $f\in \mathcal{S}(\Heis)$ satisfies \eqref{supp-Fourier}, then there exists $C=C_{b,\,\beta}$ such that 
\begin{equation}
\label{est-h-rho-b}
\lVert \Delta^k f\rVert_{L^\alpha_{\rho_b}}\leq C\tau^{k+(n+1)(\frac{1}{\beta}-\frac{1}{\alpha})}\lVert f\rVert_{L^\beta_{\rho_b }}.
\end{equation}
The proof follows the lines of the proof for Proposition \ref{Bernstein}. One only need to use the facts that 
\[
\rho_b(p)=1+|p|_*^b\le 1+(|q|_*+|q^{-1}p|_*)^b\le C_b(1+|q|_*^b)(1+|q^{-1}p|_*^b)=C_b\rho_b(q)\rho_b(q^{-1}p)
\]
and 
\begin{align}\label{eq-g-tau-rho-b}
\lVert g_\tau^{(k)} \rVert_{L^\gamma_{\rho_b }}\leq C\tau^{(n+1)(1-\frac{1}{\gamma})}.
\end{align}
In particular, the above inequality \eqref{eq-g-tau-rho-b} follows from similar argument as in \eqref{eq-bd-gk}.
\end{remark}

\section{Weighted Besov spaces}\label{sec-basic}
In this section we introduce a proper notion of weighted Besov space on $\Heis$ and derive corresponding smoothing effects for the heat flow. 
\subsection{Basic definitions} The blocks in Littlewood-Paley analysis are based on partitions of unity. Let us first construct those partitions in $\fHeis$. To this aim we start with a partition of unity in $\R^n$. As in \cite{BCD2}, this is defined in  the following way.
\begin{proposition}\label{prop-a}
There exists a pair $\chi, \tilde{\chi}\in\mathcal{G}^s(\R^n)$ such that: 
\begin{enumerate}[wide, labelwidth=!, labelindent=0pt, label= \emph{(\roman*)}]
\setlength\itemsep{.05in}

\item  For $x\in\R^n$, we have $0
    \leq\chi(x), \tilde{\chi}(x)\leq 1$.
  
    \item The supports of $\chi,\tilde{\chi}$ satisfy 
    \begin{equation}\label{eq-chi-supp}
\textnormal{Supp}\,\tilde{\chi}\subset B_{4/3}(0):=\{x\in\R^n:|x|<4/3\}, \quad
\textnormal{Supp}\,\chi\subset \mathcal{C}^*:=\{x\in\R^n:3/4\leq|x|<8/3\}.
\end{equation}
\item We have
\[
\sum_{k=-1}^\infty\chi_k(x) =1,\quad \text{where} \ \chi_{-1}=\tilde{\chi} \ \text{ and }\ \chi_{k}= \chi( \cdot / 2^k) \ \text{ for }\ k\geq 0.
\]
\end{enumerate}
\end{proposition}
In Proposition \ref{prop 2.13}, we used the existence of a function in $\cg^s(\R^n)$ in order to produce an example of function in $\cg^s(\fHeis)$. We now proceed in the same way for the partition of unity on the diagonal of $\fHeis$. Our result is summarized in the lemma below.

\begin{lemma}\label{Definition phi_k}
    Let $\{\chi_k: k\geq -1\}$ be the partition of unity in $\R^n$ produced in Proposition \ref{prop-a}. We set $f_k=\chi_k\,\mathbbm{1}_{\{\kappa=0\}}$, and \begin{equation}
        \label{eqn-b}\tilde{\phi}_k(m,\ell,\lambda)=\Theta_{f_k}(m,\ell,\lambda)=\chi_k(|\lambda|R(m,\ell))\,\mathbbm{1}_{\{m=\ell\}},
    \end{equation}where we recall the notation $R$ and $\Theta_{f_k}$ from \eqref{def-theta-f}. Then we have 
\begin{enumerate}[wide, labelwidth=!, labelindent=0pt, label= \emph{(\roman*)}]
\setlength\itemsep{.05in}

    \item Every $\tilde{\phi}_k$ is an element of $\cg^s(\fHeis)$.
    \item The following relation holds true:
    \begin{equation}\label{eqn-d}    \sum_{k=-1}^\infty\tilde{\phi}_k(m,\ell,\lambda)=\mathbbm{1}_{\{m=\ell\}}.
    \end{equation}
    \end{enumerate}
\end{lemma}
\begin{proof}
    Item (i) is an immediate consequence of Proposition \ref{prop 2.13}. Item (ii) is derived trivially from  Proposition \ref{prop-a}.
\end{proof}

With the partition of unity in hand we can now define the Littlewood-Paley blocks similarly to what is done in the flat case. This is formalized in the lemma below about Schwartz type functions.

\begin{lemma}
    Consider $f \in \mathcal{S} (\Heis)$, where we recall that $\mathcal{S} (\Heis)$ is introduced in Definition~\ref{def:schwartz-and-multiplication}. For $k\geq -1$, we
   define the block $\sigma_k f$ of order $k$ by its Fourier transform as 
\begin{equation}\label{eqn-c}
\widehat{\sigma_{k} f}  = \tilde{\phi}_k\cdot\hat{f},
\end{equation}
where $\tilde{\phi}_k$ is defined in lemma \ref{Definition phi_k}.
We also define $S_k f$ as 
\[
S_k f=\sum_{i<k}\sigma_{i}f.
\]
Then we have the pointwise convergence of  $S_kf$  to $f$. 
\end{lemma}

\begin{proof}
    Owing to \eqref{eqn-b}, every $\tilde{\phi}_k$ is supported on the diagonal $\{m=\ell\}$ in $\fHeis$. Therefore we have 
    \begin{align}\label{d3}
\left(\tilde{\phi}_k\cdot\hat{f}\right)(m,\ell,\lambda) 
&= \sum_{j\in\N^n}\tilde{\phi}_k(m,j,\lambda)\hat{f}(j,\ell,\lambda) \notag\\ 
&= \tilde{\phi}_k(m,m,\lambda)\hat{f}(m,\ell,\lambda).
\end{align}
Hence according to our definition \eqref{eqn-c}, we get 
\begin{align*}
\widehat{S_k f}(m,\ell,\lambda)&= \sum_{i<k}\tilde{\phi}_i(m,m,\lambda)\hat{f}(m,\ell,\lambda).
\end{align*}
Now thanks to Proposition \ref{isomorphism} we have that $\hat{f}\in \mathcal{S} (\fHeis)$. Thus $\sum_{i<k}\tilde{\phi}_i(m,m,\lambda)\hat{f}(m,\ell,\lambda)$ converges to $\hat{f}(m,\ell,\lambda)$, due to \eqref{eqn-d}. Applying Proposition~\ref{isomorphism} again, we thus get the desired convergence of $S_kf$  to $f$.
\end{proof}
We close this section by giving a formal definition of our weighted Besov spaces.

\begin{definition}\label{def-BS} Let $f\in\mathcal{S}(\Heis)$. For $k\geq -1$ consider the function $\sigma_k f$ defined by \eqref{eqn-c}. 
For  any $\gamma\in \R$, $\nu>0$, and $\alpha,\beta\in[1,\infty]$ , we define a norm
\[
\| f \|_{\mathfrak{B}^{\gamma,\, \nu}_{\alpha,\,\beta}}=\left\{\sum_{k=-1}^{\infty} \left( 2^{\gamma k}  \| \sigma_{k} f \|_{L^\alpha_\nu}\right)^\beta \right\}^{\frac{1}{\beta}},
\]
where the spaces $L^\alpha_\nu$ are those introduced in Definition \ref{weight}. 
The weighted Besov space $\mathfrak{B}^{\gamma,\, \nu}_{\alpha,\,\beta}(\Heis)$ is defined as the completion of $\mathcal{S} (\Heis)$ with respect to this norm.
\end{definition}

\subsection{Heat flow and Besov spaces}\label{sec-flow}
In this section we will quantify the smoothing effects of the heat semigroup in the weighted Besov spaces $\mathfrak{B}^{\gamma,\, \nu}_{\alpha,\,\beta}$. As a preliminary result, we first state a space-time  decay property for functions with Fourier transforms supported in an annulus.  
\begin{lemma}\label{lemma-exp-decay}
Let $\varphi\in\cg^s(\R^n)$ be a Gevrey function supported in an annulus $ \mathcal{C} = \{x\in\R^n: R\leq |x|_2\leq R'\}$. As in \eqref{func-f} and \eqref{func-thetaf}, we consider the function  \begin{equation}
\tilde{\phi}(m,\ell,\lambda) = \Theta_f(m,\ell,\lambda), \quad\text{where}\quad f(x, \kappa, \lambda) = \varphi(x)\,\mathbbm{1}_{\{\kappa=0\}}.
\end{equation} Also recall that the heat kernel $p_t$ has been defined by \eqref{eq-Heis-kernel}. We denote $g=g_1$ where $g_\tau$ is introduced in Lemma \ref{lemma-g-tau}. Specifically we have $\hat{g}=\tilde{\phi}$. Then  
 there exist constants $C, c>0$ such that  
  \begin{equation}\label{space-time-exp-decay}
|(g\star p_t)(q)|\leq Ce^{-ct-c\lvert q \rvert_{\mathrm{h}}^{1/s}}\ \, \text{for all}\ \, q\in\Heis.
\end{equation}
\end{lemma}

\noindent {\it Proof.}
    As stated in \eqref{eq-pt-hat}, we have \begin{equation}\label{c}\hat{p}_t(m,\ell,\lambda) =  e^{-4t|\lambda|(2|m|+n)}\,\mathbbm{1}_{\{m=\ell\}}.\end{equation}
     Moreover, it is readily checked that  
    \begin{align}
    \nonumber\big(\widehat{g\star p_t} \big)(m,\ell,\lambda)&=\big(\tilde{\phi}\cdot \hat{p}_t\big)(m,\ell,\lambda)\\ \nonumber
    &= \sum_{j\in\N^n} \tilde{\phi}(m,j,\lambda)\, \hat{p}_t(j,\ell,\lambda)\\ \label{j}
    &= \tilde{\phi}(m,\ell,\lambda)\, \hat{p}_t(m,\ell,\lambda).
    \end{align} Gathering the expressions \eqref{func-f} and \eqref{func-thetaf} for $\tilde{\phi}$ and our formula \eqref{c} for $\hat{p}_t$, we get 
 \[\big(\widehat{g\star p_t} \big)(m,\ell,\lambda) = \varphi(|\lambda|R(m,\ell))\,\mathbbm{1}_{\{m=\ell\}}\,e^{-4 t |\lambda|(2|m|+n)}.\]Recalling again the definition  \eqref{def-theta-f} of $\Theta_f$, one discovers that $\widehat{g\star p_t}$ can be written as 
 \begin{equation}\label{dd}
 \big(\widehat{g\star p_t} \big)(m,\ell,\lambda) = \Theta_f(m,\ell,\lambda), \quad\text{with}\quad
 f(x,\kappa,\lambda) = \varphi(x)\,e^{-4 t (x_1+...+x_n)} \mathbbm{1}_{\{\kappa = 0 \}}.
 \end{equation}
 With 
 \eqref{dd} in hand, we will prove that $\widehat{g\star p_t}$ satisfies  \eqref{space-time-exp-decay} by a mere elaboration of Proposition \ref{prop 2.13}. Specifically we let the reader check that one is just reduced to prove an inequality similar to \eqref{est-gevrey}, weighted by the exponential term $e^{-ct}$. Namely for all $t\geq 0$, $x\in[0,\infty)^n$ and $\alpha\in\N^n$, it is sufficient to prove that  for any compact set $K\subset \fHeis$, there exists a 
 positive constant $C_K$ such  that 
 \begin{equation}\label{exp-est0}
  \sup_K\left( \max\left(|\HDelta^{N}\Theta_f|,\,  |\HDL^{N}\Theta_f|
     \right)\right)\leq C_K^{2N+1}[(2N)!]^se^{-c t }\ \,\text{for all $N\geq 0$.}
 \end{equation}
 Otherwise stated, going back to \eqref{g-est}, we wish to show that the function 
 \begin{equation}\label{gg}
h(x)=\varphi(x)\zeta_t(x),\quad \text{where}\quad\zeta_t(x)=e^{-4t(x_1+\cdots+x_n)},
 \end{equation} is an element of $\cg^s(\R^n).$ Invoking the fact that $\varphi\in\cg^s(\R^n)$ and stability arguments for products in $\cg^s(\R^n)$, we are further reduced to show that  
\begin{equation}\label{exp-est}
\partial^\alpha_x\zeta_t(x)\leq  C^{|\alpha|+1}(\alpha!)^s\,e^{-c t }
\end{equation} for    $x\in\R^n$, $t\geq 0 $ and $\alpha\in\N^n$. In addition, observe that in \eqref{gg}, the function $\zeta_t$ is multiplied by $\varphi$, whose support is included in the annulus $\mathcal{C}$. Therefore we can content ourselves with a proof of \eqref{exp-est} for $x\in \mathcal{C}$ only. Recalling that $ \mathcal{C} = \{x\in\R^n: R\leq |x|_2\leq R'\}$, 
we set $r = R/2$.  With all those preliminary reductions in hand, we now evaluate the left hand side of \eqref{exp-est}. An  elementary computation shows that for $\alpha\in\N^n$ and $\zeta_t$ defined by \eqref{gg}, we have 
\begin{equation}\label{hh}
   \left\lvert\partial^\alpha_x \zeta_t(x)\right\rvert =(4t)^{|\alpha|}e^{-4 t (x_1+\cdots+x_n)}.
\end{equation}
Moreover for $k\in\N$, $t\geq 0$ and an additional parameter $r\geq 0$, it is readily checked that 
\[(4t)^k = r^{-k} (4rt)^k\leq r^{-k}\, k!\, e^{4rt}.\]Plugging this rough estimate into \eqref{hh}, we end up with
\begin{align*}
\left\lvert\partial^\alpha_x \zeta_t(x)\right\rvert\leq r^{-|\alpha|}\,|\alpha|!\,e^{-4 t (x_1+\cdots+x_n-r)}.
\end{align*} Now choose $r=R/2$ and invoke the fact that $|x|_2\geq R$ if $x\in\mathcal{C}$. Also, resort to the inequality $|\alpha|!/\alpha!\leq n^{|\alpha|}$, valid for all $\alpha\in\N^n$. This yields the existence of two constants $ C, c >0 $ such that
\[
\left\lvert\partial^\alpha_x \zeta_t(x)\right\rvert\leq  C^{|\alpha|+1}\alpha!\,e^{-c t }.
\]Summarizing our considerations, we have just proved relation \eqref{exp-est}. Due to our series of reductions \eqref{exp-est0}--\eqref{gg}, this shows relation \eqref{space-time-exp-decay} and finishes our proof.\qed\medskip

We turn to a first quantification of the smoothing effects of the heat semigroup for functions whose support is bounded in Fourier modes. 
\begin{proposition}[Smoothing effect of the heat flow]\label{eff-heat-flow} For two constants $0< R_1<R_2$, consider the following region in $\fHeis$:
\begin{equation}
    \label{aa}
\tilde{\mathcal{R}} = \{(m, \ell,\lambda)\in\fHeis: R_1\leq |\lambda|(2|m|+n)\leq R_2\}.
\end{equation}
Recall that the heat semigroup $P_t=e^{t\Delta/2}$ is defined by \eqref{eq-semigroup}. Let also $\nu_0>0$ be an additional parameter and we pick an arbitrary $\alpha\geq 1$. Then there exist constants $C, c>0$ such that for every $0< \nu\leq\nu_0$, $t\geq 0$ and $\tau\geq 1$, $F \in \mathcal{S}(\Heis)$, if we assume
    \begin{equation}
    \label{bb}    
    \text{Supp}\, \hat{F}\subset \tau\tilde{\mathcal{R}} = \{(m, \ell,\lambda)\in\fHeis: \tau R_1\leq |\lambda|(2|m|+n)\leq \tau R_2\},
    \end{equation}
    then the following inequality holds true:\begin{equation}\label{heat-flow-est}
    \lVert P_t F\rVert_{L^\alpha_\nu}\leq C e^{-c\tau t} \lVert  F\rVert_{L^\alpha_\nu}.
    \end{equation}
\end{proposition}
\begin{proof}
Let $\varphi\in\cg^s(\R^n)$ be a Gevrey function supported inside an annulus in $\R^n$ and such that $\varphi\equiv 1$ on  $  \{x\in\R^n: R_1\leq |x|_1\leq R_2\}$.  As in Lemma \ref{lemma-exp-decay} and \eqref{func-f}--\eqref{func-thetaf}, we consider the function  
\begin{equation}\label{kk}
\tilde{\phi}(m,\ell,\lambda) = \Theta_f(m,\ell,\lambda), \quad\text{where}\quad f(x, \kappa, \lambda) = \varphi(x)\,\mathbbm{1}_{\{\kappa=0\}}.
\end{equation} Then $\tilde{\phi}\in \cg^s(\fHeis)$ is a  Gevrey function with compact support in $\fHeis$ and such that $\tilde{\phi}\equiv 1$ on $\tilde{\mathcal{R}}\cap D$, where $D$ is the diagonal $(m=\ell)$.   Recall that for all our computations we are considering $\eta$ such that $\eta<1/s$ for the Definition \ref{weight} of our exponential weights $w_\nu$. As in Lemma \ref{lemma-g-tau}, we write $\tilde{\phi}_\tau=\tilde{\phi}(m,\ell,\lambda/\tau)$ and $\hat{g_\tau} = \tilde{\phi}_\tau$, so that 
\begin{align} \label{eq-g-tau}
g_\tau(q) = \tau^{n+1} g_1(\delta_{\tau^{1/2}}\, q), \quad\text{for $q\in\Heis$.}
\end{align}
We have
\begin{align*}
     \widehat{P_t F}(m,\ell,\lambda) &= \sum_{j\in\N^n} \hat{p}_t(m,j,\lambda) \hat{F}(j,\ell,\lambda)\\
     &=\hat{p}_t(m,m,\lambda) \hat{F}(m,\ell,\lambda)\\&= \hat{p}_t(m,m,\lambda) \hat{F}(m,\ell,\lambda) \tilde{\phi}_\tau(m,m,\lambda),
\end{align*} where we have used the fact that $\tilde{\phi}_\tau=1$ on $\tau\tilde{\mathcal{R}}\cap D$ for the last identity. It follows that \begin{align*}
\widehat{P_t F}(m,\ell,\lambda) &= ( \tilde{\phi}_\tau\cdot \hat{p}_t)(m,m,\lambda)\hat{F}(m,\ell,\lambda)\\ &= ((\tilde{\phi}_\tau\cdot \hat{p}_t)\cdot\hat{F})(m,\ell,\lambda).
\end{align*}Taking inverse Fourier transforms above, we have thus obtained that\[P_t F=(g_\tau\star p_t)\star F.\]
Thus applying Proposition \ref{young-ineq}, we get 
\begin{align}\label{eq-PtF}
\lVert P_t F\rVert_{L^\alpha_\nu} = \lVert ((g_\tau \star p_t)\star F)\, w_\nu\rVert_{L^\alpha}\leq \lVert (g_\tau \star p_t)\, w_{-\nu} \rVert_{L^1}  \lVert F\rVert_{L^\alpha_\nu}.  
\end{align}
To establish \eqref{heat-flow-est}, it then suffices to show that
\begin{equation}\label{est0}
\lVert (g_\tau \star p_t)\, w_{-\nu} \rVert_{L^1}\leq Ce^{-c\tau t}.
\end{equation}
In the remainder of the proof we will focus on deriving the above inequality. Let us first use the identity \eqref{eq-g-tau} together with $\delta_{\tau^{1/2}}(qx^{-1})=\delta_{\tau^{1/2}q}(\delta_{\tau^{1/2}}x)^{-1}$. This enables to write  
\begin{align}\label{eq-gpt-conv}
(g_\tau\star p_t)(q) = \tau^{n+1} \int_{\Heis} g_1(\delta_{\tau^{1/2}}\,q(\delta_{\tau^{1/2}}\,x)^{-1}) p_t(x)d\mu(x).
\end{align}
We now perform the change of variable $\delta_{\tau^{1/2}}x=y$ in \eqref{eq-gpt-conv}, which yields
\begin{align}\label{eq-gpt-convolution-1}
(g_\tau\star p_t)(q) = \int_{\Heis} g_1((\delta_{\tau^{1/2}}\,q)\,  y^{-1}) p_t(\delta_{\tau^{-1/2}}\,y)d\mu(y).
\end{align}
In addition, it is readily checked from \eqref{eq-Heis-kernel} that $p_t((\delta_{\tau^{-1/2}}\,q))=p_{\tau t}(q)$. Plugging this identity into \eqref{eq-gpt-convolution-1} we end up with 
\begin{equation}\label{ll}
(g_\tau\star p_t)(q) = (g_1\star p_{\tau t})(\delta_{\tau^{1/2}}\,q), \quad \text{for $q\in\Heis$.}
\end{equation}
Thus using  Lemma \ref{lemma-exp-decay} and another change of variable $q:=\delta_{\tau^{1/2}q}$, we get
\begin{align*}
\lVert (g_\tau \star p_t)\, w_{-\nu} \rVert_{L^1}
&\leq C e^{-c\tau t}\int_\Heis e^{-c\,\lvert\delta_{\tau^{1/2}}\,q\rvert^{1/s}_h} e^{\nu\,|q|^{\eta}_*}d\mu(q)\\
&\leq C e^{-c\tau t}\tau^{-(n+1)}\int_\Heis e^{-c\,\lvert q\rvert^{1/s}_h} e^{\nu\,\lvert\delta_{\tau^{-1/2}}\,q\rvert^{\eta}_*}d\mu(q).
\end{align*}
In addition, since $\tau\geq 1$, we have $|\delta_{\tau^{-1/2}\,}q|_*\leq |q|_*$. Reporting this information in the right hand side above, we end up with
\begin{align*}
\lVert (g_\tau \star p_t)\, w_{-\nu} \rVert_{L^1}\leq C e^{-c\tau t}\int_\Heis e^{-c\,\lvert q\rvert^{1/s}_h} e^{\nu\,\lvert q\rvert^{\eta}_*}d\mu(q).
\end{align*}Owing to the  fact that we have chosen $\eta<1/s$ we conclude that \eqref{est0}   holds. This finishes our proof.
\end{proof}
As a last preliminary result, we now state a proposition on time regularity for the action of the heat semigroup.
    \begin{proposition}[Time regularity  of the heat flow]\label{prop-time-reg}
    For a parameter $\tau\geq 1$, we consider a function $F \in \mathcal{S}(\Heis)$ such that  
    \[
    \emph{Supp}\, \hat{F}\subset \{(m, \ell,\lambda)\in\fHeis: |\lambda|(2|m|+n)< \tau\}.
    \]
    Then for all $\alpha\in[1, \infty)$ and $\nu\leq \nu_0$, there exists a constant $C>0$ such that for every $t\geq 0$ and $\tau\geq 1$, we have
    \begin{equation}\label{time-est0}
    \lVert (\emph{Id} - P_t) F\rVert_{L^\alpha_\nu}\leq C (t\tau\wedge 1) \lVert  F\rVert_{L^\alpha_\nu}.
    \end{equation}
    \end{proposition}
\begin{proof}
Thanks to Proposition \ref{eff-heat-flow}, we have $\Vert P_t F\Vert_{L^\alpha_\nu}\lesssim \left\Vert F\right\Vert_{L^\alpha_\nu}$. Hence in order to show \eqref{time-est0}, it is sufficient to prove 
\begin{equation}\label{time-est}
  \lVert (\text{Id} - P_t) F\rVert_{L^\alpha_\nu}\leq C t\tau \lVert  F\rVert_{L^\alpha_\nu}.   
\end{equation}
In order to derive \eqref{time-est}, let $\tilde{\phi}$ be as in Lemma \ref{lemma-phi}, and let $\tilde{\phi}_\tau(m,\ell,\lambda) = \tilde{\phi}(m,\ell,\lambda/\tau)$. Since $\text{Supp}\, \hat{F}\subset B_\tau$, we have 
\[\tilde{\phi}_\tau\cdot \hat{F} = \hat{F} = \mathbbm{1}_{D}\cdot \hat{F}\] where $D$ is the diagonal in $\fHeis$. Hence we can write
\begin{align*}
     \widehat{(\text{Id} - P_t) F}&= (\mathbbm{1}_{D}-(\tilde{\phi}_\tau\cdot \hat{p}_t))\cdot \hat{F}\\ &=\tilde{\phi}_\tau\cdot(\mathbbm{1}_{D}- \hat{p}_t))\cdot \hat{F}\\ &= \widehat{(g_\tau\star H_t)\star F},
\end{align*}where $\hat{H_t}=\mathbbm{1}_{D}-\hat{p}_t$. As in \eqref{ll}, we have 
\begin{align}\label{eq-gHt-conv}
(g_\tau\star H_t)(q) =  (g_1\star H_{\tau t})(\delta_{\tau^{1/2}}\,q), \quad \text{for $q\in\Heis$.}
\end{align}

We now use some arguments which are directly imported from the proof of Lemma \ref{lemma-exp-decay}. In order to make the comparison more transparent, let us focus on the case $\tau=1$ in \eqref{eq-gHt-conv}. We let the reader  fill the details for a general $\tau>0$. First a series of elementary computations lead to the following  relation that is equivalent to \eqref{dd}:
 \[
 \big(\widehat{g\star H_t} \big)(m,\ell,\lambda) = \Theta_f(m,\ell,\lambda), \quad\text{with}\quad
 f(x,\kappa,\lambda) = \varphi \left(\frac{x}{\tau}\right)(1-e^{-4 t |x|}) \mathbbm{1}_{\{\kappa = 0 \}}.
 \]
 Next similarly to \eqref{gg} one can write 
 \[
f(x,\kappa, \lambda)=t\varphi(x)\zeta_t(x) \mathbbm{1}_{\{\kappa=0\}},
 \]
 where the function $\zeta_t$ is given by 
 \[
 \zeta_t(x)=\frac{1-e^{-4|x|t}}{t}.
 \]
 It is now easy to see that $\zeta_t$ verifies the series of inequalities \eqref{exp-est}. As in Lemma \ref{lemma-exp-decay}, it is enough to assert
\[
|(g_\tau\star H_{t})(q)|\leq Ct\tau e^{-c|q|_h^{1/s}}, \quad \text{for $q\in\Heis$}.
\]  Thus \eqref{time-est} follows as in the proof of Lemma \ref{eff-heat-flow}.
\end{proof}

Our desired smoothing effects for the heat semigroup are now direct consequences of Propositions \ref{eff-heat-flow} and \ref{prop-time-reg}. Our findings are summarized below.
 \begin{proposition}\label{prop-continuity}
 For $\kappa\in\R$, $\nu>0$ and $\alpha,\beta\in [1,\infty)$, let $\mathfrak{B}^{\ka,\nu}_{\alpha,\,\beta}$ be the Besov spaces introduced in Definition \ref{def-BS}. We also consider a time horizon $\tau$ and $\kappa'>\kappa$. Then there exists a constant $C=C_{\alpha,\beta,\kappa',\nu}$ such that for all $t\in [0,\tau]$ we have 
  \begin{equation}\label{eq-smoothing-a}
 \|P_tf\|_{ \mathfrak{B}^{\kappa', \nu}_{\alpha,\,\beta}}\le C t^{-\frac{\kappa'-\kappa}{2}}\|f\|_{ \mathfrak{B}^{\kappa,\, \nu}_{\alpha,\,\beta}}.
 \end{equation}
 Moreover, if $\gamma$ is an additional positive parameter lying in $(0,1)$, the following time regularity holds true:
 \begin{equation}\label{eq-smoothing-b}
 \|(\Id-P_t)f\|_{ \mathfrak{B}^{\kappa-2\gamma,\, \nu}_{\alpha,\,\beta}}\le C t^{\gamma}\|f\|_{ \mathfrak{B}^{\kappa,\, \nu}_{\alpha,\,\beta}}.
 \end{equation}
 \end{proposition}
  \begin{proof}
 According to \eqref{eq-chi-supp} and \eqref{eqn-b}, the function $\sigma_k f$ is such that \[\text{Supp}(\widehat{\sigma_k f})\subset \{2^k R_1\leq |\lambda|(2|m|+n)\leq 2^k R_2\},\] for two constants $0<R_1<R_2$. Otherwise stated, according to our notation \eqref{aa} and \eqref{bb}, we have $\text{Supp} (\widehat{\sigma_k f})\subset 2^k \tilde{R}$. Hence applying Proposition \ref{eff-heat-flow}, we obtain\begin{equation}  
 \label{cc} \lVert P_t(\sigma_k f)\rVert_{L^\alpha_\nu}\leq C
_0\,  e^{-c 2^k t} \lVert  \sigma_k f\rVert_{L^\alpha_\nu}.\end{equation} In addition, since both $\sigma_k f$ and $P_t f$ are obtained through multiplication operators in Fourier modes, it is clear that $\sigma_k(P_t f) = P_t(\sigma_k f)$. Therefore one can recast \eqref{cc} as  \begin{equation}
    \label{ddd}
    \lVert \sigma_k(P_t f)\rVert_{L^\alpha_\nu}\leq C
_0\,  e^{-c 2^k t} \lVert  \sigma_k f\rVert_{L^\alpha_\nu}.
\end{equation} Consider now $\kappa' >\kappa$ as in our inequality \eqref{eq-smoothing-a}. Thanks to \eqref{ddd}, we have \begin{align*}
 2^{\kappa' k}\lVert \sigma_k(P_t f)\rVert_{L^\alpha_\nu}&\leq C
_0\,  t^{-\frac{\kappa'-\kappa}{2}}\left[(2^kt)^{\frac{\kappa'-\kappa}{2}} e^{-c2^{k}t}\right]2^{\kappa k t} \lVert  \sigma_k f\rVert_{L^\alpha_\nu}.
 \end{align*} Moreover, some elementary calculus considerations show that $(2^kt)^{\frac{\kappa'-\kappa}{2}} e^{-c2^{k}t}$ is bounded by a constant which only depends on $\kappa,\kappa'$. Thus we discover that \begin{equation}
     \label{ee}
     2^{\kappa' k}\lVert \sigma_k(P_t f)\rVert_{L^\alpha_\nu}\leq C
_{\kappa,\kappa'}\,  t^{-\frac{\kappa'-\kappa}{2}}\, 2^{\kappa k t}\, \lVert  \sigma_k f\rVert_{L^\alpha_\nu}.
 \end{equation}
    We now easily get our claim \eqref{eq-smoothing-a} by taking  $\ell^\beta$ norms on both sides of \eqref{ee}. \\

In order to prove \eqref{eq-smoothing-b}, we proceed very similarly to \eqref{eq-smoothing-a}. Namely, we invoke the fact that $\text{Supp} (\widehat{\sigma_k f})\subset 2^k \tilde{R}$ and the relation $\sigma_k(\text{Id} - P_t)f = (\text{Id} - P_t)\sigma_k f$. This allows to obtain, similarly to \eqref{ddd}, that \begin{equation}
        \label{ff}
        \lVert \sigma_k((\text{Id}- P_t) f)\rVert_{L^\alpha_\nu}\lesssim (2^kt\wedge 1)\,  \lVert \sigma_k f\rVert_{L^\alpha_\nu}.
    \end{equation}
    Now multiply both sides of \eqref{ff} by $2^{(\kappa-\gamma) k}$. Some trivial algebraic manipulations show that 
    \begin{align}\label{ggg}
 2^{(\kappa-\gamma) k}\,\lVert \sigma_k((\text{Id}- P_t) f)\rVert_{L^\alpha_\nu}&\lesssim 2^{(\kappa-\gamma) k}\,(2^kt\wedge 1)\,  \lVert \sigma_k f\rVert_{L^\alpha_\nu}\\
 &\lesssim t^{\gamma}\left[(2^kt)^{-\gamma} (2^kt\wedge 1)\right]2^{\kappa k} \lVert  \sigma_k f\rVert_{L^\alpha_\nu}.\nonumber
 \end{align}We now invoke the fact that $\gamma\in (0,1)$. In that case, the term $(2^kt)^{-\gamma} (2^kt\wedge 1)$ is uniformly bounded. Therefore \eqref{ggg} becomes \begin{equation}
     \label{hhh}
     2^{(\kappa-\gamma) k}\,\lVert \sigma_k((\text{Id}- P_t) f)\rVert_{L^\alpha_\nu}\lesssim  t^{\gamma}2^{\kappa k} \lVert  \sigma_k f\rVert_{L^\alpha_\nu}.
 \end{equation}Thus one can conclude that \eqref{eq-smoothing-b} holds by taking $\ell^\beta$ norms of both sides of \eqref{hhh}. This concludes our proof.
 \end{proof}

\subsection{Paraproduct}\label{sec-paraprod}
In this section we define products of functions in our weighted Besov spaces $\mathfrak{B}^{\gamma}_{\alpha,\,\beta}$. More specifically we are interested in products of the form $f_{1}\cdot f_{2}$, with a function $f_{1}$ in a Besov space such that $\gamma_{1}>0$ and a distribution $f_{2}$ in a Besov space with $\gamma_{2}<0$. This will allow us to deal with the noisy products in equation~\eqref{eq:pam}.

In order to state our main result, and still with equation~\eqref{eq:pam} in mind, we will extend the class of exponentially weighted Besov spaces (introduced in Definition~\ref{weight}) to a wider class of weights including polynomials. 
\begin{definition}\label{def-weight-class}
In the sequel we will handle a class of $\ct$ of weights $w: \Heis\to \R_+$ such that $\ct=\hat{\ct}\cup\tilde{\ct}$ with 
\begin{equation}\label{eq-hat-ct}
\hat{\ct}=\big\{ w:\Heis\to\R_+; \, w(q)=e^{-\nu|q|_*^\eta}, \ \text{for}\ \nu\in\R\ \text{and}\ \eta\in(0,1/s)\bigg\}
\end{equation}
and where $\tilde{\ct}$ is defined by 
\begin{equation}\label{eq-til-ct}
\tilde{\ct}=\big\{ \rho:\Heis\to\R_+; \, \rho(q)=|q|_*^{-b}, \ \text{for}\ b\in\R\bigg\}.
\end{equation}
We will write $w_\nu$ (as in Definition \ref{weight}) for a generic element of $\hat{\ct}$, as well $\rho_b$ for a generic element of $\tilde{\ct}$.
\end{definition}

The above definition leads to a natural extension of Definition \ref{def-BS} to general weighted Besov spaces. We label this definition here for further use.

\begin{definition}\label{def-BS-general} Let $f\in\mathcal{S}(\Heis)$, and $w:\Heis\to \R_+$ a weight function in the class $\ct$ introduced in Definition \ref{def-weight-class}. We denote by $L_w^\alpha$ the corresponding space $L^\alpha(\Heis,w(\cdot)\, d\mu(\cdot))$, and  by $(\cdot,\cdot)_w$ the inner product in $L^2_w$. For  any $\gamma\in \R$ and $\alpha,\beta\in[1,\infty]$ , we define the norm
\[
\| f \|_{\mathfrak{B}^{\gamma,\, w}_{\alpha,\,\beta}}=\left\{\sum_{k=-1}^{\infty} \left( 2^{\gamma k}  \| \sigma_{k} f \|_{L^\alpha_w}\right)^\beta \right\}^{\frac{1}{\beta}},
\]
where $\sigma_k f$ are defined by \eqref{eqn-c} for $k\geq -1$. The weighted Besov space $\mathfrak{B}^{\gamma,\, w}_{\alpha,\,\beta}(\Heis)$ is defined as the completion of $\mathcal{S} (\Heis)$ with respect to this norm.
\end{definition}

We can now spell out our main result on distributional paraproduct, for our general class of weights $\ct$.

\begin{proposition}\label{prop-product}
Let $w_1$, $w_2$ be two weight functions in the class $\ct$, and $\kappa_1>0>\kappa_2$ such that $\kappa_1>|\kappa_2|$. Let $\alpha,\alpha_1,\alpha_2,\beta \in [1,\infty]$ be parameters satisfying
 $
 \frac{1}{\alpha} = \frac{1}{\alpha_1} + \frac{1}{\alpha_2}.
 $
 Then the product map 
\[
(f_1, f_2)\in \mathfrak{B}^{\kappa_1,\, w_1}_{\alpha_1,\,\beta}\times \mathfrak{B}^{\kappa_2,\, w_2}_{\alpha_2,\,\beta}\to f_1f_2\in \mathfrak{B}^{\kappa_2,\, w_1w_2}_{\alpha,\,\beta}
\]
is continuous and the following inequality holds true:
\[
\|f_1f_2\|_{\mathfrak{B}^{\kappa_2,\, w_1w_2}_{\alpha,\,\beta}}\le \|f_1\|_{\mathfrak{B}^{\kappa_1,\, w_1}_{\alpha_1,\,\beta}}\|f_2\|_{\mathfrak{B}^{\kappa_2,\, w_2}_{\alpha_2,\,\beta}}.
\]
\end{proposition}

\begin{remark}
Note that Proposition \ref{prop-product} reads similarly to more classical paraproduct inequalities in $\R^d$ (see e.g. \cite[Theorem 2.82]{BCD2}). However, our method of proof has to depart from the flat case in very fundamental aspects. At the heart of the problem lies the following fact: in $\R^d$, for two Schwartz functions $\varphi, \psi$, we have
\begin{align}\label{eq-F-prod}
\cf_{\R^d}(\varphi \, \psi)=\cf_{\R^d}(\varphi)\star\cf_{\R^d}(\psi).
\end{align}
This enables to prove that $\supp \cf_{\R^d}(\varphi)\subset \mathcal{C}_1$ and  $\supp \cf_{\R^d}(\psi)\subset \mathcal{B}$ imply that 
\begin{align}\label{eq-inclus}
\supp \cf_{\R^d}(\varphi \, \psi)\subset \mathcal{C}_2,
\end{align}
where $\mathcal{C}_1$, $\mathcal{C}_2$ are annulus and $\mathcal{B}$ is a ball. Unfortunately, there is no equivalent of relation~\eqref{eq-F-prod} in $\Heis$. As a result, the road to establish an inclusion like \eqref{eq-inclus} for the Heisenberg group is quite long. We shall describe this path in the next subsection.
\end{remark}
We finish the statement of our main paraproduct results by giving an immediate consequence of Proposition \ref{prop-product} that will be used later. 
\begin{corollary}
Let $w_1,w_2\in \ct$ and $\kappa_1, \kappa_2\in\R$ be as given in Proposition \ref{prop-product}. Then it holds that 
\[
\|f_1f_2\|_{\mathfrak{B}^{\kappa_2,\, w_1w_2}_{\alpha,\,\beta}}\le \|f_1\|_{\mathfrak{B}^{\kappa_1,\, w_1}_{\alpha,\,\beta}}\|f_2\|_{\mathfrak{B}^{\kappa_2,\, w_2}_{\infty,\,\beta}}.
\]
\end{corollary}
\begin{proof}
The conclusion follows from plugging $\alpha_1=\alpha$ and $\alpha_2=\infty$ in Proposition \ref{prop-product}.
\end{proof}
\subsubsection{Preliminary results}\label{sec-preliminary}

The paraproduct problem in $\Heis$ has been addressed in \cite{BG}, albeit in a different Fourier setting. Our strategy consists in mapping their results to ours by comparing Fourier projections. In the process, radial functions and Laguerre calculus play a prominent role. We first recall some basic facts about those objects here. Let us start with a basic definition.

\begin{definition}\label{def-radial-f}
Let $f:\Heis\to\R$ be a continuous function. We say that $f$ is radial if there exists another function $\rf: \R_+\times \R\to\R$ such that for all $q=(x,y,z)$ we have
\[
f(x,y,z)=\rf(|v|,z),
\]
where we have set $v=(x,y)$ and the norm $|v|$ has been defined in~\eqref{eq:radial}.
\end{definition} 

Our next proposition features Laguerre polynomials, that we proceed to define now.

\begin{definition}
For $n,k\ge0$, the Laguerre polynomial $L^k_n$ of order $k$ and type $n$ is defined by
\begin{align}\label{eq-Lauguerre}
L_k^{n}(x)=\sum_{j=0}^k (-1)^j \begin{pmatrix} k+n \\  k-j \end{pmatrix} \frac{x^j}{j!}
\end{align}
\end{definition}

We now state a result taken from \cite[relation (3.12)]{BBG}.
\begin{lemma}\label{lemma-Four-radial}
Let $g$ be a radial function in $\cs(\Heis)$, as given in Definition \ref{def-radial-f}. We consider the Fourier transform $\hat{g}$ of $g$, as spelled out in Definition \ref{def:projective-fourier}. Then there exists $\rg:\N\times\R\to\C$ such that for $(m,\ell, \lambda)\in\fHeis$ we have  
\begin{align}\label{eq-radial-FT}
\hat{g}(m,\ell, \lambda)=\rg(|m|, \lambda)\delta_{m,\ell}
\end{align}
where for any $k\in \N$ we have set
\[
\rg(k,\lambda)=\begin{pmatrix} k+n-1 \\  k \end{pmatrix} ^{-1}\int_\Heis \overline{e^{iz\lambda} \ck_{k, \lambda}(x,y)}g(q)d\mu(q),
\]
and where the function $\ck_{k, \lambda}$ can be expressed in terms of the kernels $K$ in \eqref{eq-FT-K} and Laguerre polynomials $L$ in \eqref{eq-Lauguerre}:
\[  
\ck_{k, \lambda}(x,y):=\sum_{m\in \N^n,|m|=k}K_{m,m,\lambda}(x,y)=e^{-|\lambda| |v|^2} 
L_k^{n-1}\!\left(2|\lambda| |v|^2\right),
\]
where we recall that $v$ stands for the couple $(x,y)$.
\end{lemma}

%

In the sequel we will need to consider radial test functions. The following lemma asserts that this type of functions exists.
\begin{lemma}
Let $s>1$ and recall that ${B}_1$ is the ball in $\fHeis$ given by
\[
{B}_1=\{\hw\in\fHeis; \hat{d}_0(\hw)<1 \}.
\]
Let $\rphi$ be the function such that $\rphi\equiv1$ on ${B}_1$, and $\rphi\in \cg^s(\fHeis)$, as constructed in Lemma~\ref{lemma-phi}. We set $g=\cf^{-1}(\rphi)$. Then $g$ is a radial function.
\end{lemma}
\begin{proof}
The proof follows from considerations in \cite[page 17-18]{BBG}. Namely the function $\rphi$ constructed in Lemma \ref{lemma-phi} is supported on the diagonal $\{m=\ell\}$, In this situation, the Fourier inversion formula takes the form (recall that $g\equiv \cf^{-1}(\rphi)$):
\begin{align}\label{eq-g-function}
g(q)=\frac{2^{n-1}}{\pi^{n+1}}\sum_{\ell\in \N}\int_\R e^{i\lambda z}\omega_{\ell,\lambda}(v)\phi (\ell, \ell, \lambda)|\lambda|^d d\lambda,
\end{align}
where we have set $v=(x,y)$ and $q=(v,z)$ for a generic element $q\in \Heis$, and where the function $\omega$ is given by 
\[
\omega_{\ell,\lambda}(v)=e^{-|\lambda| |v|^2} L_\ell^{(n-1)}(2|\lambda| |v|^2).
\]
Since $\omega_{\ell,\lambda}$ is obviously radial, we deduce from \eqref{eq-g-function} that $g$ is radial as well.
\end{proof}


\subsubsection{Mapping Fourier settings}

As mentioned at the beginning of Section \ref{sec-preliminary}, our strategy towards a proper definition of paraproduct hinges on the Fourier setting from \cite{BG}. We shall first briefly recall this framework. 

\begin{lemma}\label{lem:representations-bg}
Consider the state space
\begin{equation}\label{k3}
\ch
=
L^{2}\lp  \C^{n}; \,  d\xi \rp \, .
\end{equation}
As in Definition \ref{def-radial-f}, we write $q=(x,y,z)=(v,z)$ for a generic element of $\Heis$. Then a set of representations $\{\tu_{q}^{\la}; \, q\in\Heis, \, \la\in \R\}$ is given by:
\begin{equation}\label{k4}
\tu_{q}^{\la}f(\xi)
=
\lp\frac{2|\la|}{\pi}  \rp^{n/2} e^{-|\la| |\xi|^{2}}
\exp\lp \la\lp \ii z + 2\lla \xi, v  \rra  - {|v|^{2}}   \rp\rp
f\lp \xi-\bar{v}  \rp \,  .
\end{equation}
\end{lemma}

\noindent
The representations given in Lemma \ref{lem:representations-bg} give raise to a Fourier transform which is labeled below for further use. This is the Fourier transform used in \cite{BG}.
\begin{definition}\label{def:fourier-bg}
We will call $\tcf$ the Fourier transform based on the representations $\tu_{q}^{\la}$ introduced in~\eqref{k4}. That is similarly to~\eqref{eq-Fourier-transf}, for a function $f\in L^1(\Heis)$ we set
\begin{equation}\label{eq-Fourier-transf-2}
\tcf(f) (\lambda)=\int_{\Heis} f(q)\tu_{q}^\lambda \,  d\mu(q) .
\end{equation}
\end{definition}

\begin{remark}
In a somewhat non standard way, the state space in \cite{BG} was depending on the parameter $\la\in\R$.
Namely for a given $\la$, the state space in \cite{BG} is
\begin{equation}\label{k1}
\ch_{\la}
=
L^{2}\lp  \C^{n}; \, \lp \frac{2|\la|}{\pi}  \rp^{n} e^{-2|\la| |\xi|^{2}} d\xi \rp \, .
\end{equation}
Then the representations in \cite{BG} are given by 
\begin{equation}\label{k2}
V_{q}^{\la}F(\xi)
=
\exp\lp \la\lp \ii z + 2\lla \xi, v  \rra  \rp - \frac{|z|^{2}}{2}  \rp
F\lp \xi-\bar{v}  \rp \, ,
\quad\text{for all } F\in\ch_{\la} \, .
\end{equation}
However, it is easy to go from~\eqref{k1}-\eqref{k2} to a single state space, as we did in~\eqref{k3}-\eqref{k4}.
\end{remark}

The analysis in \cite[Proposition 2.1]{BG} relies on a set of eigenfunctions that we now proceed to describe.
\begin{lemma}\label{lem:eigenfunctions-bg}
For $m\in\N^{n}$, $\la\in\R$ and $\xi\in\C^{n}$ define
\begin{equation}
F_{m}^{\la}(\xi)
=
\frac{\lp 2|\la|\xi  \rp^{m}}{(m{!})^{1/2}}
:=
\prod_{j=1}^{n} \frac{\lp 2|\la|\xi_{j}  \rp^{m_{j}}}{(m_{j}!)^{1/2}} \, .
\end{equation}
Then $\{F_{m}^{\la}; \, m\in\N^{n}\}$ is an orthonormal basis for $\ch$. Furthermore, whenever $f$ is radial we have
\begin{equation}\label{eq-F-R}
\lla  \tcf f(\la) F_{m}^{\la}, F_{m'}^{\la} \rra
=
\mathbbm{1}_{(m=m')} \, R_{m}^{f,\la} \, ,
\end{equation}
where the scalar $R_{m}^{f,\la}$ is defined by
\begin{equation}\label{eq-R-m}
R_{m}^{f,\la}
=
\binom{|m|+n-1}{|m|}^{-1}
\int_{\Heis} f(q) e^{\ii\la z} L_{m}^{(n-1)}(2|\la| |v|^{2}) e^{-2|\la| |v|^{2}} \, \mu(dq) \, ,
\end{equation}
and where the generic element of $\Heis$ is written $q=(v,z)$.
\end{lemma}
We now start to relate our notion of Fourier transform to the one put forward in \cite{BG}. To this aim, observe first that the representations~\eqref{k3}-\eqref{k4} are obviously different from our representations $U_q^\lambda$ in~\eqref{eq:representation1}. However, the general representation theory in $\Heis$ asserts that the two objects are equivalent. Let us specify this point in another lemma.

\begin{lemma}\label{lem:eigenfunctions-bg}
The representations $\tu_{q}^{\la}$ given in Lemma~\ref{lem:representations-bg} are unitarily equivalent to~\eqref{eq:representation1}. Otherwise stated, there exits a unitary operator $A: L^{2}(\C^{n})\to L^{2}(\R^{n})$ such that for every $q\in\Heis,\, \la\in\R$ we have
\begin{equation}\label{k5}
\tu_{q}^{\la} = A^{*} U_{q}^{-\la} A,
\quad\text{and}\quad
U_{q}^{\la} = A \,\tu_{q}^{-\la} A^{*} .
\end{equation}
In addition, one can choose the operator $A$ so that
\begin{equation}\label{eq-F-A-Phi}
F_{m}^{\la} = A^{*}\Phi_{m}^{\la} .
\end{equation}
\end{lemma}

\begin{proof}
The existence of $A$ such that \eqref{k5} holds true is ensured by general representation theory considerations. Hence we also have
\begin{equation}\label{eq-FourierTs}
\cf f(\la) = A\,\tcf f(-\la) A^{*}.
\end{equation}
Apply the above relation to $\Phi_{m}^{\la}$. This yields
\begin{equation}\label{k6}
\lla  \cf f(\la) \Phi_{m}^{\la}, \Phi_{m'}^{\la} \rra
= 
\lla  A\,\tcf f(-\la) A^{*} \Phi_{m}^{\la}, \Phi_{m'}^{\la} \rra
=
\lla  \tcf f(-\la) A^{*} \Phi_{m}^{\la}, A^{*}\Phi_{m'}^{\la} \rra .
\end{equation}
Moreover, one can easily generate operators of the form $\tcf f(\la)$ which are invertible by considering radial functions $f$. 
Therefore invoking~\eqref{k6}, we get $F_{m}^{\la} = A^{*}\Phi_{m}^{\la}$ if one can prove that for a generic radial function $f$ we have
\begin{equation}\label{eq-cf-phi}
\lla  \cf f(\la) \Phi_{m}^{\la}, \Phi_{m'}^{\la} \rra 
=
\mathbbm{1}_{(m=m')} \, R_{m}^{f,-\la},
\end{equation}
where we recall that $R_{m}^{f,\la}$ is defined by \eqref{eq-R-m}.
To verify \eqref{eq-cf-phi}, we just need to recall \eqref{eq-F-trans-decomp} and apply Lemma \ref{lemma-Four-radial}. We obtain that for any radial function $f$, 
\begin{multline*}
\lla  \cf f(\la) \Phi_{m}^{\la}, \Phi_{m'}^{\la} \rra \\
=\mathbbm{1}_{(m=m')} 
\begin{pmatrix} |m|+n-1 \\  |m| \end{pmatrix} ^{-1}
\int_\Heis {e^{-iz\lambda} e^{-|\lambda|(|x|^2+|y|^2)} L_{|m|}^{n-1}(2|\lambda| |(x,y)|^2)}f(q)dq \, ,
\end{multline*}
where $L_k^{n-1}$ denotes the Laguerre polynomial of order $k$ and type $n-1$ as given in~\eqref{eq-Lauguerre}. This matches with \eqref{eq-R-m} for $R^{f,-\lambda}_m$, which ends the proof.

\end{proof}

\subsubsection{Frequency localization of paraproducts}

An important step in the definition of products is the following: show that the product of 2 functions which are localized in Fourier modes is still localized in Fourier modes. Below we prove a projective version of this fact, still inspired by~\cite{BG}.

\begin{proposition}\label{prop-support}
Consider an annulus $\cac_{0}$ in $\R$ and $j,l\in\Z$. Let $f,g$ be radial functions in $\cs(\Heis)$ such that for all 
$m\in \N^n$ we have
\begin{align}\label{eq-cond-f}
\cf (f)(\lambda)\Phi^\lambda_m=0 \quad \text{for all } \lambda\not\in (2|m|+n)^{-1}2^{2j} \cac_0
\end{align}
and
\begin{align}\label{eq-cond-g}
\cf (g)(\lambda)\Phi^\lambda_m=0 \quad \text{for all } \lambda\not\in (2|m|+n)^{-1}2^{2l} \cac_0.
\end{align} 
Then the following holds true:
 \begin{enumerate}[wide, labelwidth=!, labelindent=0pt, label= \emph{\textbf{(\arabic*)}}]
\setlength\itemsep{.05in}

\item
If $\ell>j+1$, then  there exists an annulus $\cac_0'$ such that 
\begin{align*}
\cf (fg)(\lambda) \Phi^\lambda_m=0 \quad \text{for all } \lambda\not\in (2|m|+n)^{-1}2^{2j} \cac'_0
\end{align*} 

\item
If $|\ell-j|\le 1$, there exists a ball $\mathcal{B}'_0$ such that 
\begin{align*}
\cf (fg)(\lambda) \Phi^\lambda_m=0 \quad \text{for all } \lambda\not\in (2|m|+n)^{-1}2^{2j} \mathcal{B}'_0.
\end{align*} 
\end{enumerate}
\end{proposition}

\begin{proof}  
First  combining \eqref{eq-FourierTs} and \eqref{eq-F-A-Phi} we have for any $h\in  \cs'({\Heis})$  that
\begin{align}\label{m1}
 \cf h(-\la)\Phi_m^\la =A \tcf h(\la)F_m^\la.
\end{align}
Therefore the conditions \eqref{eq-cond-f} and \eqref{eq-cond-g} implies that
\[
 \tcf f(\la)F_m^\la = \mathbbm{1}_{(2|m|+n)^{-1}2^{2j} \cac_0}(\la) \, \tcf f(\la)F_m^\la
\]
and 
\[
 \tcf g(\la)F_m^\la = \mathbbm{1}_{(2|m|+n)^{-1}2^{2l} \cac_0}(\la) \, \tcf g(\la)F_m^\la.
\]
As announced at the beginning of the section, we will now be able to transfer some of the results from \cite{BG} to our setting. Namely resorting to \cite[Proposition 4.2]{BG} we obtain that 
 \begin{enumerate}[wide, labelwidth=!, labelindent=0pt, label= {{(\roman*)}}]
\setlength\itemsep{.05in}

\item
When $\ell-j>1$, there exists an annulus $\chi_0'$ such that 
\begin{align}\label{eq-supp-fg-C}
 \tcf (fg)(\la)F_m^\la = \mathbbm{1}_{(2|m|+n)^{-1}2^{2j} \chi'_0}(\la) \,\tcf (fg)(\la)F_m^\la.
\end{align} 

\item 
When $|\ell-j|\le 1$, there exists a ball $\mathcal{B}'_0$ such that 
\begin{align}\label{eq-supp-fg-B}
 \tcf (fg)(\la)F_m^\la = \mathbbm{1}_{(2|m|+n)^{-1}2^{2\ell} \mathcal{B}'_0}(\la) \, \tcf (fg)(\la)F_m^\la.
\end{align} 
\end{enumerate}
Here it should be recalled that the simple localization formulae \eqref{eq-supp-fg-C}-\eqref{eq-supp-fg-B} are obtained in~\cite{BG} thanks to the diagonalization property \eqref{eq-cf-phi} for $\cf f(\la)$, whenever $f$ is radial. We now conclude by using \eqref{m1} again. Namely we have
\[
 \cf (fg)(-\la)\Phi_m^\la =A \tcf(fg)(\la)F_m^\la.
\]
Hence by symmetry of $\cac_0$, if $\tcf(fg)(\la)F_m^\la=0$ for $\la\not\in (2|m|+n)^{-1}2^{2j}\cac_0$, the same is true for $\cf(fg)(\la)\Phi_m^\la$. This yields \eqref{eq-cond-f}. Relation \eqref{eq-cond-g} is obtained in the same way.
\end{proof}

\smallskip

\begin{remark}
Proposition 4.2 in~\cite{BG} is proved by relating Fourier transforms in $\Heis$ to Fourier transforms in $\R^{2n}$ for specific situations. We could have adopted the same strategy in our situation. However, it looked simpler to just import the beautiful computations of~\cite{BG} to our setting.

\end{remark}

For a function $f\in L^2(\Heis)$, the blocks are good examples of functions in $\mathcal{S} (\Heis)$. Therefore a direct application of Proposition \ref{prop-support} yields the following corollary, which mirrors the remark in \cite[page 92]{BG}.

\begin{corollary}\label{cor-support}
Let $f \in \mathcal{S} (\Heis)$. For $k\geq -1$, recall $\sigma_k f$ and $S_k f$ as defined in \eqref{eqn-c}. 
For $\varepsilon\in\{-1,0,1\}$ and $k\ge 0$   we set
\[
h_k^S:= S_{k-1}f\sigma_k g, \quad\text{and} \quad h_k^{\sigma,\epsilon}= \sigma_{k-\epsilon}f\sigma_k g.
\]
Then there exists an annulus $\cac_{0}^{\prime}$ and a ball $\cb_{0}^{\prime}$ given as in Proposition~\ref{prop-support} such that for all $(m,\ell,\lambda)\in \fHeis$ we have
\begin{eqnarray*}
\hat{h}_k^S(m,\ell,\lambda)&=&\mathbbm{1}_{ (2|m|+n)^{-1}2^{2k} \cac'_0}(\la) \, \hat{h}_k^S(m,\ell,\lambda)
 \\
\hat{h}_k^{\sigma,\epsilon}(m,\ell,\lambda)
&=&
\mathbbm{1}_{ (2|m|+n)^{-1}2^{2k} \mathcal{B}'_0}(\la) \, \hat{h}_k^{\sigma,\epsilon}(m,\ell,\lambda) .
\end{eqnarray*}

\end{corollary}
\subsubsection{Definition of paraproducts}
With the preliminary notions introduced above, we can now turn to a definition of paraproduct. Notice that this can now be done similarly to the flat case, as in e.g \cite{BCD}.

\begin{definition}  Let $f, g$ be smooth functions in $\mathcal{S}(\Heis)$. We define the quantities $ f\varolessthan g$ and $ f \varodot g$ by
 \begin{align*}
 f\varolessthan g &=  \sum_{j < k-1} \sigma_j f \ \sigma_k g = \sum_{k} S_{k-1} f \ \sigma_k g.
 \\
 f \varodot g &= \sum_{|j-k| \leq 1} \sigma_j f \ \sigma_k g,
 \end{align*}
 where  we recall that $\sigma_k f$ and $S_k f$ are defined in \eqref{eqn-c}. Then the product $fg$ can be decomposed  as
   \begin{equation}\label{prod-decomp}
  fg = f\varolessthan g + f \varodot g + f\varogreaterthan g,
  \end{equation} where $f\varogreaterthan g = g\varolessthan f$. 
\end{definition}  
  Our main task in this section is to extend the  product~\eqref{prod-decomp} to elements $f$ and $g$ in suitable  Besov spaces (especially when $g$ is a distribution). We shall do this by considering each term in the decomposition~\eqref{prod-decomp} separately. First we gather some relevant lemmas for Besov estimates of functions which are localized in Fourier modes. Those lemmas are elaborations of \cite[Lemma 2.23]{BCD2}
  
   \begin{lemma}\label{lemma-decomp}
   Let $\kappa\in\R$, $\alpha,\beta\in [1,\infty]$,  and let $\mathcal{C}\subset\fHeis$ be an  annulus. Then there exists a constant $C<\infty$ such that for any sequence $\{f_k;k\ge0\}$ of functions  in $\mathcal{S}(\Heis)$ satisfying 
   \begin{equation}
       \text{Supp} \hat{f}_k\subseteq 2^k \mathcal{C}\, , 
       \quad \text{and}\quad \left( 2^{\kappa k} \lVert f_k\rVert_{L^\alpha_\nu}\right)_{k\in\N} \in \ell^\beta,
   \end{equation} 
   the function $ f \defeq \sum^\infty_{k=0} f_k$ is an element of $\mathfrak{B}^{\kappa,\nu}_{\alpha,\beta}$ and the following holds true:
   \begin{equation}\label{eq-decomp-C}
\lVert f\rVert_{\mathfrak{B}^{\kappa,\nu}_{\alpha,\beta}}\leq C \left\lVert \left( 2^{\kappa k} \lVert f_k\rVert_{L^\alpha_\nu}\right)_{k\in\N}\right\rVert_{\ell^\beta}. 
   \end{equation}
  \end{lemma}
  
  \begin{proof}
  Clearly there exists an $N_0>0$ such that for all $|j-k|>N_0$ it holds that 
  \[
  \sigma_jf_k=0.
  \]
  Therefore for an arbitrary $j\in \N$ one has
  \begin{align}\label{eq-sig-f}
    \|\sigma_jf\|_{L_{\nu}^\alpha}= \left\Vert\sum_{|j-k|\le N_0}\sigma_jf_k\right\Vert_{L_{\nu}^\alpha} \le C\sum_{|j-k|\le N_0} \left\Vert \sigma_jf_k\right\Vert_{L_{\nu}^\alpha}.
  \end{align}
 In addition, recall from \eqref{eqn-c} that $\sigma_jf_k=\phi_j\star f_k$ for a smooth function $\phi_j$. Hence applying Proposition \ref{young-ineq} similarly to what we did in \eqref{eq-PtF} and observing that $\|\phi_j w_{-\nu}\|_{L^1}$ is uniformly bounded in $j$, relation \eqref{eq-sig-f} becomes
  \begin{align}\label{eq-sig-f1}
    \|\sigma_jf\|_{L_{\nu}^\alpha} \le C\sum_{|j-k|\le N_0} \left\Vert f_k\right\Vert_{L_{\nu}^\alpha}.
  \end{align}
 With an easy algebraic manipulation, we end up with
  \begin{align}\label{eq-sig-f2}
 2^{\kappa j}  \|\sigma_jf\|_{L_{\nu}^\alpha}\le C\sum_{|j-k|\le N_0}2^{\kappa k} \left\Vert f_k\right\Vert_{L_{\nu}^\alpha},
  \end{align}
  where we recall that the generic constant $C$ can change from line to line. The right hand side of \eqref{eq-sig-f2} can now be interpreted as a discrete convolution. Indeed, setting $F(\ell)=C\mathbbm{1}_{[-N_0,N_0]}(\ell)$ and $G(k)=2^{k\kappa} \left\Vert f_k\right\Vert_{L_{\nu}^\alpha}$, one can recast \eqref{eq-sig-f2} as
  \begin{align}\label{eq-FG}
  2^{\kappa j}  \|\sigma_jf\|_{L_{\nu}^\alpha}\le  (F\star G)(j). 
  \end{align}
Applying Young's inequality to the $\ell^\beta$ norm of the sequence $\{(F\star G)(j); j\ge0 \}$ we thus get
\begin{equation}\label{k7}
\lVert f\rVert_{\mathfrak{B}^{\kappa,\nu}_{\alpha,\beta}}\leq \|F\star G\|_{\ell^\beta}\le C \| (2^{k\kappa} \left\Vert f_k\right\Vert_{L_{\nu}^\alpha}) _{k\in \N} \|_{\ell^\beta},
\end{equation}
  where we have resorted to the $\ell^{1}$-norm of $F$ (which is obviously finite) in our application of Young's inequality.
Now relation~\eqref{k7} is our claim \eqref{eq-decomp-C}, and the proof is completed.
  \end{proof}
  
  When the Fourier modes of a function are localized in a ball, the estimates we obtain are not as good as in Lemma \ref{lemma-decomp} (namely they will only hold for a regularity parameter $\ka>0$). Nevertheless, the following lemma will be useful in order to define paraproducts in Besov spaces.
  
    \begin{lemma}\label{lemma-decomp-2}
    As in Lemma \ref{lemma-decomp}, consider a sequence $\{f_k;k\ge0\}$ of functions in $\mathcal{S}(\Heis)$. We assume that $\ka>0$ and
    \begin{equation}\label{eq-decom-assump2}
       \text{Supp} \hat{f}_k\subseteq 2^k \mathcal{B} \, ,
       \quad \text{and}\quad 
       \left( 2^{\kappa k} \lVert f_k\rVert_{L^\alpha_\nu}\right)_{k\in\N} \in \ell^\beta,
   \end{equation} 
   where $\mathcal{B}\subset\fHeis$ is a ball in $\fHeis$ with respect to the distance $\hat{d}_0$, and where  $\alpha,\beta\in [1,\infty]$. Define $f$ as  $  f = \sum^\infty_{k=0} f_k$. Then
   \begin{equation}\label{eq-decom-conclu2}
       f\in\mathfrak{B}^{\kappa,\nu}_{\alpha,\beta}\quad \text{and}\quad \lVert f\rVert_{\mathfrak{B}^{\kappa,\nu}_{\alpha,\beta}}\leq C \left\lVert \left( 2^{\kappa k} \lVert f_k\rVert_{L^\alpha_\nu}\right)_{k\in\N}\right\rVert_{\ell^\beta}. 
   \end{equation}
   \end{lemma}
\begin{proof}
This proof goes along the same lines as for Lemma \ref{lemma-decomp}. Details will thus be omitted. 
Recall in the definition \eqref{eqn-c} of $\sigma_j f_k$ that $\tilde{\phi}_j$ is supported in an annulus $2^j\mathcal{C}$. Since $\mathcal{B}$ is a ball, there exists an $N_1$ such that for all $j\ge k+N_1$ it holds that 
\[
\sigma_jf_k=0.
\]
Similarly to \eqref{eq-sig-f}-\eqref{eq-sig-f2}, this implies that 
\[
   \|\sigma_jf\|_{L_{\nu}^\alpha}= \left\Vert\sum_{k> j-N_1}\sigma_jf_k\right\Vert_{L_{\nu}^\alpha} \le C\sum_{k> j- N_1} \left\Vert f_k\right\Vert_{L_{\nu}^\alpha}.
\]
Therefore like in \eqref{eq-sig-f2} we obtain
  \[
 2^{\kappa j}  \|\sigma_jf\|_{L_{\nu}^\alpha}\le C\sum_{k>j- N_1}2^{\kappa(j-k)}2^{\kappa k} \left\Vert f_k\right\Vert_{L_{\nu}^\alpha}.
  \]
Hence using the same discrete convolution arguments as in \eqref{eq-FG}, we get
  \[
  2^{\kappa j}  \|\sigma_jf\|_{L_{\nu}^\alpha}\le  (F\star G)(j), 
  \]
  where $F(\ell)=C2^{\kappa \ell}\mathbf{1}_{[-\infty,N_1]}(\ell)$ and $G(k)=2^{k\kappa} \left\Vert f_k\right\Vert_{L_{\nu}^\alpha}$. Applying Young's inequality we obtain that 
  \[
\lVert f\rVert_{\mathfrak{B}^{\kappa,\nu}_{\alpha,\beta}}\leq \|F\star G\|_{\ell^\beta}\le C \| (2^{k\kappa} \left\Vert f_k\right\Vert_{L_{\nu}^\alpha}) _{k\in \N} \|_{\ell^\beta},
  \]
    where we have resorted to the $\ell^{1}$-norm of $F$ (which is  finite whenever $\ka>0$) in our application of Young's inequality.
 The proof is thus completed.
\end{proof}
  
  We now turn to the main result of this section, giving the definition of paraproducts in our weighted Besov spaces. 
  
  \begin{lemma}\label{lemma-est-prod} 
  \emph{(1)}  Suppose $\kappa, \kappa_1, \kappa_2 \in \R$ and $\alpha,\alpha_1,\alpha_2,\beta \in [1,\infty]$ are such that
 $$
 \kappa_1 \neq 0, \quad \kappa = (\kappa_1\wedge 0) + \kappa_2 \quad \text{and} \quad \frac{1}{\alpha} = \frac{1}{\alpha_1} + \frac{1}{\alpha_2}.
 $$
  Then the mapping $(f,g) \mapsto f \varolessthan g$ (defined for $f,g \in \mathcal{S}$) can be extended to a continuous bilinear map  $\mathfrak{B}^{\kappa_1,w_1}_{\alpha_1,\beta} \times {\mathfrak{B}^{\kappa_2,w_2}_{\alpha_2,\beta}}\to\mathfrak{B}^{\kappa, w_1w_2}_{\alpha, \beta}$, Moreover, there exists a constant $C < \infty$ such that 
 \begin{equation}
 \label{e:para}
 \|f \varolessthan g\|_{\mathfrak{B}^{\kappa,w_1w_2}_{\alpha,\beta}} \le C \, \|f\|_{B^{\kappa_1,w_1}_{\alpha_1,\beta}} \, \|g\|_{B^{\kappa_2,w_2}_{\alpha_2,\beta}}.
 \end{equation}
\emph{(2)} Suppose  $\kappa_1, \kappa_2 \in \R$ are such that $\kappa = \kappa_1 + \kappa_2 > 0$, and let $\alpha,\alpha_1,\alpha_2,\beta$ be as above. Then the mapping $(f,g) \mapsto f \varodot g$ can be extended to a continuous bilinear map  $\mathfrak{B}^{\kappa_1,w_1}_{\alpha_1,\beta} \times {\mathfrak{B}^{\kappa_2,w_2}_{\alpha_2,\beta}}\to\mathfrak{B}^{\kappa,w_1w_2}_{\alpha,\beta}$. Moreover, there exists a constant $C < \infty$ such that 
\begin{align}\label{eq-fn-k}
 \|f \varodot g\|_{\mathfrak{B}^{\kappa,w_1w_2}_{\alpha,\beta}} \le C \, \|f\|_{\kB^{\kappa_1,w_1}_{\alpha_1,\beta}} \, \|g\|_{\kB^{\kappa_2,w_2}_{\alpha_2,\beta}}.
 \end{align}
    
  \end{lemma}
\begin{proof}
We will prove the two items in our lemma separately. 

\noindent
\emph{Proof of item (1).} Recall that $f \varolessthan g$ is defined as $\sum_{k=0}^\infty h_k^s$, with $h_k^s=S_{k-1}f\sigma_kg$. Hence one can invoke Corollary \ref{cor-support} to assert that $\hat{h}_k^s$ is supported in an annulus $(2|m|+n)^{-1}2^{2k} \chi'_0$. Therefore we can apply a slight elaboration of \eqref{eq-decomp-C} in Lemma \ref{lemma-decomp}, which yields
\begin{align}\label{eq-fn-h}
 \|f \varolessthan g\|_{\mathfrak{B}^{\kappa,w_1w_2}_{\alpha,\beta}} \le  C\left\lVert \left( 2^{\kappa k} \lVert h_k^s\rVert_{L^\alpha_{w_1w_2}}\right)_{k\in\N}\right\rVert_{\ell^\beta}.
\end{align}
We now bound the right hand side of \eqref{eq-fn-h}. First recalling that $h_k^s=S_{k-1}f\sigma_kg$ and resorting to H\"older's inequality we have 
  \begin{align}\label{eq-bd-hks}
 \lVert h_k^s\rVert_{L^\alpha_{w_1w_2}} \le  \lVert S_{k-1} f \rVert_{L^{\alpha_1}_{w_1}}  \lVert   \sigma_k g\rVert_{L^{\alpha_2}_{w2}}.
\end{align}
In addition, from Definition \ref{def-BS} we have that
\begin{align}\label{eq-fn-i}
\lVert \sigma_{j} f \rVert_{L^{\alpha_1}_{w_1}} \le 2^{-\kappa_1 j} \lVert  f \rVert_{\mathfrak{B}^{\kappa_1,\, w_1}_{\alpha_1,\beta}}.
\end{align}
Hence summing over $j$ we get
\begin{align}\label{eq-fn-j}
\lVert S_{k-1} f \rVert_{L^{\alpha_1}_{w_1}}\le \sum_{j=-1}^{k-2}\lVert \sigma_{j} f \rVert_{L^{\alpha_1}_{w_1}}
\le \sum_{j=-1}^{k-2} 2^{-\kappa_1 j} \lVert  f \rVert_{\mathfrak{B}^{\kappa_1,\, w_1}_{\alpha_1,\beta}}\lesssim 2^{-(\kappa_1 \wedge 0)k} \lVert  f \rVert_{\mathfrak{B}^{\kappa_1,\, w_1}_{\alpha_1,\beta}}.
\end{align}
Plugging \eqref{eq-fn-i}  and \eqref{eq-fn-j} into \eqref{eq-bd-hks}, we then obtain
\[
 \lVert S_{k-1} f \ \sigma_k g\rVert_{L^\alpha_{w_1w_2}}\lesssim 2^{-(\kappa_1 \wedge 0)k} \lVert  f \rVert_{\mathfrak{B}^{\kappa_1,\, w_1}_{\alpha_1,\beta}}\lVert   \sigma_k g\rVert_{L^{\alpha_2}_{w_2}}.
\]
Therefore the right hand side of \eqref{eq-fn-h}  can be bounded as follows:
\begin{eqnarray}\label{eq-fn-m}
\left\lVert \left( 2^{\kappa k} \lVert S_{k-1} f \ \sigma_k g\rVert_{L^\alpha_{w_1w_2}}\right)_{k\in\N}\right\rVert_{\ell^\beta}
&\lesssim&   
\lVert  f \rVert_{\mathfrak{B}^{\kappa_1,\, w_1}_{\alpha_1,\beta}}\left\lVert \left( 2^{\kappa k_2} \lVert   \sigma_k g\rVert_{L^{\alpha_2}_{w_2}}\right)_{k\in\N}\right\rVert_{\ell^\beta}\notag\\
&=& 
\lVert  f \rVert_{\mathfrak{B}^{\kappa_1,\, w_1}_{\alpha_1,\beta}}  \lVert  g \rVert_{\mathfrak{B}^{\kappa_2,\, w_2}_{\alpha_2,\beta}}.
\end{eqnarray}
This proves our claim \eqref{e:para}.

\noindent
\emph{Proof of item (2).} The proof of \eqref{eq-fn-k} follows closely the arguments of item (1). Namely one writes
\[
 \|f \varodot g\|_{\mathfrak{B}^{\kappa,w_1w_2}_{\alpha,\beta}} = \sum_{k=0}^\infty \sum_{\varepsilon\in \{-1,0,1\}} h_{k}^{\sigma, \varepsilon},
\]
where $h_{k}^{\sigma, \varepsilon}=\sigma_{k-\varepsilon}f \sigma_k g$. Moreover, according to Corollary \ref{cor-support} we have $\hat{h}_{k}^{\sigma, \varepsilon}$ supported in a ball $(2|m|+n)^{-1}2^{2k}\mathcal{B}'_0$. Since we have assumed $\kappa>0$, Lemma \ref{lemma-decomp-2} allows to write 
\begin{align}\label{eq-fn-l}
 \|f \varodot g\|_{\mathfrak{B}^{\kappa,w_1w_2}_{\alpha,\beta}} \le C \sum_{\varepsilon\in\{-1,0,1\}}\left\lVert \left( 2^{\kappa k} \lVert h_{k}^{\sigma, \varepsilon}\rVert_{L^\alpha_{w_1w_2}}\right)_{k\in\N}\right\rVert_{\ell^\beta}
\end{align}
We now proceed to bound the right hand side of \eqref{eq-fn-l}. That is applying  H\"older inequality, for all $\varepsilon\in\{-1,0,1\}$ and $k\ge0$ we have
\begin{align}
 \lVert \sigma_{k-\varepsilon} f \ \sigma_k g\rVert_{L^\alpha_{w_1w_2}} \le  \lVert  \sigma_{k-\varepsilon} f \rVert_{L^{\alpha_1}_{w_1}}  \lVert   \sigma_k g\rVert_{L^{\alpha_2}_{w_2}}.
\end{align}
Moreover, from Definition \ref{def-BS} it holds that
\[
\lVert \sigma_{k-\varepsilon} f \rVert_{L^{\alpha_1}_{w_1}} \le 2^{-\kappa_1 (k-\varepsilon)} \lVert  f \rVert_{\mathfrak{B}^{\kappa_1,\, w_1}_{\alpha_1,\beta}},\quad
\lVert \sigma_{k} g \rVert_{L^{\alpha_2}_{w_2}} \le 2^{-\kappa_2 k} \lVert  f \rVert_{\mathfrak{B}^{\kappa_2,\, w_2}_{\alpha_2,\beta}}.
\]
Hence similarly to what we did for \eqref{eq-fn-m}, we end up with 
\[
\left\lVert \left( 2^{\kappa k} \lVert \sigma_{k-\varepsilon} f \ \sigma_k g\rVert_{L^\alpha_{w_1w_2}}\right)_{k\in\N}\right\rVert_{\ell^\beta}
\lesssim \|f\|_{\kB^{\kappa_1,w_1}_{\alpha_1,\beta}} \, \|g\|_{\kB^{\kappa_2,w_2}_{\alpha_2,\beta}}.
\]
This establishes \eqref{eq-fn-k} and finishes our proof.
\end{proof}
 
 We complete this section by giving a short proof of our product rules in Proposition \ref{prop-product}.
 
\begin{proof}[Proof of Proposition \ref{prop-product}] We start from the decomposition \eqref{prod-decomp} and bound the 3 terms $f \varolessthan g$, $f \varodot g$, and $f \varogreaterthan g$ seperately. Now for $f \varolessthan g$, Lemma \ref{lemma-est-prod} yields
\[
 \|f \varolessthan g\|_{\mathfrak{B}^{\kappa_2,w_1w_2}_{\alpha,\beta}} \le \|f \varolessthan g\|_{\mathfrak{B}^{\kappa,w_1w_2}_{\alpha,\beta}} \le C \, \|f\|_{\mathfrak{B}^{\kappa_1,w_1}_{\alpha_1,\beta}} \, \|g\|_{\mathfrak{B}^{\kappa_2,w_2}_{\alpha_2,\beta}}.
\]
Similar estimates also hold for $f \varogreaterthan g$ and $f \varodot g$.  This concludes the proof.
\end{proof}

\section{Existence and uniqueness of pathwise solutions for a class of SPDEs}\label{sec:spdes-existence}

As an illustration of the use of weighted Besov spaces, this section addresses the existence and uniqueness of pathwise solutions to \eqref{eq:pam}, for noises $\dW$ in certain Besov spaces. For this purpose  we shall replace first the noise $\dW$ by $\cW$, a nonrandom H\"older continuous function in time with values in a Besov space of distributions. 
Then we will prove that the noises $\dW^{\zeta,\alpha}$ in \eqref{eq-intro-cov} sit in convenient Besov spaces under proper condition on $\zeta, \alpha$.

\subsection{Paths in weighted spaces}\label{sec-path-weight}

 
We begin this section by introducing sets of H\"older functions in time, with values in weighted Besov spaces on $\Heis$. This is the kind of noise $\cW$ which will drive Equation \eqref{eq:pam}. Our basic spaces of H\"older functions are defined as 
 \begin{equation}\label{eq-holder-space}
C^{\vartheta,\kappa,w}_{\alpha,\beta}:=C^\vartheta([0,T]; \mathfrak{B}^{\kappa,\, w}_{\alpha,\,\beta}).
\end{equation}
We shall also need elaborations of those spaces taking into account weights $w_t$ decreasing with time. We label a definition in that sense below.
\begin{definition}\label{def-D-space}
Consider three positive coefficients $\nu, b,\eta$ and define a subset $\{w_t;t\in[0,T]\}$ of $\ct$ by setting
\begin{equation}\label{eq-weight-wt}
w_t(q)=e^{-(\nu+bt)|q|_*^\eta},\quad  q\in\Heis.
\end{equation}
We also consider the  space of $\theta$-H\"older continuous functions on $[0,T]$ that takes values in $\mathfrak{B}^{\kappa,\, w_t}_{\alpha,\,\beta}$ for each $t$:
\begin{align}\label{eq-holder-theta}
\cd^{\theta,\kappa,\nu,b}_{\alpha,\beta}:&=\lbrace f\in C([0,T]\times \Heis), \,\|f_t\|_{\mathfrak{B}^{\kappa,\, w_t}_{\alpha,\,\beta}}\le c_{T,f},\notag \\
&\qquad\qquad \text{and}\  \|f_t-f_s\|_{\mathfrak{B}^{\kappa,\, w_t}_{\alpha,\,\beta}}\le c_f|t-s|^\theta,\ \text{for all }\  0\le s<t\le T
\rbrace
\end{align}
We equip the space $\cd^{\theta,\kappa,\nu,b}_{\alpha,\beta}$ with the H\"older norm
\[
\|f\|_{\cd^{\theta,\kappa,\nu,b}_{\alpha,\beta}}:=\sup_{ 0\le s<t\le T}\frac{ \|\delta f_{st}\|_{\mathfrak{B}^{\kappa,\, w_t}_{\alpha,\,\beta}}}{|t-s|^\theta},
\]
where here and in the sequel we write time increments as 
\begin{equation}\label{eq-time-inc}
\delta f_{st}=f_t-f_s.
\end{equation}
\end{definition}


\subsection{Existence and uniqueness for the SHE}
The definitions in the previous section allow to give a precise meaning to the notion of solution to our rough PDE \eqref{eq:pam}. Namely a solution is defined in the following way.
\begin{definition}\label{def-mild-soln}
Consider the H\"older spaces $C_{\alpha,\beta}^{\vartheta,\kappa,w}$ defined by \eqref{eq-holder-space}. We are given a noise $\dot{\cW}\in C_{\alpha_2,\beta}^{\vartheta,-\gamma,\rho_b}$ for some $\gamma>0$, as well as a weight $\rho_b=c(1+|q|_*^b)$. The initial condition for our system is a function $u_0$ sitting in the Besov space $\mathfrak{B}_{\alpha_1,\beta}^{\kappa,\, w}$ from Definition \ref{def-BS}. Consider $u\in \cd_{\alpha_1,\beta}^{\theta,\kappa, \nu, b}$ for some $\kappa>0$, where the space $\cd_{\alpha_1,\beta}^{\theta,\kappa, \nu, b}$ is introduced in~\eqref{eq-holder-theta}. 
We say that $u$ is a mild solution to the equation \eqref{eq:pam} with initial condition $u_0$ if it holds that 
\begin{equation}\label{eq-mild-soln}
u(t,p)=P_tu_0(p)+\int_0^t \int_\Heis p_{t-s}(q^{-1}p)u(s,q)\dot{\cW}(ds,q)d\mu(q),
\end{equation}
where the time integral in \eqref{eq-mild-soln} is understood as a Young integral, while the spatial product $u(s,q)\dot{\cW}(ds, q)$ is meant as in Proposition \ref{prop-product}.
\end{definition}

\begin{remark}
In Definition \ref{def-mild-soln}, we implicitly assume that $\vartheta$, $\kappa$, $w$, $\alpha$, $\beta$, $\gamma$, $\rho_b$, $\alpha_2$ satisfy the assumptions of Proposition \ref{prop-product}. Also notice that in the sequel the parameter $\alpha_2$ for the noise $\dot{\cW}$ will be chosen as $\alpha_2=\infty$.
\end{remark}

\begin{remark}
We refer to \cite{FV}, for instance, for a proper definition of Young's integrals. One way to understand the right hand side of \eqref{eq-Phi} is to consider the time integral as a limit of Riemann sums. Namely consider a generic sequence of partitions $\{\Pi^n=(r_j)_{j\le n};n\ge1\}$ whose mesh goes to $0$ as $n\to\infty$. Then one defines
\begin{align}\label{eq-Riem-sum}
\int_0^t\int_\Heis &p_{t-s}(q^{-1}p)u(s,q)\dot{\cW}(ds,q)d\mu(q)\notag \\
&\quad\quad=\lim_{n\to\infty}\sum_{j=0}^{n-1}\int_\Heis p_{t-r_j}(q^{-1}p)u(r_j, q)(\delta\dot{\cW})_{r_jr_{j+1}}(q)d\mu(q),
\end{align}
where we recall the convention \eqref{eq-time-inc} for the time increments. 
\end{remark}
Let us now state the main assumption on the noise $\dot{\cW}$.
\begin{hypothesis}\label{hypo}
The noise $\dot{\cW}$ driving equation \eqref{eq-mild-soln} belongs to a space $C_{\infty,\beta}^{\vartheta,-\gamma,\rho_b}$ for a given $b>0$. In addition, we assume that $\vartheta$ and $\gamma$ are two positive numbers in $(0,1)$ such that $\vartheta>\frac{1+\gamma}2$. 
\end{hypothesis}
Given a $\dot{\cW}$ satisfying Hypothesis \ref{hypo}, we define the map $\Phi$  on $\cd_{\alpha,\beta}^{\theta,\kappa}$ that is given by
\begin{equation}\label{eq-Phi}
\Phi(v)(t,p):=P_tu_0(p)+\int_0^t \int_\Heis p_{t-s}(q^{-1}p)v(s,q)\dot{\cW}(ds,q)d\mu(q).
\end{equation}
Clearly a solution to \eqref{eq-mild-soln}, if exists, is a fixed point of $\Phi$. In the lemma below, we prove that $\Phi$ is a contraction on 
$\cd_{\alpha,\beta}^{\theta,\kappa, \nu,b}$ whenever the time interval is small enough.
\begin{lemma}\label{lemma-exi-uni}
Fix a time interval $[0,\tau]$ for $\tau\le T$. Given an input distribution $\dot{\cW}\in C_{\infty,\beta}^{\vartheta,-\gamma,\rho_b}$ that satisfies Hypothesis \ref{hypo}, let $\Phi$ be a map as defined in \eqref{eq-Phi}.  Then for $\tau$ small enough, there exists $\theta, \kappa$ satisfying $\theta+\vartheta>1$ and $\gamma<\kappa<1$ such that $\Phi$ is a contraction map from $\cd_{\alpha,\beta}^{\theta,\kappa, \nu,b}$ to $\cd_{\alpha,\beta}^{\theta,\kappa, \nu,b}$.
\end{lemma}
\begin{proof}
Consider $v^1, v^2\in \cd^{\theta,\kappa,\nu,b}_{\alpha,\beta}$, and we denote the difference by $v^{12}:=v^1-v^2$. Let
\[
\Phi^{12}:=\Phi(v^1)-\Phi(v^2).
\]
Then from the definition of $\Phi$ in \eqref{eq-Phi} we can easily observe that
\begin{equation}\label{eq-conte-mid1}
\Phi^{12}(t,p)=\int_0^t \int_\Heis p_{t-s}(q^{-1}p)v^{12}(s,q)\dot{\cW}(ds,q)d\mu(q).
\end{equation}
We claim that for $\tau$ small enough it holds that 
\begin{equation}\label{eq-claim-contraction}
\|\Phi^{12}\|_{\cd_{\alpha,\beta}^{\theta,\kappa,\nu,b}  }\le \frac12\|v^{12}\|_{\cd_{\alpha,\beta}^{\theta,\kappa,\nu,b}  },
\end{equation}
and notice that \eqref{eq-claim-contraction} is enough to prove our lemma.  We now divide the proof of \eqref{eq-claim-contraction} into several steps.

\noindent
\emph{Step1: Decomposition of $\Phi^{12}$.} 
Let us first change our space-time notation and let $\Phi^{12}_t(\cdot):=\Phi^{12}(t,\cdot)$ for all $t\in[0,\tau]$. Next according to the definition \eqref{eq-semigroup} of the heat semigroup, we can rewrite \eqref{eq-conte-mid1} as 
\begin{equation}\label{eq-phi-12-int}
\Phi^{12}_t=\int_0^t P_{t-r}(v_r^{12}\dot{\cW}(dr)).
\end{equation}
Thanks to the above expression, we now obtain an expression for the time increments of $\Phi^{12}$. Namely using \eqref{eq-phi-12-int} and the semigroup identity
$P_{t-r}-P_{s-r}=(P_{t-s}-\mathrm{Id})P_{s-r}$, valid for all $r<s<t$, we let the reader check that
\begin{align}\label{eq-decomp-i}
\Phi^{12}_t-\Phi^{12}_s=I_{st}+II_{st}
\end{align}
where the terms $I$, $II$ are respectively defined by
\begin{equation}\label{eq-def-I-II}
I_{st}=\int_0^s (P_{t-s}-\Id)P_{s-r}(v^{12}\dot{\cW}(dr)),
\quad\text{and}\quad 
II_{st}=\int_s^t P_{t-r}(v^{12}\dot{\cW}(dr)).
\end{equation}
We estimate the terms $I$ and $II$ separately in the sequel.

\noindent
\emph{Step 2: Approximation of $I$.} 
In order to get an upper bound for $I$, we will interpret the time integral in the Young sense as in \eqref{eq-Riem-sum}. Specifically one considers the dyadic partition of $[0,s]$, of the form $\{r_j=2^{-n}js; j=0,\dots, 2^n\}$. Then one writes
\begin{equation}\label{eq-I-limit}
I_{st}=\lim_{n\to\infty}I^n_{st},
\end{equation}
where the Riemann type sum $I^n_{st}$ is defined by 
\begin{equation}\label{eq-I-n}
I^n_{st}=\sum_{j=0}^{2^n-1}(P_{t-s}-\Id)P_{s-r_j}(v^{12}_{r_j}\, (\delta\dot{\cW})_{r_jr_{j+1}}).
\end{equation}
In \eqref{eq-I-n}, notice that we have invoked our convention \eqref{eq-time-inc} for the time increments again, and that the product $v^{12}\delta\dot{\cW}$ is understood as in Proposition \ref{prop-product}.

In order to establish the convergence of $I^n_{st}$, we will study the differences $I^n_{st}-I^{n-1}_{st}$. Let us first observe that $r^{n-1}_j=r^n_{2j}$ due to the dyadic nature of the family $(r^n_j)_{j\le 2^{n}}$. Hence setting $r_j=r_j^n$ and writing the definition \eqref{eq-I-n} of $I_{st}^n$, we get
\begin{multline*}
I^n_{st}-I^{n-1}_{st}=\sum_{j=0}^{2^n-1}(P_{t-s}-\Id)P_{s-r_j}(v^{12}_{r_j}\, (\delta\dot{\cW})_{r_jr_{j+1}})
\\
-\sum_{j=0}^{2^{n-1}-1}(P_{t-s}-\Id)P_{s-r_{2j}}(v^{12}_{r_{2j}}\, (\delta\dot{\cW})_{(r_{2j}r_{2j+2})}).
\end{multline*}
Therefore some trivial cancellations yield the expansion
\begin{multline}\label{eq-I-n-h}
I^n_{st}-I^{n-1}_{st}=\sum_{j=0}^{2^{n-1}-1}(P_{t-s}-\Id)\bigg[P_{s-r_{2j+1}}(v^{12}_{r_{2j+1}}\, (\delta\dot{\cW})_{r_{2j+1}r_{2j+2}})
\\
-P_{s-r_{2j}}(v^{12}_{r_{2j}}\, (\delta\dot{\cW})_{r_{2j+1}r_{2j+2}})\bigg].
\end{multline}
In addition, the semigroup property of $P$ enables to write
\begin{align*}
&P_{s-r_{2j+1}}(v^{12}_{r_{2j+1}}\, (\delta\dot{\cW})_{r_{2j+1}r_{2j+2}}))-P_{s-r_{2j}}(v^{12}_{r_{2j}}\, (\delta\dot{\cW})_{r_{2j+1}r_{2j+2}}))=L^{1,n,j}_{st}-L^{2,n,j}_{st},
\end{align*}
where $L^{1,n,j}$ and $L^{2,n,j}$ are respectively defined by
\begin{equation}\label{eq-L-1nj}
L^{1,n,j}_{st}=P_{s-r_{2j+1}}((\delta v^{12})_{r_{2j}r_{2j+1}}\, (\delta\dot{\cW})_{r_{2j+1}r_{2j+2}})
\end{equation}
and 
\begin{equation}\label{eq-L-2nj}
L^{2,n,j}_{st}=(P_{r_{2j+1}-r_{2j}}- \Id)P_{s-r_{2j+1}}(v^{12}_{r_{2j}}\, (\delta\dot{\cW})_{r_{2j+1}r_{2j+2}}).
\end{equation}
Plugging this information into \eqref{eq-I-n-h}, we have thus obtained that 
\begin{equation}\label{eq-I-L-nj}
I^n_{st}-I^{n-1}_{st}=\sum_{j=0}^{2^{n-1}-1}(P_{t-s}-\Id)(L^{1,n,j}_{st}-L^{2,n,j}_{st}).
\end{equation}
In the sequel we will estimate the two terms $(P_{t-s}-\Id)L^{1,n,j}_{st}$ and $(P_{t-s}-\Id)L^{2,n,j}_{st}$ separately.

\noindent
\emph{Step 3: Estimate for $(P_{t-s}-\Id)L^{1,n,j}_{st}$.} 
We shall bound the term $(P_{t-s}-\Id)L^{1,n,j}_{st}$ thanks to our Propositions about products and heat semigroup smoothing in Besov spaces (see Sections~\ref{sec-flow} and~\ref{sec-path-weight}). First we invoke \eqref{eq-smoothing-b} applied to Besov spaces with weight $w_t$ (see \eqref{eq-weight-wt} for the definition of $w_t$), which enables to write
\begin{align}\label{eq-Pts-L1nj}
\|(P_{t-s}-\Id)L^{1,n,j}_{st}\|_{\mathfrak{B}_{\alpha,\beta}^{\kappa,\, w_t}} &\le C(t-s)^{\theta}\|L^{1,n,j}_{st} \|_{\mathfrak{B}_{\alpha,\beta}^{\kappa+2\theta,\, w_t}}.
\end{align}
Next with the expression \eqref{eq-L-1nj} of $L^{1,n,j}$ in mind, a simple application of \eqref{eq-smoothing-a} yields
\begin{align}\label{eq-P-Id-m}
\|(P_{t-s}-\Id)L^{1,n,j}_{st}\|_{\mathfrak{B}_{\alpha,\beta}^{\kappa,\, w_t}}
\le 
\frac{C(t-s)^{\theta}\|(\delta v^{12})_{r_{2j}r_{2j+1}}\, (\delta\dot{\cW})_{r_{2j+1}r_{2j+2}}\|_{\mathfrak{B}_{\alpha,\beta}^{-\gamma,\, w_t}}}{(s-r_{2j+1})^{\frac{\kappa+\gamma}2+\theta}}.
\end{align}
Owing to definition \eqref{eq-weight-wt}, whenever $t\in[r_{2j+1}, r_{2j+2}]$ the weight $w_t$ verifies
\begin{equation}\label{eq-w_t}
w_t(q)=e^{-(\nu+bt)|q|_*^\eta}\le \frac{c}{(t-r_{2j+1})^{b/\eta}}\frac{e^{-(\nu+br_{2j+1})|q|_*^\eta}}{1+|q|_*^{b}}= \frac{c}{(t-r_{2j+1})^{b/\eta}}w_{r_{2j+1}}(q)\rho_b(q),
\end{equation}
for all $q\in \Heis$. In addition, if we consider $v^{12}\in  \cd_{\alpha,\beta}^{\theta,\kappa,\nu,\beta}$, we have $(\delta v^{12})_{r_{2j}r_{2j+1}} \in \mathfrak{B}_{\alpha,\beta}^{\kappa,\, w_{r_{2j+1}}}$. In the same way if $\dot{\cW}\in C_{\infty,\beta}^{\vartheta,-\gamma,\rho_b}$, it is readily checked that $(\delta\dot{\cW})_{r_{2j+1}r_{2j+2}}$ is an element of $\mathfrak{B}_{\infty,\beta}^{-\gamma,\, \rho_b}$. Since we have assumed $\kappa>\gamma$, Proposition \ref{prop-product} ensures that 
\begin{multline}\label{eq-prod-ell}
\|(\delta v^{12})_{r_{2j}r_{2j+1}}\, (\delta\dot{\cW})_{r_{2j+1}r_{2j+2}}\|_{\mathfrak{B}_{\alpha,\beta}^{-\gamma,\, w_{r_{2j+1}}\rho_b}} \\
\le
\|(\delta v^{12})_{r_{2j}r_{2j+1}}\|_{\mathfrak{B}_{\alpha,\beta}^{\kappa,\, w_{r_{2j+1}}}} 
\|(\delta\dot{\cW})_{r_{2j+1}r_{2j+2}}\|_{\mathfrak{B}_{\infty,\beta}^{-\gamma,\, \rho_b}} .
\end{multline}
Plugging \eqref{eq-weight-wt} and \eqref{eq-prod-ell} into \eqref{eq-P-Id-m}, we end up with 
\begin{align}\label{proof-mid-1}
\|(P_{t-s}-\Id)L^{1,n,j}_{st}\|_{\mathfrak{B}_{\alpha,\beta}^{\kappa,\, w_t}} &\le \frac{C(t-s)^{\theta}\|(\delta v^{12})_{r_{2j}r_{2j+1}}\|_{\mathfrak{B}_{\alpha,\beta}^{\kappa,\, w_{r_{2j+1}}}} 
\|(\delta\dot{\cW})_{r_{2j+1}r_{2j+2}}\|_{\mathfrak{B}_{\infty,\beta}^{-\gamma,\, \rho_b}}}{(s-r_{2j+1})^{\frac{\kappa+\gamma}2+\theta+\frac{b}{\eta}}}.
\end{align}
It remains to estimate the increments $\delta v^{12}$ and $\delta\dot{\cW}$. Due to the facts that $v^{12}\in \cd_{\alpha,\beta}^{\theta,\kappa,\nu,\beta}$ and $\dot{\cW}\in C_{\infty,\beta}^{\vartheta,-\gamma,\rho_b}$ we know that
\begin{eqnarray*}
\|(\delta v^{12})_{r_{2j}r_{2j+1}}\|_{\mathfrak{B}_{\alpha,\beta}^{\kappa,\, w_{r_{2j+1}}}} 
&\le& 
C\left(\frac{s}{2^{n}}\right)^{\theta} \|v^{12}\|_{\cd_{\alpha,\beta}^{\theta,\kappa,\nu,\beta}}, \\
\|(\delta\dot{\cW})_{r_{2j+1}r_{2j+2}}\|_{\mathfrak{B}_{\infty,\beta}^{-\gamma,\, \rho_b}} &\le& 
C\left(\frac{s}{2^{n}}\right)^{\vartheta}\| \dot{\cW}\|_{C_{\infty,\beta}^{\vartheta,-\gamma,\rho_b}}.
\end{eqnarray*}
Plugging the above estimates back into \eqref{proof-mid-1} we obtain that
\begin{align}\label{proof-mid-6}
\|(P_{t-s}-\Id)L^{1,n,j}_{st}\|_{\mathfrak{B}_{\alpha,\beta}^{\kappa,\, w_t}} &\le \frac{C(t-s)^{\theta} \|v^{12}\|_{\cd_{\alpha,\beta}^{\theta,\kappa,\nu,b}}
\| \dot{\cW}\|_{C_{\infty,\beta}^{\vartheta,-\gamma,\rho_b}}}{(s-r_{2j+1})^{\frac{\kappa+\gamma}2+\theta+\frac{b}{\eta}}}\left(\frac{s}{2^{n}}\right)^{\theta+\vartheta}.
\end{align}

\noindent
\emph{Step 4: Estimate for  $(P_{t-s}-\Id)L^{2,n,j}_{st}$.} 
Recall that $L^{2,n,j}_{st}$ is defined by \eqref{eq-L-2nj}.
In order to bound this term, similarly to what we did in \eqref{eq-Pts-L1nj} for $L^{1,n,j}_{st}$, we resort to \eqref{eq-smoothing-b}. This yields
\begin{align}\label{proof-mid-2}
\|(P_{t-s}-\Id)L^{2,n,j}_{st}\|_{\mathfrak{B}_{\alpha,\beta}^{\kappa,\, w_t} } &\le C(t-s)^{\theta}\|L^{2,n,j}_{st} \|_{\mathfrak{B}_{\alpha,\beta}^{\kappa+2\theta,\, w_t} }.
\end{align}
Next we go back to our expression \eqref{eq-L-2nj} for  $L^{2,n,j}_{st}$, we consider a small parameter $\epsilon>0$ and we invoke inequality \eqref{eq-smoothing-b} again. We end up with
\begin{align}\label{proof-mid-3-1}
\|L^{2,n,j}_{st} \|_{\mathfrak{B}_{\alpha,\beta}^{\kappa+2\theta,\, w_t} }&=\|(P_{r_{2j+1}-r_{2j}}- \Id)P_{s-r_{2j+1}}(v^{12}_{r_{2j}}\, (\delta\dot{\cW})_{r_{2j+1}r_{2j+2}})\|_{\mathfrak{B}_{\alpha,\beta}^{\kappa+2\theta,\, w_t}} \notag \\
&\le C (r_{2j+1}-r_{2j})^{1-\vartheta+\epsilon}\| P_{s-r_{2j+1}}(v^{12}_{r_{2j}}\, (\delta\dot{\cW})_{r_{2j+1}r_{2j+2}})\|_{\mathfrak{B}_{\alpha,\beta}^{\kappa+2\theta+2(1-\vartheta+\epsilon),\, w_t}}.\end{align}
We can now safely apply Proposition \ref{prop-continuity} to the right hand side of \eqref{proof-mid-3-1}. We get
\begin{align}\label{proof-mid-3}
\|L^{2,n,j}_{st} \|_{\mathfrak{B}_{\alpha,\beta}^{\kappa+2\theta,\, w_t} }
\le C \frac{(r_{2j+1}-r_{2j})^{1-\vartheta+\epsilon}}{(s-r_{2j+1})^{1-\vartheta+\epsilon+\theta+\frac{\kappa+\gamma}{2}}} \| v^{12}_{r_{2j}}\, (\delta\dot{\cW})_{r_{2j+1}r_{2j+2}}\|_{\mathfrak{B}_{\alpha,\beta}^{-\gamma,\, w_t}  }.
\end{align}
Moreover, invoking \eqref{eq-w_t} along the same lines as what we did in  \eqref{eq-prod-ell}, we have 
\begin{align}\label{proof-mid-4}
\| v^{12}_{r_{2j}}\, (\delta\dot{\cW})_{r_{2j+1}r_{2j+2}})\|_{\mathfrak{B}_{\alpha,\beta}^{-\gamma,\, w_t}  }
&\le \frac{C}{(t-r_{2j})^{b/\eta}}  \|v^{12}_{r_{2j}}\|_{\mathfrak{B}_{\alpha,\beta}^{\kappa,\, w_{r_{2j}}}  }
\|(\delta\dot{\cW})_{r_{2j+1}r_{2j+2}}\|_{\mathfrak{B}_{\infty,\beta}^{-\gamma,\, \rho_b}  } \notag\\
&\le  \frac{C}{(s-r_{2j+1})^{b/\eta}}  \|v^{12}\|_{ \cd_{\alpha,\beta}^{\theta,\kappa,\nu,b}  }
\|\dot{\cW}\|_{C_{\infty,\beta}^{\vartheta,-\gamma,\rho_b}  }\left( \frac{s}{2^n}\right)^\vartheta
\end{align}

Reporting \eqref{proof-mid-4} into \eqref{proof-mid-3} and then \eqref{proof-mid-3-1} and \eqref{proof-mid-2}, we end up with
\begin{align}\label{proof-mid-4-1}
\|(P_{t-s}-\Id)L^{2,n,j}_{st}\|_{\mathfrak{B}_{\alpha,\beta}^{\kappa,\, w_t} } &\le \frac{C(t-s)^{\theta}(r_{2j+1}-r_{2j})^{1-\vartheta+\epsilon} }{(s-r_{2j+1})^{1-\vartheta+\epsilon+\frac{\kappa+\gamma}2+\theta+\frac{b}{\eta}}}  \|v^{12}\|_{\cd_{\alpha,\beta}^{\theta,\kappa,\nu,b}}
\| \dot{\cW}\|_{C_{\infty,\beta}^{\vartheta,-\gamma,\rho_b}} \left(\frac{s}{2^{n}}\right)^{\theta+\vartheta}
\end{align}

\noindent
\emph{Step 5: Conclusion.} Let us summarize our considerations so far. Combining \eqref{proof-mid-6} and \eqref{proof-mid-4-1}  and picking $\theta=1-\vartheta+\epsilon$, we have obtained that
\[
\|(P_{t-s}-\Id)(L^{1,n,j}_{st}-L^{2,n,j}_{st})\|_{\mathfrak{B}_{\alpha,\beta}^{\kappa,\, w_t}  }  \le \frac{C(t-s)^{\theta}}{(s-r_{2j+1})^{2(1-\vartheta+\epsilon)+\frac{\kappa+\gamma}{2}+b/\eta}} \|v^{12}\|_{ \cd_{\alpha,\beta}^{\theta,\kappa,\nu,b}  }
\|\dot{\cW}\|_{C_{\infty,\beta}^{\vartheta,-\gamma,\rho_b}  } \left( \frac{s}{2^n}\right)^{1+\epsilon}.
\]
Now, chose $\kappa=\gamma+2\epsilon$. By Hypothesis \ref{hypo} 
we obtain that for sufficiently small  $\epsilon,b>0$, there exists a $\delta>0$ such that
\begin{align}\label{proof-mid-5}
\|(P_{t-s}-\Id)(L^{1,n,j}_{st}-L^{2,n,j}_{st})\|_{\mathfrak{B}_{\alpha,\beta}^{\kappa,\, w_t}  }  \le \frac{C(t-s)^{\theta}}{(s-r_{2j+1})^{1-\delta}} \|v^{12}\|_{ \cd_{\alpha,\beta}^{\theta,\kappa,\nu,b}  }
\|\dot{\cW}\|_{C_{\infty,\beta}^{\vartheta,-\gamma,\rho_b}  } \left( \frac{s}{2^n}\right)^{1+\epsilon}.
\end{align}
Next we recall that $I^n_{st}-I^{n-1}_{st}$ is defined in \eqref{eq-I-L-nj}. Hence summing \eqref{proof-mid-5} over $j$ and $n$ we discover that
\begin{multline}\label{eq-sum-g}
\sum_{n=1}^\infty \|I^n_{st}-I^{n-1}_{st}\|_{\mathfrak{B}_{\alpha,\beta}^{\kappa,\, w_t}  }  \\
\le 
C(t-s)^{\theta} \|v^{12}\|_{ \cd_{\alpha,\beta}^{\theta,\kappa,\nu,b}  } \|\dot{\cW}\|_{C_{\infty,\beta}^{\vartheta,-\gamma,\rho_b}  } \sum_{n=1}^\infty \left( \frac{s}{2^{n-1}}\right)^{\epsilon}\bigg( \frac{s}{2^{n-1}}\sum_{j=0}^{2^{n-1}-1}\frac{1}{(s-r_{2j+1})^{1-\delta}}\bigg) .
\end{multline}
The right hand side of \eqref{eq-sum-g} can be estimated thanks to elementary Riemann sums arguments. We thus get the following estimate:
\begin{align}\label{eq-sum-g-2}
\sum_{n=1}^\infty \|I^n_{st}-I^{n-1}_{st}\|_{\mathfrak{B}_{\alpha,\beta}^{\kappa,\, w_t}  }  
\le  C(t-s)^{\theta} \|v^{12}\|_{ \cd_{\alpha,\beta}^{\theta,\kappa,\nu,b}  } \|\dot{\cW}\|_{C_{\infty,\beta}^{\vartheta,-\gamma,\rho_b}  } s^{\epsilon+\delta}.
\end{align}
The fact that $\sum_{n=1}^\infty  \|I^n_{st}-I^{n-1}_{st}\|_{\mathfrak{B}_{\alpha,\beta}^{\kappa,\, w_t}  }  $ is an absolutely convergent series in $\mathfrak{B}_{\alpha,\beta}^{\kappa,\, w_t} $ allows to conclude that the limit of $I^n_{st}$ exists. Specifically, taking into account that $I^0_{st}=0$ we obtain the limit $I_{st}$ in \eqref{eq-I-limit} as the following series in $\mathfrak{B}_{\alpha,\beta}^{\kappa,\, w_t} $:
\[
I_{st}=\sum_{n=1}^\infty (I^n_{st}-I^{n-1}_{st}).
\]
For $s,t\in[0,\tau]$, relation \eqref{eq-sum-g-2} can thus be read as
\begin{equation}\label{eq-est-I}
\|I_{st}\|_{\mathfrak{B}_{\alpha,\beta}^{\kappa,\, w_t}  } \le C\tau^{\epsilon+\delta}(t-s)^{\theta} \|v^{12}\|_{ \cd_{\alpha,\beta}^{\theta,\kappa,\nu,b}  } \|\dot{\cW}\|_{C_{\infty,\beta}^{\vartheta,-\gamma,\rho_b}  }. 
\end{equation}
The upper bounds for the term $II_{st}$ defined by \eqref{eq-def-I-II} are obtained with very similar arguments. We let the patient reader check that we obtain the same bound as \eqref{eq-est-I}

We are now ready to conclude about the announced contraction property. Namely putting together relation \eqref{eq-est-I}, the similar bound we have outlined for  $II_{st}$ and decomposition \eqref{eq-decomp-i}, it is readily checked that the term $\Phi^{12}$ given by \eqref{eq-phi-12-int} fulfills 
\begin{equation}\label{eq-phi-est-j}
\|\Phi^{12}_t-\Phi^{12}_s\|_{\mathfrak{B}_{\alpha,\beta}^{\kappa,\, w_t}  }\le  C\tau^{\epsilon+\delta}(t-s)^{\theta} \|v^{12}\|_{ \cd_{\alpha,\beta}^{\theta,\kappa,\nu,b}  } \|\dot{\cW}\|_{C_{\infty,\beta}^{\vartheta,-\gamma,\rho_b}  } \, ,
\end{equation}
for all $0\le s<t\le \tau$. We now pick $\tau$ small enough, namely 
\begin{equation}\label{eq-tau-k}
\tau=\left(\frac{C}2 \|\dot{\cW}\|_{C_{\infty,\beta}^{\vartheta,-\gamma,\rho_b}  } \right)^{-\frac1{\epsilon+\delta}}.
\end{equation}
Then one can recast \eqref{eq-phi-est-j} as
\[
\|\Phi^{12}_t-\Phi^{12}_s\|_{\mathfrak{B}_{\alpha,\beta}^{\kappa,\, w_t}  }\le \frac12(t-s)^{\theta} \|v^{12}\|_{ \cd_{\alpha,\beta}^{\theta,\kappa,\nu,b}  } 
\]
which is our desired contraction property.
\end{proof}
We now turn to the main existence and uniqueness result of this section. 
\begin{theorem}\label{thm-exi-uni}
Let $\tau$ be a fixed time horizon. Consider a noisy input $\dot{\bW}$ in a distribution space $C_{\infty,\beta}^{\vartheta,-\gamma,\rho_b} $ satisfying Hypothesis \ref{hypo}. Choose some exponents $\theta,\kappa$ such that $\theta+\vartheta>1$ and $\gamma<\kappa<1$. Then there exists a unique solution to equation \eqref{eq-mild-soln}, interpreted as in Definition \ref{def-mild-soln}, sitting in the space $\cd_{\alpha,\beta}^{\theta,\kappa,\nu,b}$ from Definition \ref{def-D-space}.
\end{theorem}
\begin{proof}
The contraction properties from Lemma \ref{lemma-exi-uni} yields existence and uniqueness on $[0,\tau]\times \Heis$, where $\tau$ is given by \eqref{eq-tau-k}. Moreover, as one can immediately see from the right hand side of \eqref{eq-tau-k}, $\tau$ does not depend on the initial condition $u_0$. Therefore a classical patching argument allows to conclude. Further details are left to the reader for sake of conciseness.
\end{proof}

\subsection{Application to a family of Gaussian noises}\label{sec-app-Gau}
In \cite{BOTW} we have defined a family of noises on $[0,\tau]\times\Heis$ which were white in time, with a spatial covariance given by fractional powers of the Laplacian. In this paper, due to the fact that we are considering Stratonovich type equations, it is very natural to consider noises with a fractional Brownian motion type behavior in time. We briefly define this class of noisy inputs, referring to \cite{BOTW} for further details about the spatial covariance structure. 
\begin{definition}\label{def-cov-noise}
Let $\Gamma$ be a covariance function in time, satisfying 
\begin{align}\label{eq-Gamma}
0\le \Gamma(t)\le c_\Gamma|t|^{-\zeta},
\end{align}
for a constant $c_\Gamma>0$ and $0<\zeta<1$. For a regularity coefficient $\alpha\in (0,\frac{n+1}{2})$, we consider a centered Gaussian family $
\{\bW^{\zeta, \alpha}(\varphi);\varphi\in \mathcal{D}\}$ for a certain subset of functions $\cd: [0,\tau)\times \Heis\to \R$. The covariance structure of this Gaussian family is given by
\begin{align}
&\bE\left( \bW^{\zeta, \alpha}(\varphi)\bW^{\zeta, \alpha}(\psi)\right)
=\int_{[0,\tau)^2} \int_{(\Heis)^2} \varphi (s,q_1)(-\Delta)^{-2\alpha} \psi (t,q_1)\Gamma(s-t) \, d\mu(q_1)dsdt \label{eq-cov-1}\\
&\hspace{.9in}
=\int_{[0,\tau)^2} \int_{(\Heis)^2} \Gamma(s-t)G_{2\alpha}(q_1,q_2)\varphi (s,q_1)  \psi (t,q_1)\, d\mu(q_1)d\mu(q_2) dsdt \, , \label{eq-cov2}
\end{align}
where $G_{2\alpha}$ is the kernel defined by \eqref{green function}.
\end{definition}

\begin{remark}
The identity between \eqref{eq-cov-1} and \eqref{eq-cov2} is explained in \cite[Section 3]{BOTW}. In addition, one can define $\cd$ in Definition \ref{def-cov-noise} as the closure of smooth functions under the inner product given by \eqref{eq-cov2}. The time structure of this functional set is described in \cite{PT}, while the spatial component has been identified as the Sobolev space $W^{-\alpha, 2}$ in \cite[Proposition 3.1]{BOTW}.
\end{remark}
\begin{remark}\label{rmk-V}
As defined in \eqref{eq-cov2} the noise $\bW^{\zeta, \alpha}$ is a distribution in both time and space. In order to be more consistent with a noise $\dot{\bW}$ being H\"older continuous in time like in Hypothesis \ref{hypo}, let us introduce a  new piece of notation. That is for $0\le s\le t\le \tau$ and $\varphi:\Heis\to \R$ smooth enough we set
\begin{align}\label{eq-noise-V}
\nV_t(\varphi)=\bW^{\zeta, \alpha}(\mathbf{1}_{[0,t)}\otimes\varphi),\quad \text{and }\quad \delta\nV_{st}(\varphi)=\bW^{\zeta, \alpha}(\mathbf{1}_{[s,t)}\otimes\varphi),
\end{align}
where $\bW^{\zeta, \alpha}$ is still the noise as introduced in Definition \ref{def-cov-noise}.
\end{remark}
Our aim in this section is to prove that the noises $\nV$ are proper inputs for \eqref{eq-mild-soln}. According to Theorem \ref{thm-exi-uni}, this amounts to prove that $\nV$ satisfies Hypothesis \ref{hypo}. We will thus have to place $\nV$ in a space $C^{\vartheta,\kappa,w}$ as defined in \eqref{eq-holder-space}. 
This involves weighted Besov spaces $\mathfrak{B}^{\kappa,w}$, for which block decompositions are useful. We start by computing this decomposition in the lemma below.
\begin{lemma}
For $\zeta\in(0,1)$ and $\alpha\in (0,\frac{n+1}{2})$, let $\nV$ be the noise given in Definition~\ref{def-cov-noise} and Remark \ref{rmk-V}. For a distribution $f$ defined on $\Heis$, recall that the block $\sigma_kf$ is defined by \eqref{eqn-c}. Then for all $t\in(0,\tau)$ and $q\in \Heis$, the random variable $\sigma_k \nV_t(q)$ is given by the following Wiener integral:
\begin{align}\label{eq-sigma-W}
\sigma_k \nV_t(q) 
= \bW^{\zeta,\alpha}(\mathbf{1}_{[0,t)}\otimes H_{k,q}),
\end{align}
where the function $H_{k,q}$ is defined through the kernels $K_{m,\ell,\lambda}$ in \eqref{eq-FT-K} in the following way:
\begin{equation}\label{eq-H-kq}
H_{k,q}(q')=\frac{2^{n-1}}{\pi^{n+1}} \int_{\fHeis}e^{i\lambda (z-z')} K_{\hat{w}} (q) \tilde{\phi}_k(m,m,\lambda)   \overline{K_{\hat{w}} (q')} \, d \hat{w},
\end{equation}
where we have used the convention \eqref{eq-FT-int-inv} for the integration over $\hat{w}$.
\end{lemma}
\begin{proof}
The blocks are characterized through their Fourier transform according to \eqref{eqn-c}:
\begin{align*}
\widehat{\sigma_k \nV}&=(\tilde{\phi}_k\cdot \widehat{\nV})(m,\ell,\lambda)=\tilde{\phi}_k(m,m,\lambda) \widehat{\nV}(m,\ell,\lambda),
\end{align*}
where the second identity is obtained similarly to \eqref{d3}.  We now apply formula \eqref{eq-FT-int} for the Fourier transforms of $\dot{\bW}^{\zeta,\alpha}$, which yields 
\begin{align*}
\lc \widehat{\sigma_k \nV_t}\rc\!(m,\ell,\la)
=
\tilde{\phi}_k(m,m,\lambda)
\lp \int_0^t\int_\Heis  \overline{e^{i\lambda z} K_{m,\ell,\lambda} (q)} \bW^{\zeta,\alpha}(ds,dq) \rp.
\end{align*}
Therefore invoking the inverse Fourier transform \eqref{eq-FT-int-inv}, our definition \eqref{eq-H-kq} of the kernel $H_{k,q}$ and a standard use of stochastic Fubini identities, we obtain 
\begin{align}\label{eq-sigma-W-mid}
\sigma_k \nV_t(q)
&=\frac{2^{n-1}}{\pi^{n+1}} \int_{\fHeis}e^{i\lambda z} K_{\hat{w}} (q) \left(\tilde{\phi}_k(m,m,\lambda)\int_0^t\int_\Heis  \overline{e^{i\lambda z'} K_{m,\ell,\lambda} (q')} \bW^{\zeta,\alpha}(ds,dq') \right)d \hat{w}\notag \\
&=\int_0^t\int_{\Heis}H_{k,q}(q') \bW^{\zeta,\alpha}(ds,dq')
= \bW^{\zeta,\alpha}(\mathbf{1}_{[0,t)}\otimes H_{k,q}).
\end{align}
Formula \eqref{eq-sigma-W-mid} proves relation \eqref{eq-sigma-W} for fixed $t\in[0,\tau)$ and $q\in\Heis$. It remains to see that this formula is valid almost surely for all $t\in[0,\tau)$ and $q\in\Heis$. Since $k$ is fixed, this easily stems from continuity arguments based on Kolmogorov's criterion. We skip the details for sake of conciseness.
\end{proof}

%

We now delve deeper into the structure of the functions $H_{k,q}$. Our findings are summarized in the following lemma.
\begin{lemma}
For $q,q'\in\Heis$ and $k\ge -1$, let $H_{k,q}(q')$ be  the quantity  given by \eqref{eq-H-kq}. Recall that the inverse Fourier transform $\mathcal{F}^{-1}$ is given by \eqref{eq-FT-int-inv}. Then we have
\begin{equation}\label{eq-H-Fouri}
H_{k,q}(q')=\mathcal{F}^{-1}(\tilde{\phi}_k)(qq'^{-1})
\end{equation}
where $\tilde{\phi}_k$ is given as in \eqref{eqn-b}. 
\end{lemma}
\begin{proof}
Let us explicit all the terms in our definition \eqref{eq-H-kq} of the kernel $H_{k,q}$. That is we use the fact, apparent from \eqref{eq-FT-K}, that for $\hat{w}=(m,\ell,\lambda)$,
\[
\overline{K}_{\hat{w}} (q')=\overline{K}_{m,\ell,\lambda} (q')=K_{m,\ell,-\lambda} (q').
\] 
Then we plug \eqref{eq-FT-K} into \eqref{eq-H-kq} and invoke the convention \eqref{eq-FT-int-inv} on integration over $d\hat{w}$. We end up with the identity
\begin{align}\label{eq-expan-H}
H_{k,q}(q')=\frac{2^{n-1}}{\pi^{n+1}} \sum_{m,\ell\in \N^n} &\int_{\R} d\lambda\,|\lambda|^n e^{i\lambda (z-z')} \tilde{\phi}_k(m,m,\lambda)  \int_{\R^{2n}} d\xi d\xi'  e^{2i\lambda\left(\langle y,\xi\rangle- \langle y',\xi'\rangle\right)} \notag\\
&\quad\quad \times \Phi^\lambda_m(x+\xi)\Phi^\lambda_m(x'+\xi')\Phi^\lambda_\ell(-x+\xi)\Phi^\lambda_\ell(-x'+\xi').
\end{align}
We can simplify expression \eqref{eq-expan-H} by recalling (see Section \ref{sec:fourier-general}) that for every $\lambda$ the family $\{\Phi_m^\lambda;m\in \N^n\}$ forms an orthonormal basis of $L^2(\R^n)$. For every $f\in L^2(\R^n)$ we thus have 
\[
f(x)=\sum_{\ell\in \N^n}\left(\int_{\R^n}\Phi_{\ell}^\lambda(y)f(y)dy\right)\Phi_{\ell}^\lambda(x),
\]
from which it is easily seen that the following relation holds true in a distributional sense over $\R^n$:
\begin{align}\label{eq-phi-ell-m}
\sum_{\ell\in\N^n}\Phi^\lambda_\ell(-x+\xi)\Phi^\lambda_\ell(-x'+\xi')=\delta(x-x'-(\xi-\xi')).
\end{align}
Plugging \eqref{eq-phi-ell-m} into \eqref{eq-expan-H} we thus obtain
\begin{align}\label{eq-H-kq1}
H_{k,q}(q')=\frac{2^{n-1}}{\pi^{n+1}}\sum_{m\in \N^n} & \int_{\R} d\lambda\, |\lambda|^n e^{i\lambda (z-z')} \tilde{\phi}_k(m,m,\lambda)  \notag 
\\
&\quad \int_{\R^{n}} d\xi \,e^{2i\lambda\left(\langle y-y',\xi\rangle+ \langle y',x-x'\rangle\right)} \Phi^\lambda_m(x+\xi)\Phi^\lambda_m(\xi-x+2x') .
\end{align}
In the above integral with respect to $\xi$, we perform the change of variable $\xi'=\xi+x$. With relation~\eqref{eq-FT-K} in mind, we realize that
\begin{align}\label{eq-H-mid}
\int_{\R^{n}} &e^{2i\lambda\left(\langle y-y',\xi\rangle+ \langle y',x-x'\rangle\right)}\sum_{m\in \N^n}\Phi^\lambda_m(x+\xi)\Phi^\lambda_m(\xi-x+2x')d\xi \notag \\
&\quad\quad=e^{-2i\lambda \langle y-y', x'\rangle} K_{m,m,\lambda}(x-x',y-y',z-z'+2\langle y',x\rangle-2\langle y,x'\rangle).
\end{align}
Gathering this information into \eqref{eq-H-kq1} we get
\begin{align}\label{eq-H-d}
H_{k,q}(q')=\frac{2^{n-1}}{\pi^{n+1}}\sum_{m\in \N^n} & \int_{\R} d\lambda\, |\lambda|^n e^{i\lambda (z-z'+2\langle y', x\rangle- 2\langle y,x'\rangle)} \tilde{\phi}_k(m,m,\lambda) \notag\\ 
&\times K_{m,m,\lambda}(x-x',y-y',z-z'+2\langle y',x\rangle-2\langle y,x'\rangle) .
\end{align}
Now we simply recall from \eqref{eq-groupaction}  that $(x-x',y-y',z-z'+2\langle y',x\rangle-2\langle y,x'\rangle)=q q'^{-1}$. Taking into account the fact that $\tilde{\phi}_k$ is supported on the diagonal ($m=\ell$) and invoking relation~\eqref{eq-FT-K} again, one can recast \eqref{eq-H-d} as
\begin{align}\label{eq-H-kq3}
H_{k,q}(q')&=\frac{2^{n-1}}{\pi^{n+1}} \int_{\R}e^{i\lambda (qq^{-1} )_z}\sum_{m\in \N^n}K_{m,m,\lambda}(q q'^{-1})
\tilde{\phi}_k(m,m,\lambda) |\lambda|^nd\lambda\\
&=\cf^{-1}({\tilde{\phi}}_k)(q q'^{-1}),
\end{align}
where $(qq^{-1})_z$ stands for the $z$-coordinate of $qq^{-1}\in \Heis$. This finishes the proof of our claim~\eqref{eq-H-Fouri}.
\end{proof} 
In order to compute expectations in a convenient way, we will place our noise $\dot{\bW}^{\zeta,\alpha}$ in a specific class of Besov spaces. We label a proper definition below.

\begin{definition}\label{def-Besov-V}
For $a\in [1, \infty)$, $\gamma>0$ and a weight $\rho_b$ defined as in \eqref{eq-til-ct}, we denote by $\mathfrak{A}_a^{-\gamma}$ the Besov space $\mathfrak{B}^{-\gamma,\rho_b}_{2a,2a}$. Recall from Definition \ref{def-BS} that  the norm on $\mathfrak{A}_a^{-\gamma}$ is given by
\begin{equation}\label{eq-A-norm}
\|f\|_{\mathfrak{A}_a^{-\gamma}}^{2a}=\| f \|^{2a}_{\mathfrak{B}^{-\gamma,\, \rho_b}_{2a,\,2a}}=\sum_{k=-1}^{\infty} \left( 2^{-2\gamma a k}  \| \sigma_{k} f \|^{2a}_{L^{2a}_{\rho_b}}\right).
\end{equation}
\end{definition}

In the lemma below we prove an embedding result for the Besov spaces.
\begin{lemma}
Let $\mathfrak{A}_a^{-\gamma}$ be a Besov space as defined above. If $f\in\mathfrak{A}_a^{-\gamma}$ for all $a\in [1, \infty)$. Then $f\in\mathfrak{A}_\infty^{-\gamma-\epsilon}$ for any $\epsilon>0$.
\end{lemma}
\begin{proof}
First we show that for any $0<\alpha_1\le\alpha_2<\infty$ and $0< \beta_1\le\beta_2\le\infty$ it holds that
\begin{align}\label{eq-embedding}
\mathfrak{B}^{-\gamma,\rho_b}_{\alpha_1,\beta_1}\subset \mathfrak{B}^{-\gamma-(n+1)(\frac1{\alpha_1}-\frac1{\alpha_2}),\rho_b}_{\alpha_2,\beta_2}.
\end{align}
To this aim, note from \eqref{est-h-rho-b} with $k=0$ that the following holds true for all $\alpha_2\ge\alpha_1$:
\[
\lVert  \sigma_k f\rVert_{L^{\alpha_2}_{\rho_b}}\leq C2^{k+(n+1)(\frac{1}{\alpha_1}-\frac{1}{\alpha_2})}\lVert \sigma_kf\rVert_{L^{\alpha_1}_{\rho_b }},
\]
given that ${\rm{Supp}}\, \mathcal{F}({\sigma_k f}) \subset  \{(m,\ell,\lambda):|\lambda|(2|m|+n)< C\cdot2^{k}\}$. Therefore 
\begin{align}\label{eq-embed-1}
\|f\|_{\mathfrak{B}^{-\gamma-(n+1)(\frac1{\alpha_1}-\frac1{\alpha_2}),\rho_b}_{\alpha_2,\beta_1}}
&=\bigg(\sum_{k=-1}^{\infty} \left( 2^{-\gamma k-(n+1)k(\frac1{\alpha_1}-\frac1{\alpha_2})}  \| \sigma_{k} f \|_{L^{\alpha_2}_{\rho_b}}\right)^{\beta_1}  \bigg)^{\frac1{\beta_1}}\notag \\
&\le C \bigg(\sum_{k=-1}^{\infty} \left( 2^{-\gamma k}  \| \sigma_{k} f \|_{L^{\alpha_1}_{\rho_b}}\right)^{\beta_1}  \bigg)^{\frac1{\beta_1}}=
  C\|f\|_{\mathfrak{B}^{-\gamma,\rho_b}_{\alpha_1,\beta_1}}.
\end{align}
We then obtain \eqref{eq-embedding} by using the fact that $\ell^{\beta_1}(\Z)$ is continuously embedded in $\ell^{\beta_2}(\Z)$ for all $0< \beta_1\le\beta_2\le\infty$.

Now using \eqref{eq-embedding} with $\alpha_1=\beta_1=2a$ and $\alpha_2=\beta_2=\infty$  we have  $f\in\mathfrak{A}_a^{-\gamma}$ for all $a\in [1, \infty)$ implies that  $f\in\mathfrak{B}^{-\gamma-\frac{n+1}{2a},\rho_b}_{\infty,\infty}$ for all $a\in [1, \infty)$. This then implies the desired conclusion.
\end{proof}

Our next lemma establishes how a noise with a given covariance function sits in spaces of the form $\mathfrak{A}_a^{-\gamma}$.
\begin{lemma}\label{lemma-nV}
Let ${\bW}^{\zeta,\alpha}$ be a noise as given in Definition \ref{def-cov-noise}, with $\zeta\in(0,1)$ and $\alpha\in (0,\frac{n+1}2)$. Recall that we have defined a time-integrated noise $\nV$ in \eqref{eq-noise-V} and that the spaces $\mathfrak{A}_a^{-\gamma}$ are introduced in Definition \ref{def-Besov-V}. Then for  $a\ge1$ large enough, $\gamma>\frac{n+1}{2}-\alpha$ and $\vartheta\in (0,1-\frac{\zeta}2)$, there exists a random variable $Z$ admitting moments of all orders such that the following inequality holds true:
\begin{align}\label{eq-time-holder}
\|\delta\nV_{st}\|_{\mathfrak{A}_a^{-\gamma}}\le Z (t-s)^\vartheta.
\end{align}
As a consequence, we have that ${\bW}^{\zeta,\alpha}$ is in the space $C^{\vartheta,-\gamma,\rho_b}_{\infty,\infty}$ for any $\gamma> \frac{n+1}2-\alpha$ and $\vartheta\in (0,1-\frac{\zeta}2)$.
\end{lemma}

\begin{proof}
Consider a generic $a\ge1$ and $\vartheta, \gamma$ satisfying the conditions of our lemma. We will first bound the moments of the increments of $\nV$. Namely a simple application of Fubini's identity to \eqref{eq-A-norm} yields
\begin{equation}\label{eq-E-delta-1}
\mathbf{E}\left( \|\delta \nV_{st} \|^{2a}_{\mathfrak{A}_a^{-\gamma}} \right)=\sum_{k=-1}^{\infty}  2^{-2\gamma a k} \int_{\Heis}\mathbf{E}\left( | \sigma_{k} (\delta \nV_{st})(q)|^{2a}\right)\rho_b(q)^{2a}\, d\mu(q).
\end{equation}
We proceed to estimate the expected value in \eqref{eq-E-delta-1}. To this aim recall from \eqref{eq-sigma-W} that 
\[
\sigma_k\,\delta\nV_{st}=\bW^{\zeta,\alpha}(\mathbf{1}_{[s,t]}\otimes H_{k,q}).
\]
Moreover, thanks to the fact that $\bW^{\zeta,\alpha}$ is a Gaussian family, for every $a>1$ there exists a constant $c_a$ such that 
\begin{equation}\label{eq-E-delta}
\mathbf{E}\left( |(\sigma_k\,\delta\nV_{st})(q)|^{2a}\right)=c_a\bigg[ \mathbf{E}\left(|\bW^{\zeta,\alpha}(\mathbf{1}_{[s,t]}\otimes H_{k,q}) |^2 \right)\bigg]^a.
\end{equation}
We are thus reduced to estimate the variance of each random variable $\bW^{\zeta,\alpha}(\mathbf{1}_{[s,t]}\otimes H_{k,q})$. Next we resort to relation \eqref{eq-cov-1}, the bound \eqref{eq-Gamma} and Plancherel's identity \eqref{eq-planch}, which enables to write
\begin{align}\label{eq-E-bd-1}
\mathbf{E}\left(|\bW^{\zeta,\alpha}(\mathbf{1}_{[s,t]}\otimes H_{k,q}) |^2 \right)&=\int_{[s,t]^2}\left(\int_{\Heis} | (-\Delta)^{-\alpha} H_{k,q} |^2\mu(dq) \right)\Gamma(u-v)dudv\notag\\
&\le \frac{2^{n-1}}{\pi^{n+1}}(t-s)^{2-\zeta}\sum_{m,\ell\in \N^n} \int_{\R} \, \vert \widehat{ (-\Delta)^{-\alpha}H_{k,q}}\vert^2 |\lambda |^n d\lambda. 
\end{align}
In addition invoke expression \eqref{eq-H-Fouri} for $H_{k,q}$, relation \eqref{eq-FT-delN} for the Laplace transform of $(-\Delta)^{-\alpha} H_{k,q}$ and the fact that 
$|\left[\cf f(q*\cdot) \right](\hat{w})|=|\cf f(\hat{w})|$ for a generic function $f$. We end up with
\[
| \widehat{ (-\Delta)^{-\alpha}H_{k,q}}(m,\ell,\lambda)|^2=4^{-2\alpha}|\lambda|^{-2\alpha}(2|m|+n)^{-2\alpha}|\tilde{\phi}_k(m,\ell,\lambda)|^2.
\]
Plugging this identity into  \eqref{eq-E-bd-1} and recalling from \eqref{eq-chi-supp} that $\tilde{\phi}_k$ is supported on the diagonal $(m=\ell)$, we thus get the existence of a constant $c_{n,\alpha}$ such that 
\begin{equation}\label{eq-E-bd-2}
\mathbf{E}\left(|\bW^{\zeta,\alpha}(\mathbf{1}_{[s,t]}\otimes H_{k,q}) |^2 \right) 
\le  
c_{n,\alpha}(t-s)^{2-\zeta}
\sum_{m\in \N^n} J_{m,\ell, k} ,
\end{equation}
where we have set
\begin{equation}\label{f1}
J_{m,\ell, k}:= \int_{\R} \,|\lambda|^{-2\alpha}(2|m|+n)^{-2\alpha}|\tilde{\phi}_k(m,m,\lambda)|^2 |\lambda |^n d\lambda .
\end{equation}

Let us now take care of the integral $J_{m,\ell, k}$ defined by \eqref{f1}.
Thanks to the definition \eqref{eqn-b} of $\tilde{\phi}_k$, it is readily checked that 
\[
J_{m,\ell, k}=(2|m|+n)^{-2\alpha} \int_\R \left(\chi({2^{-k}|\lambda|R(m,m)})\right)^2 |\lambda|^{n-2\alpha}d\lambda.
\]
Moreover the function $\chi$ given in Proposition \ref{prop-a} is supported in $\{x\in\R^n:3/4\leq|x|<8/3\}$. Therefore $J_{m,\ell, k}$ is such that 
\[
J_{m,\ell, k}\le (2|m|+n)^{-2\alpha} \int_\R \mathbf{1}_{\lp |\lambda|<\frac{2^{k+2}}{(2|m|+n)}\rp} |\lambda|^{n-2\alpha}d\lambda.
\]
We now integrate over $\lambda$ in an explicit way, under the assumption $n-2\alpha>-1$ (that is $\alpha<\frac{n+1}2$). This yields
\begin{align}\label{eq-J-m-l-k}
J_{m,\ell, k}\le \frac{c_{n,\alpha}}{(2|m|+n)^{2\alpha}} \left( \frac{2^{k+2}}{ 2|m|+n }\right)^{n-2\alpha+1}=c_{n,\alpha} \frac{2^{(k+2)(n+1-2\alpha)}}{(2|m|+n )^{n+1}} .
\end{align}
Gathering \eqref{eq-J-m-l-k} and \eqref{eq-E-bd-2}, we have thus obtained the upper bound
\begin{align}\label{eq-bd-W-zeta}
\mathbf{E}\left(|\bW^{\zeta,\alpha}(\mathbf{1}_{[s,t]}\otimes H_{k,q}) |^2 \right)\le  c_{n,\alpha} \,
2^{k(n+1-2\alpha)} (t-s)^{2-\zeta} \sum_{m\in \N ^n} \frac{1}{(2|m|+n )^{n+1}},
\end{align}
under the assumption $\alpha<\frac{n+1}2$. Eventually observe that the sum over $m$ in \eqref{eq-bd-W-zeta} is finite, by elementary arguments. We have thus achieved the desired bound on the variance of the increments:
\begin{align}\label{eq-bd-increm-var}
\mathbf{E}\left( |(\sigma_k\,\delta\nV_{st})(q)|^{2} \right)\le c_{n,\alpha}  2^{(n+1-2\alpha) k} (t-s)^{(2-\zeta)}, \quad \text{for all }q\in \Heis.
\end{align}
We now conclude the lemma by using relation \eqref{eq-bd-increm-var} in order to bound sums in $\mathfrak{A}_a^{-\gamma}$. That is combine \eqref{eq-bd-increm-var} and \eqref{eq-E-delta} to obtain the following uniform estimate over $q\in\Heis$:
\begin{align}\label{eq-fn-o}
\mathbf{E}\left( |(\sigma_k\,\delta\nV_{st})(q)|^{2a}\right)\le c_{n,\alpha,a} (t-s)^{a(2-\zeta)} 2^{ak(n+1-2\alpha)}.
\end{align}
Now plug this inequality into \eqref{eq-E-delta-1}. If $a$ is chosen large enough (namely such that $ab>n+1$), the function $(\rho_b)^a$ becomes integrable on $\Heis$. Therefore combining \eqref{eq-fn-o} and \eqref{eq-E-delta-1} we get 
\begin{align*}
\mathbf{E}\left( \|\delta \,\nV_{st} \|^{2a}_{\mathfrak{A}_a^{-\gamma}} \right)\le c_{n,a,\alpha}(t-s)^{(2-\zeta)a}\sum_{k=-1}^{\infty}  2^{(n+1-2\alpha-2\gamma) a k} .
\end{align*}
Hence if the coefficient $\gamma$ satisfies our standing assumption $\gamma>\frac{n+1}{2}-\alpha$, the summation above is finite and for every $a>1$ we have 
\begin{align}\label{eq-bd-A}
\mathbf{E}\left( \|\delta \,\nV_{st} \|^{2a}_{\mathfrak{A}_a^{-\gamma}} \right)\le c_{n,a,\alpha,\gamma}(t-s)^{(2-\zeta)a}.
\end{align}
Our proof of \eqref{eq-time-holder} is now finished by standard arguments. Namely one goes from the moment inequality \eqref{eq-bd-A} to the almost sure relation~\eqref{eq-time-holder} resorting to Garsia's lemma
and Fernique's inequality (recall that $\nV$ is a Gaussian process). Details are omitted for the sake of conciseness.
Lastly, combining  \eqref{eq-embed-1} and \eqref{eq-time-holder} we obtain that 
\begin{align}\label{eq-fn-n}
\|\delta\nV_{st}\|_{\mathfrak{B}^{-\gamma-\frac{(n+1)}{2a},\rho_b}_{\infty, \infty}} \le \|\delta\nV_{st}\|_{\mathfrak{A}_a^{-\gamma}}\le Z (t-s)^\vartheta
\end{align}
for all $\gamma+\frac{n+1}{2a}>\frac{n+1}2-\alpha$. Since the parameter $a$ in \eqref{eq-fn-n} can be chosen as arbitrarily large, we get that $\bW^{\zeta,\alpha}$ sits in the space $C_{\infty,\infty}^{\vartheta,-\gamma-\epsilon, \rho_b} $ for any $\gamma>\frac{n+1}2-\alpha$, for any $\epsilon>0$ and $\vartheta\in (0,1-\zeta/2)$. This concludes the proof.
\end{proof}

We are now ready to state the main result of this section, ensuring existence and uniqueness for our stochastic heat equation \eqref{eq-mild-soln} under proper conditions on the noise $\nV$.

\begin{theorem}\label{prop-w-dot}
Let  $\nV$ be the noise given in Definition~\ref{def-cov-noise} and Remark \ref{rmk-V}. We assume that the parameters $\alpha, \zeta$ satisfy the conditions 
\begin{equation}\label{eq-cond-nV}
\zeta\in(0,1),
\quad\text{and}\quad 
\frac{n+1}2-(1-\zeta)<\alpha<\frac{n+1}2.
\end{equation}
Then $\nV$ verifies the conditions of Hypothesis \ref{hypo}. Namely there exists $0<\gamma<1-\zeta$, $\frac{1+\gamma}2<\vartheta<1-\frac{\zeta}2$ and $\rho_b(q)=c(1+|q|_*^b)$ for some $c,b>0$ such that $\nV\in C_{\infty,\infty}^{\vartheta,-\gamma,\rho_b}$.
\end{theorem}
\begin{proof}
In Lemma \ref{lemma-nV} we have considered $\zeta\in (0,1)$, $\alpha\in (0,\frac{n+1}2)$. Then we have proved that $\nV$ almost surely belongs to the space $C^{\vartheta}(\mathfrak{A}_a^{-\gamma})$ for two parameters $\vartheta$, $\gamma$ such that 
\begin{equation}\label{eq-cond-nV-1}
\vartheta<1-\frac{\zeta}2,\quad \mbox{and}\quad  \gamma>\frac{n+1}2-\alpha.
\end{equation}
Now hypothesis \ref{hypo} specifies that we should have $\vartheta>\frac{1+\gamma}2$, which can be recast as 
\begin{equation}\label{eq-cond-nV-2}
2\vartheta>1+ \gamma.
\end{equation}
We can make \eqref{eq-cond-nV-1} and \eqref{eq-cond-nV-2} compatible iff $\alpha$ and $\zeta$ satisfy 
\[
2-\zeta<1+\frac{n+1}2-\alpha,
\]
which is easily seen to be equivalent to 
\[
\alpha>\frac{n+1}2-(1-\zeta).
\]
This finishes our proof.
\end{proof}

\section*{Acknowledgment}
We would like to thank Isabelle Gallagher for some valuable conversation about paraproducts, which have been very useful in order to complete this project.

\end{document}